\documentclass[12pt,twoside,a4paper]{book}
\usepackage[utf8]{inputenc}
\usepackage[english]{babel}

\usepackage{amsmath,amsfonts,amsthm,amssymb,bbm}
\usepackage{color}
\usepackage{dsfont}
\usepackage{enumerate, enumitem}
\usepackage{float} 
\usepackage[bottom]{footmisc}
\usepackage[margin=.8in]{geometry}
\usepackage{graphicx}
\usepackage{hyperref}
	\hypersetup{linkcolor=blue, colorlinks=true, citecolor = blue}
	\usepackage[capitalise]{cleveref}
\usepackage{setspace}
\usepackage{tkz-euclide}
	\usetikzlibrary{arrows.meta}
\usepackage{url}

\usepackage[LGR, T1]{fontenc} 
\newcommand{\koppa}{%
  \mathord{\text{\usefont{LGR}{STEP-TLF}{m}{n}\symbol{"15}}}%
}

\usepackage{fancyhdr}
\numberwithin{equation}{chapter}

\allowdisplaybreaks
\newtheorem{Theorem}{Theorem}[chapter]
\newtheorem{Lemma}[Theorem]{Lemma}
\newtheorem{Proposition}[Theorem]{Proposition}
\newtheorem{Corollary}[Theorem]{Corollary}
\newtheorem{Definition}[Theorem]{Definition}
\newtheorem{Remark}[Theorem]{Remark}

\let\div\relax
\DeclareMathOperator{\div}{\mathrm{div}}
\DeclareMathOperator{\dist}{dist}
\DeclareMathOperator{\diam}{diam}
\DeclareMathOperator{\supp}{supp}

\newcommand{\Id}{\mathbb{I}}
\newcommand{\rd}{{\rm d}}
\newcommand{\dd}{\ \mathrm{d}}
\newcommand{\e}{\varepsilon}
\newcommand{\eps}{\varepsilon}
\newcommand{\del}{\partial}
\newcommand{\R}{\mathbb{R}}
\newcommand{\mf}{\mathcal{F}}
\newcommand{\ms}{\mathcal{S}}
\newcommand{\vc}{\mathbf}
\newcommand{\vu}{\vc u}
\newcommand{\bS}{\mathbb S}
\newcommand{\bD}{\mathbb D}

\title{To collide, or not to collide, that is the question -\\ a survey} 

\author{Florian Oschmann}
 
\date{}

\begin{document}
\pagenumbering{gobble}
\maketitle



\chapter*{Introduction}
\pagenumbering{roman}
\addcontentsline{toc}{chapter}{Introduction}

A simple question is not always easy to answer. In this work, the leading question we want to investigate is the following:
\begin{center}
\it Imagine you put a small ball in a glass of water, wait a little, does the ball eventually touch the ground?
\end{center}
Even a child would probably answer this with ``yes, of course'', but to solve this problem is much harder for a mathematician. Our aim in here is to give a detailed overview over known results for such collision problems, and spin them further from incompressible fluids such as water over compressible ones such as air to the case of non-Newtonian fluids, an example of which is given by blood, paint, or ketchup. A complete solution of whether or not the ball touches the bottom of the glass is, unfortunately, out of reach at least for compressible fluids. We will rather show that \emph{if} both the compressible fluid and the solid inside possess some specific properties, \emph{then} collision happens. Additionally, for the very specific case of two-dimensional Stokes flow, we will also show some converse of this.\\

Before going into deeper details, let us specify the equations under consideration. Two of the main models in fluid mechanics are the incompressible and compressible Navier-Stokes\footnote{after Claude-Louis Navier (1785--1836) and George Gabriel Stokes (1819--1903)} equations governing the motion of a fluid in some domain $\Omega \subset \R^d$, $d \in \{2, 3\}$, the ``skeleton'' of those is given by the system of equations
\begin{align*}
\begin{cases}
\del_t \rho + \div(\rho \vu) = 0 & \text{(mass conservation)},\\
\del_t(\rho \vu) + \div(\rho \vu \otimes \vu) - \div \bS + \nabla p = \rho \vc f & \text{(momentum conservation)},
\end{cases}
\end{align*}
supplemented with suitable initial and boundary conditions. Here, $\rho$ and $\vu$ denote the fluid's density and velocity, respectively, $p$ is the fluid's pressure, $\bS$ the (viscous) stress tensor, and $\vc f$ a given external force density. For the compressible case $\rho \neq \text{constant}$, the pressure commonly depends on the density, whereas for the incompressible case $\rho \equiv \text{constant}$, it can be seen as a Lagrange multiplier to the divergence-free condition $\div \vu=0$. Inserting in the fluid some solid obstacle $\ms$, the fluid's and solid's velocity are clearly not independent, and the above system of equations needs to be extended accordingly:
\begin{itemize}
\item The solid has its own translational and rotational velocity. Therefore, the fluid shall follow this motion on the common boundary $\del \ms$.
\item By Newton's laws of motion, the movement of the solid is governed by the force the fluid imposes to it, thus leading to additional equations for the conservation of linear and angular momentum of the solid.
\end{itemize}
The above issues are covered by the additional equations for the fluid-solid-interaction, namely
\begin{align*}
\begin{cases}
\vu |_{\del \ms} = \dot{\vc G}(t) + \omega(t) \times (x - \vc G(t)) & \text{(fluid's = solid's velocity)},\\
m \ddot{\vc G}(t) = -\int_{\del \ms} (\bS - p\Id) \vc n \dd \sigma + \int_\ms \rho_\ms \vc f \dd x & \text{(linear momentum)},\\
\frac{\rd}{\rd t}(\mathbb J \omega) = -\int_{\del \ms} (x-\vc G) \times (\bS - p\Id)\vc n \dd \sigma + \int_\ms (x-\vc G) \times \rho_\ms \vc f \dd x & \text{(angular momentum)}.
\end{cases}
\end{align*}
Here, the notation ``$\times$'' shall be read as $\omega(t) (x - \vc G(t))^\perp$ if $d=2$, where $(x_1, x_2)^\perp = (-x_2, x_1)$, and similarly for the equation of angular momentum. In the above, we denoted by $\rho_\ms>0$ the solid's density, $\vc G$ is the center of mass of the body $\ms$, $\omega$ is its rotational velocity, $m>0$ is the object's mass given by
\begin{align*}
m=\int_\ms \rho_\ms \dd x,
\end{align*}
and $\mathbb{J}$ is the inertial tensor (moment of inertia) given by
\begin{align*}
\mathbb{J}=\int_\ms \rho_\ms \big( |x-\vc G|^2 \Id - (x-\vc G)\otimes (x-\vc G) \big) \dd x.
\end{align*}
In order to understand what is behind these equations and where they come from, let us make three remarks on them:
\begin{enumerate}
\item Strictly speaking, the solid $\ms$ shall be seen as a time-dependent set-valued map 
\begin{align*}
\ms: (0,T)\times \ms_0 \to 2^{\R^d}
\end{align*}
for some fixed reference particle $\ms_0 \subset \R^d$. The assumption of a \emph{rigid} solid particle then implies
\begin{align}\label{sol1}
\begin{split}
\ms &= \ms(t) = \vc G(t) + \mathbb{O}(t) (\ms_0 - \vc G(0)) \\
&= \{x \in \R^d: x = \vc G(t) + \mathbb{O}(t)(y-\vc G(0)), \ y \in \ms_0 \}
\end{split} \tag{S}
\end{align}
for some rotation (orthogonal matrix) $\mathbb{O}(t) \in \R^{d \times d}$ with $\mathbb{O}(0)=\mathbb{O}(t)\mathbb{O}(t)^T=\Id$. In turn, differentiating \eqref{sol1} with respect to time, the velocity of the particle $\ms$ is given by
\begin{align*}
\vu_s(t,x) = \dot{x}(t,x) = \dot{\vc G}(t) + \mathbb{Q}(t)(x-\vc G(t)) \ \text{for all} \ x \in \ms,
\end{align*}
where the matrix $\mathbb{Q}$ is skew-symmetric and related to $\mathbb{O}$ via
\begin{align*}
\dot{\mathbb{O}}\mathbb{O}^T = \mathbb{Q}.
\end{align*}
The skew-symmetry follows from
\begin{align*}
0=\frac{\dd}{\dd t} \Id = \frac{\dd}{\dd t}(\mathbb{O}\mathbb{O}^T) = \dot{\mathbb{O}}\mathbb{O}^T + \mathbb{O}\dot{\mathbb{O}}^T = \dot{\mathbb{O}}\mathbb{O}^T + (\dot{\mathbb{O}}\mathbb{O}^T)^T = \mathbb{Q}+\mathbb{Q}^T.
\end{align*}
Note especially that this implies the existence of some function $\omega(t) \in \R^{2d-3}$ such that for any $x \in \ms$
\begin{align*}
\vu_\ms(t,x) &= \dot{\vc G}(t) + \omega(t) (x-\vc G(t))^\perp \ \text{if} \ d=2,\\
\vu_\ms(t,x) &= \dot{\vc G}(t) + \omega(t) \times (x-\vc G(t)) \ \text{if} \ d=3,
\end{align*}
which represents precisely the compatibility condition of the fluid's velocity on $\del \ms$.

\item Newton's second law states that acceleration of a body is due to the forces exerted to it. The solid's momentum is simply given by
\begin{align*}
m \dot{\vc G} = \dot{\vc G} \int_\ms \rho_\ms \dd x = \int_\ms \rho_\ms \vu \dd x
\end{align*}
as $\int_\ms x - \vc G \dd x = 0$ by the definition of $\vc G$. Hence, the force on $\ms$ is
\begin{align*}
m \ddot{\vc G} &= \frac{{\rm d}}{{\rm d} t} (m \dot{\vc G}) = \frac{\rm d}{{\rm d} t} \int_\ms \rho_\ms \vu \dd x = \int_\ms \del_t (\rho_\ms \vu) + \div(\rho_\ms \vu \otimes \vu) \dd x \\
&= \int_\ms \rho_\ms \vc f \dd x - \int_\ms - \div (\bS - p \Id) \dd x = \int_\ms \rho_\ms \vc f \dd x - \int_{\del \ms} (\bS - p \Id) \vc n \dd \sigma,
\end{align*}
which is conservation of linear momentum. Similar arguments yield the formula for the angular momentum.

\item The moment of inertia $\mathbb{J}$ is a symmetric positive definite matrix (a positive scalar if $d=2$) in the sense that for each $\omega \in \R^{2d-3}$,
\begin{align*}
\mathbb{J}\omega \cdot \omega = \int_\ms \rho_\ms \big( |x-\vc G|^2 |\omega|^2 - |(x-\vc G)\cdot \omega|^2 \big) \dd x = \int_\ms \rho_\ms |\omega \times (x-\vc G)|^2 \dd x \geq 0.
\end{align*}
In the literature, this is sometimes used as the definition of $\mathbb J$ in the sense that for any $a, b \in \R^{2d-3}$, we have
\begin{align*}
\mathbb{J}a \cdot b = \int_\ms \rho_\ms [ a \times (x-\vc G)]\cdot[b \times (x-\vc G)] \dd x .
\end{align*}
In three spatial dimensions, the eigenvalues of this tensor are called \emph{principal moments of inertia}, the corresponding eigenvectors are the \emph{principal axes}. Moreover, as the solid $\ms(t)$ as well as its center of mass $\vc G(t)$ are both time-dependent, one shall think that the moment of inertia $\mathbb{J}$ is as well. Using the change of variables $z=\mathbb{O}^T(x-\vc G)$, we get
\begin{align*}
\mathbb{J} &= \int_{\ms_0 - \vc G(0)} \rho_\ms \big( |z|^2 \Id - (\mathbb{O}z)\otimes (\mathbb{O}z) \big) \dd z \\
&= \int_{\ms_0 - \vc G(0)} \rho_\ms \big( |z|^2 \Id - z \otimes z \big) \dd z + \int_{\ms_0 - \vc G(0)} \rho_\ms \big( z \otimes z - (\mathbb{O}z)\otimes (\mathbb{O}z) \big) \dd z.
\end{align*}
Indeed the last integral, being the only time-dependent term, tells that $\mathbb{J}$ is almost time-independent in the sense that the principal axes of the body point in different directions due to the underlying rotation. If the body does not rotate, meaning $\mathbb{O}(t)=\Id$ for all $t$, the matrix $\mathbb{J}$ is indeed constant in time. The same holds true if $\ms_0$ is symmetric to the underlying rotation, meaning $\ms_0 - \vc G(0) = \mathbb O (\ms_0 - \vc G(0))$: indeed,
\begin{align*}
&\int_{\ms_0 - \vc G(0)} \big( z \otimes z - (\mathbb O z) \otimes (\mathbb O z) \big) \dd z \\
&= \int_{\ms_0 - \vc G(0)} z \otimes z \dd z - \int_{\mathbb O^{-1} (\ms_0 - \vc G(0))} (\mathbb O z) \otimes (\mathbb O z) \dd (\mathbb O z) = 0.
\end{align*}
In particular, for $\ms$ being a ball, $\mathbb{J}$ does not change for \emph{any} rotation\footnote{In this special case, $\mathbb J$ is a constant scalar multiple of the identity: $\mathbb{J} = \frac25 m r^2 \Id$, where $r>0$ is the radius of the ball $\ms$.}.
\end{enumerate}

Assuming the solid $\ms$ is in free fall under the force of gravity over a horizontal plate, one shall expect that the solid touches the plate in finite time, as it is observed in physics.
We will show, however, that \emph{this is not always the case}. In fact, it strongly depends on the shape of the body $\ms$ near the contact zone, as well as the boundary conditions imposed on $\del \ms$ and $\del \Omega$. Roughly speaking, at least for Newtonian fluids with a linear stress tensor, the main outcomes of the following chapters are:
\begin{itemize}
\item If the body $\ms$ is a ball, and if we impose no-slip boundary conditions on the fluid's velocity, then the solid stays away from the boundary of its container for all finite times.
\item If both the wall and the solid are ``suitably rough'' (in a sense to be specified), then collision occurs in finite time.
\end{itemize}
The term ``suitably rough'' here means different shapes of $\ms$ besides a ball, as well as different boundary conditions on the fluid's velocity. We will be precise on this in the corresponding chapters.

The physical intuitions behind the seemingly paradoxical no-collision for no-slip is the following: a ball has a critical shape in the sense that it is smooth and its lower tip is ``too flat'', and the no-slip boundary conditions on the fluid velocity makes the fluid to stick on $\ms$. The fluid has now to squeeze solely through a long channel between the ball and the plate, while at the same time resting on all boundaries, hence creating a huge drag force on the solid preventing it from touching. On the other hand, if the solid is shaped like a parabola, it ``cuts'' through the fluid and thus can collide (in other words, the channel is not as long as for a ball). For slip boundary conditions, the fluid does not completely has to squeeze through the channel but can partially escape ``to the side'', thus the drag force is lowered and collision can occur.

\tableofcontents

\chapter{Basic function spaces, notations, auxiliary results}
\pagenumbering{arabic}

To begin, we recall briefly the function spaces used in the sequel. Let $d \in \{2,3\}$, $\Omega \subset \R^d$ be a domain, and $1 \leq p,q \leq \infty$. We refer to \cite{EvansPDE2010} for a detailed overview.
\begin{itemize}
\item \textbf{Lebesgue and Sobolev spaces} will be denoted in the usual way as $L^p(\Omega)$ and $W^{k,p}(\Omega)$, respectively. We will also denote them for vector- and matrix-valued functions as in the scalar case, that is, $L^p(\Omega)$ instead of $L^p(\Omega; \R^d)$. The Sobolev space of trace-zero functions will be denoted by $W_0^{1,p}(\Omega)$.

\item \textbf{Lebesgue-Bochner spaces}: For $p, q$ finite, the Lebesgue-Bochner spaces are defined as
\begin{align*}
L^p(0,T;L^q(\Omega)) = \Big\{ f:(0,T) \times \Omega \to \R^d: f(t,\cdot) \in L^q(\Omega), \ \|f\|_{L^p(0,T;L^q(\Omega))} < \infty \Big\}
\end{align*}
with corresponding norm
\begin{align*}
\|f\|_{L^p(0,T;L^q(\Omega))} = \left( \int_0^T \|f(t, \cdot)\|_{L^p(\Omega)}^q \dd t \right)^\frac1q = \left( \int_0^T \left( \int_\Omega |f(t,x)|^p \dd x \right)^\frac{q}{p} \dd t\right)^\frac1q.
\end{align*}
Similarly, we define $L^p(0,T;W^{k,q}(\Omega))$ as
\begin{align*}
L^p(0,T;W^{k,q}(\Omega)) = \Big\{ f:(0,T) \times \Omega \to \R^d: f(t,\cdot) \in W^{k,q}(\Omega), \ \|f\|_{L^p(0,T;W^{k,q}(\Omega))} < \infty \Big\}
\end{align*}
with corresponding norm
\begin{align*}
\|f\|_{L^p(0,T;W^{k,q}(\Omega))} = \sum_{l=0}^k \|\nabla^l f\|_{L^p(0,T;L^q(\Omega))}.
\end{align*}
The definition of the spaces for $p=\infty$ or $q=\infty$ is then as usual with the help of the essential supremum.

\item \textbf{Frobenius inner product:} For each $\mathbb{A},\mathbb{B} \in \R^{d \times d}$, we set $\mathbb{A}: \mathbb{B} = \sum_{i,j=1}^d A_{ij} B_{ij}$. Further, we define the Frobenius norm by $|\mathbb{A}|^2 = \mathbb{A}:\mathbb{A} = \sum_{i,j=1}^d |A_{ij}|^2$.

\item \textbf{Korn's inequality:} There exists a constant $C>0$ such that for each $\vu \in W_0^{1,p}(\Omega)$,
\begin{align}\label{Korn}
\|\nabla \vu\|_{L^p(\Omega)} \leq C \|\bD(\vu)\|_{L^p(\Omega)},
\end{align}
where $\bD(\vu) = \frac12 (\nabla \vu + \nabla^T \vu)$ is the symmetrized gradient.

\item \textbf{Poincar\'e's inequality:} For a bounded domain $\Omega \subset \R^d$, there exists a constant $C>0$ depending just on $\Omega$ and $p$ such that for any $f \in W_0^{1,p}(\Omega)$,
\begin{align}\label{Poinc}
\|f\|_{L^p(\Omega)} \leq C \|\nabla f\|_{L^p(\Omega)}.
\end{align}
Moreover, if $0 \in \Omega$, the constant $C$ scales like $C(r \Omega) = r C(\Omega)$ for any $r>0$.

\item \textbf{Hardy's inequality:} Let $\Omega \subsetneq \R^d$ be a convex open set. Then, there exists a constant $C=C(p)>0$ such that for any $\vc u \in W_0^{1,p}(\Omega)$,
\begin{align}\label{Hardy}
\int_\Omega \bigg| \frac{\vu(x)}{\dist(x, \del \Omega)} \bigg|^p \dd x \leq C \int_\Omega |\nabla \vu|^p \dd x.
\end{align}

\item \textbf{Sobolev embedding:} Let $d, k \geq 1$ and $p \in [1,\infty]$. For $k<d/p$, denote
\begin{align*}
\frac{1}{p^\ast} = \frac1p - \frac{k}{d}.
\end{align*}
Then the embedding $W^{k,p} \hookrightarrow L^q$ is compact for any $1 \leq q < p^\ast$ and continuous for $q=p^\ast$, and there exists a constant $C>0$ depending only on $\Omega$, $d$, $k$, $q$, and $p$ such that for any $f \in W^{k,p}(\Omega)$, we have
\begin{align}\label{SobEmb}
\|f\|_{L^q(\Omega)} \leq C \|f\|_{W^{k,p}(\Omega)}.
\end{align}
Furthermore, if $k=d/p$, then the embedding $W^{k,p} \hookrightarrow L^q$ is continuous for any $1 \leq q < \infty$, and \eqref{SobEmb} holds. If $k>d/p$, then \eqref{SobEmb} holds with $L^q(\Omega)$ replaced by $C^{k-1-[d/p]}(\Omega)$, where $[d/p]$ denotes the entire part of $d/p$.

\item \textbf{Gr\"onwall's inequality:} Let $f:[0,T] \to [0, \infty)$ be integrable and assume there are constants $C_1, C_2 > 0$ with
\begin{align*}
f(t) \leq C_1 \int_0^t f(s) \dd s + C_2.
\end{align*}
Then, it holds
\begin{align}\label{Gronwall}
f(t) \leq C_2 \left( 1 + C_1 t e^{C_1 t} \right)
\end{align}
for almost every $t \in [0,T]$. In particular, there exists a constant $C=C(C_1, C_2, T)>0$ such that $f(t) \leq C$ for a.e.~$t \in [0,T]$.

\item \textbf{Young's inequality:} For any $a,b \geq 0$, and $\delta>0$, and any $1<p,q<\infty$ with $\frac1p + \frac1q = 1$, we have
\begin{align}\label{Young}
ab \leq \delta a^p + \frac{(p \delta)^{1-q}}{q} b^q.
\end{align}
Mostly, we will use the short form of this: for each $\delta>0$, all $a,b \geq 0$, and all $1<p,q<\infty$ with $\frac1p + \frac1q = 1$ there exists a constant $C=C(\delta,p)>0$ such that $ab \leq \delta a^p + C b^q$.

\item \textbf{H\"older's inequality:} Let $N \geq 1$, $i \in \{1,\dots, N \}$, $p_i, p \in [1, \infty]$, and $f_i \in L^{p_i}(\Omega)$ with $\sum_{i=1}^N p_i^{-1} = p^{-1}$. Then $f = \prod_{i=1}^N f_i \in L^p(\Omega)$ with
\begin{align}\label{Holder}
\|f\|_{L^p(\Omega)} \leq \prod_{i=1}^N \|f_i\|_{L^{p_i}(\Omega)}.
\end{align}

\item \textbf{Riesz transform (see \cite{Duoandikoetxea2001}):} Let $j \in \{1,\dots,d\}$. There exists a bounded linear operator $R_j : L^2(\R^d) \to L^2(\R^d)$ and a constant $C=C(d)>0$ such that for any $f \in L^2(\R^d)$, we have
\begin{align*}
R_j f = \del_j (-\Delta)^{-\frac12} f, && \|R_j f\|_{L^2(\R^d)} \leq C \|f\|_{L^2(\R^d)},
\end{align*}
where $(- \Delta)^{-\frac12}$ is the Fourier multiplier with symbol $-1/|\xi|$. In particular, for any $i,j \in \{1,\dots,d\}$ and any $f \in L^2(\R^d)$ such that $\Delta f \in L^2(\R^d)$, we have
\begin{align}\label{Riesz}
R_i R_j \Delta f = -\frac{\del^2 f}{\del x_i \del x_j}, && \bigg\| \frac{\del^2 f}{\del x_i \del x_j} \bigg\|_{L^2(\R^d)} \leq C \|\Delta f\|_{L^2(\R^d)}.
\end{align}
\end{itemize}

To lean the notation, we will write $a \lesssim b$ if there is a generic constant $C>0$ which is independent of $a,b$, and the parameters of interest such that $a \leq C b$. The constant might change its value from line to line. The solid $\ms(t) \subset \R^d$ is assumed to be a simply connected compact set, the motion of which is continuous in time. The domain occupied by the fluid is denoted by $\mf(t)=\Omega \setminus \ms(t)$.



\chapter{General setting: collision for parabolic shapes}
In this chapter, we will focus on the three-dimensional case $d=3$, and just later on go to the two-dimensional setting, since it contains several issues that do not occur here. We will first focus on no-slip boundary conditions and introduce the ``roughness'' by a parameter $\alpha>0$ modelling the shape of the solid. All the outcomes given in this and the next chapter are taken from \cite{JNOR2025} and \cite{NecasovaOschmann2024}.

Let $\Omega\subset \mathbb{R}^3$ be a bounded domain and $\ms(t)$ be a rigid body with center of mass at $\mathbf{G}(t)$ moving inside a viscous fluid, where the fluid domain is $\mf (t)=\Omega\setminus \ms (t)$. The equations of motion are given by model-precise variants of the basic Navier-Stokes system
\begin{equation}\label{B-NSF}
\begin{cases}
\del_t \rho + \div(\rho \vu)=0 & \mbox{ in }\mf(t),\\
\del_t(\rho \vu) + \div(\rho \vu \otimes \vu) - \div \bS + \nabla p = \rho \mathbf f & \mbox{ in }\mf(t),\\
\vu =\dot{\mathbf G}(t)+ \omega(t)\times { (x-\mathbf G(t))} & \mbox{ on } {\del \ms(t)},\\
\vu=0 & \mbox{ on } {\del \Omega},
\end{cases}
\end{equation}
together with some initial conditions and the necessary compatibility conditions between the time derivatives of linear and angular momentum. Here, $\vu$ and $\rho$ are the fluid's velocity and density, and $\dot{\mathbf G}(t)$ and $\omega(t)$ are the translational and rotational velocities of the rigid body, respectively. Moreover, we assume the force $\mathbf{f} \in L^\infty((0,T) \times \R^3)$ to be given.\\

We remark that, depending on the model at hand, one needs to add equations (inequalities) for heat, energy, and entropy, respectively, and accordingly make assumptions on the form of the pressure $p$ and the stress tensor $\bS$. Moreover, we will not state explicitly the dependence of the pressure $p$ on the density and the temperature; we will just require that the pressure behaves ``nicely'' in order to get reasonable bounds on the density and the temperature. As we will focus on collision, which just needs the momentum equation, we will state some precise systems containing all the necessary (in)equalities in Chapter~\ref{ch:Ex}.\\
 
The stress tensor $\bS$ will depend on the symmetrized velocity gradient $\bD(\vu)=\frac12 ( \nabla \vu + \nabla^T \vu)$. The precise assumptions on $\bS$ are stated in \ref{S1}--\ref{S3} below. Further, we assume that the solid is homogeneous with constant mass density $\rho_\ms>0$. The mass and centre of mass of the rigid body are given by
\begin{equation*}
    m=\rho_\ms|\ms(0)|, \quad  \mathbf{G}(t)=\frac{1}{m}\int_{\ms(t)}\rho_\ms x \dd x.
\end{equation*}
We will also assume that the solid's mass is independent of time, that is, $m=\rho_\ms |\ms(t)|$ for any $t\geq 0$, leading to the density-independent expression $\vc G(t) = |\ms(t)|^{-1} \int_{\ms(t)} x \dd x$.

\section{Heuristics ensuring collision}
The collision result, for bodies of class $C^{1,\alpha}$ and in its easiest form, relies essentially on three main points:
\begin{itemize}
\item uniform bounds: The velocity and, if under consideration, density and temperature shall obey bounds that are independent of the solid's distance $h$ to the boundary of its container. Such estimates usually follow easily from the energy (respectively entropy) and Gr\"onwall's inequality \eqref{Gronwall}, and thus can essentially be considered as ``given''.

\item appropriate test function for the momentum equation: The question whether collision occurs or not is answered by testing the momentum equation against a ``well constructed'' test function $\mathbf w_h$ and estimating all occurring terms. This function being the same for both compressible and incompressible systems, and being additionally divergence-free, the only term that changes its form in the momentum equation is the diffusive one containing the stress tensor $\bS$.

\item body orientation: This is indeed one of the main points when considering solids of $C^{1,\alpha}$ regularity, since one needs to ensure that the equation describing the lower tip of the body keeps its form during the free fall; in particular, rotations of the body (except rotations around the $x_3$-axis) must be excluded.
\end{itemize}

The rest of the proof is easily explained: once the test function $\mathbf w_h$ is constructed, one estimates its norm and the norm of its derivatives in different $L^q$-spaces such that these norms are uniformly bounded independent of $h$. This requirement will give conditions on $q$ in terms of $\alpha$. Testing the momentum equation with $\mathbf w_h$, using the regularity of the velocity, density, and temperature, and estimating all occurring terms will then lead to restrictions on $\alpha$. Provided these restrictions hold, the final inequality takes the form
\begin{align*}
T \leq C (1+T),
\end{align*}
where $T>0$ is the maximal existence time before collision. Under some energy and mass assumptions, one then ensures $C<1$ such that $T<\infty$, meaning collision occurs in finite time.\\

Let us remark that collision may also occur if $\mathbf w_h$ is \emph{not} uniformly bounded with respect to $h$. This behavior shows up since, at least in the incompressible case and for Newtonian fluids, the distance $h$ is related to the drag force $\mathcal{D}_h$ via an ODE of the form $\ddot{h} + \dot{h} \mathcal{D}_h=f$. For bodies of $C^{1,\alpha}$ shape, one can show that $\mathcal{D}_h \sim h^{-\beta}$ for some $\beta=\beta(\alpha)<1$, thus collision can occur in finite time. In contrast, for a sphere $\mathcal{D}_h \sim h^{-1}$, hence collision is forbidden. The question of the drag forces for \emph{compressible} fluids is rather different, and it is not clear whether one can derive a similar ODE for $h$, and even how $\mathcal{D}_h$ depends on the fluid's density. The author thinks that also the construction of the test function $\mathbf w_h$ has to be changed according to the non-constant fluid's density. Moreover, it requires a much more detailed analysis of the involved terms. For these reasons, when talking about \emph{compressible} fluids, we will just consider cases where $\mathbf w_h$ is uniformly bounded with respect to $h$. A wider class of obstacles with unbounded $\vc w_h$ in the \emph{incompressible} setting will be given in and after Chapter~\ref{ch:6}.



\section{The stress tensor and uniform bounds}
Like mentioned in the introduction, the crucial part in analyzing collisions is to investigate the form of the stress tensor $\bS$. We will make the following assumptions:
\begin{enumerate}[label=(S\arabic*)]
\item Continuity: $\bS$ is a continuous mapping from $(0,\infty)\times \R_{\rm sym}^{3\times 3}$ to $\R_{\rm sym}^{3\times 3}$, and depends continuously on the temperature $\vartheta \in (0,\infty)$ and the symmetric gradient $\bD(\vu) = \frac12 (\nabla\vu + \nabla^T \vu) \in \R_{\rm sym}^{3 \times 3}$. \label{S1}
\item Monotonicity: For any $\mathbb M, \mathbb N \in \R_{\rm sym}^{3 \times 3}$, we have $[\bS(\cdot,\mathbb{M}) - \bS(\cdot,\mathbb{N})] : (\mathbb{M} - \mathbb{N}) \geq 0$.\label{S2}
\item Growth: There are absolute constants $\delta \geq 0$ and $0<c_0\leq c_1<\infty$ such that for some $p>1$, for all $\vartheta>0$, and all $\mathbb M \in \R_{\rm sym}^{3\times 3}$, we have $c_0 |\mathbb{M}|^p - \delta \leq \bS(\vartheta,\mathbb{M}):\mathbb{M} \leq c_1 |\mathbb{M}|^p$.\label{S3}
\end{enumerate}
We note that two main models fall into this setting: classical power-law fluids like $\bS = |\bD(\vu)|^{p-2} \bD(\vu)$, and so-called activated Euler fluids with $\bS=\max\{|\bD(\vu)|-\delta_0, 0\}|\bD(\vu)|^{-1} \bD(\vu)$ for some $\delta_0>0$. Moreover, we emphasize that in view of \ref{S3}, the temperature plays essentially no role for our discussion. One can think of temperature dependent viscosity coefficients that are uniformly bounded. Another example of viscosities \emph{growing} with the temperature will be given in Chapter~\ref{ch:4}. Note moreover that condition \ref{S3} implies by duality $\bS \in L^{p'}((0,T)\times \Omega)$ since
\begin{align}\label{Sp}
\begin{split}
\|\bS\|_{L^{p'}((0,T)\times \Omega)} &= \sup_{\|\mathbb{M}\|_{L^p((0,T) \times \Omega)} \leq 1} \int_0^T \int_\Omega \bS : \mathbb{M} \dd x \dd t \\
&\leq c_1 \sup_{\|\mathbb{M}\|_{L^p((0,T) \times \Omega)} \leq 1} \int_0^T \int_\Omega |\mathbb{M}|^p \dd x \dd t \leq c_1.
\end{split}
\end{align}

\begin{Remark}
As the proof of Theorem~\ref{theo1} below will show, we are able to catch stress tensors of different growth for small and large values of $|\mathbb{M}|$; in fact, we might also consider
\begin{align}\label{stress2}
c_0 |\mathbb{M}|^p + c_2 |\mathbb{M}|^q - \delta \leq \bS(\vartheta, \mathbb{M}):\mathbb{M} \leq c_1 |\mathbb{M}|^p + c_3 |\mathbb{M}|^q.
\end{align}
The conditions required in Theorem~\ref{theo1} below then have to be modified in an obvious way. Note that the advantage in allowing for different growth is that one may also take into account other fluid models such as the so-called Carreau-Yasuda law, where the stress tensor is given by
\begin{align*}
\bS(\mathbb{M}) = \mu (1 + |\mathbb{M}|^2)^{\frac{p}{2}-1} \mathbb{M} + \lambda \, {\rm trace}(\mathbb{M}) \Id, \quad \mu>0, \, \lambda \geq 0,
\end{align*}
giving rise to a growth with $q=2$ in \eqref{stress2}.
\end{Remark}

To start analyzing the collision behavior, one first needs uniform bounds on the velocity, density, and temperature. With a slight abuse of notation, we extend $\rho$ and $\vu$ by
\begin{align}\label{extended}
\rho = \begin{cases}
\rho & \text{in } \mf(t),\\
\rho_\ms & \text{in } \ms(t),
\end{cases} &&
\vu = \begin{cases}
\vu & \text{in } \mf(t),\\
\dot{\vc G}(t) + \omega(t) \times (x - \vc G(t)) & \text{in } \ms(t).
\end{cases}
\end{align}
For the time being, we will \emph{assume that the following bounds hold}, and give for some specific models the proofs in Chapter~\ref{ch:Ex}:

\begin{align}\label{UnifBds}
\gamma>\frac32, \quad \|\rho\|_{L^\infty(0,T;L^\gamma(\mf(\cdot)))}^\gamma + \|\vu\|_{L^p(0,T;W_0^{1,p}(\Omega))}^p + \|\rho |\vu|^2\|_{L^\infty(0,T;L^1(\Omega))} \lesssim E_0 + 1.
\end{align}
Here, $E_0$ is the initial energy of the system given by
\begin{align}\label{eq:E0}
E_0 = \int_{\mf(0)} \frac{|\mathbf m_0|^2}{2\rho_0} + P(\rho_0, \vartheta_0) \dd x + \frac{m}{2} |\mathbf V_0|^2 + \frac12 \mathbb{J}(0)\omega_0 \cdot \omega_0,
\end{align}
where $\vc m_0 = (\rho \vu)(0)$ is the fluid's initial momentum, $\vc V_0 = \dot{\vc G}(0)$ and $\omega_0 = \omega(0)$ are the initial translational and rotational speed of $\ms$, respectively, and $P(\rho,\vartheta)$ is a pressure potential associated to the original pressure $p(\rho,\vartheta)$. For instance, if the temperature is constant such that $p(\rho, \vartheta) = p(\rho)$, the potential $P(\rho)$ satisfies
\begin{align*}
\rho P'(\rho) - P(\rho) = p(\rho).
\end{align*}
Further, in sense of density, we assume that
\begin{align}\label{p-law}
p(\rho,\cdot) \sim \rho^\gamma \ \text{ for } \ \rho \geq \underline{\rho} >0, \quad \gamma>\frac32.
\end{align}
We will be more precise on this in the models stated in Chapter~\ref{ch:Ex}, together with available existence results of weak solutions to the problem under consideration.\\

In the sequel, we consider a solid of class $C^{1,\alpha}$ moving vertically over a flat horizontal surface under the influence of gravity. More precisely, we make the following assumptions (see Figure~\ref{fig1} for the main notations):
\begin{enumerate}[label=(A\arabic*)]
\item The source term is provided by the gravitational force ${\mathbf f}=-g{\mathbf e}_3$ and $g>0$. \label{a1}
\item The solid moves along and is symmetric to the vertical axis $\{x_1=x_2=0\}$. \label{a2}
\item The only possible collision point is at $x=0 \in \del \Omega$, and the solid's motion is a vertical translation.\label{a3}
\item Near $r=0$, $\partial \Omega$ is flat and horizontal, where $r=\sqrt{x_1^2+x_2^2}$. \label{a4}
\item Near $r=0$, the lower part of $\partial \ms(t)$ is given by \label{a5}
\begin{align*}
    x_3={ h(t)+r^{1+\alpha}},\ r\leq 2r_0\mbox{ for some small enough } r_0>0.
\end{align*}
\item The collision just happens near the flat boundary of $\Omega$: \label{a6}
\[
\inf_{t>0} \dist \left(\ms(t),\partial\Omega \setminus [-2r_0, 2r_0]^2\times\{0\}\right)\geq d_0>0.
\]

\end{enumerate} 

Let us also assume that the position of the solid is characterized by its height $h(t)$, in the sense that ${\mathbf G}(t)={\mathbf G}(0)+(h(t)-h(0)){\mathbf e}_3$, and $\ms(t)=\ms(0)+(h(t)-h(0)){\mathbf e}_3$. Note especially that this means that the solid rotates at most around the $x_3$-axis. In turn, if $\ms(0) - \vc G(0)= \mathbb{O}_3 (\ms(0) - \vc G(0))$, then $\ms(t) - \vc G(t)= \mathbb{O}_3 (\ms(t) - \vc G(t))$ for all $t \geq 0$ and all rotations $\mathbb{O}_3 \in SO(3)$ around the $x_3$-axis, that is\footnote{As for the moment of inertia $\mathbb{J}$, we saw before that this symmetry forces $\mathbb J$ to be constant in time.},
\begin{align*}
\mathbb{O}_3 = \begin{pmatrix}
\cos \phi & - \sin \phi & 0\\
\sin \phi & \cos \phi & 0\\
0 & 0 & 1
\end{pmatrix}, \quad \phi \in [0, 2\pi).
\end{align*}
By rotational invariance of the Navier-Stokes equations, as long as the solution is unique (as, for instance, in 2D), this assumption can be verified rigorously, see \cite{GerardVaretHillairet2010}.

\begin{figure}[H]
\centering
\begin{tikzpicture}[scale=.7]
\draw[->] (-6,0) -- (6,0);
\draw[->] (0,0) -- (0,7);
\node at (6,0) [anchor=north] {$r$};
\node at (0,7) [anchor=east] {$x_3$};
\draw (-5,0) rectangle (5,6);
\node at (-6.1,4) [anchor=west] {$\del \Omega$};
\draw[->] (6,4.5)--(6,3.5);
\node at (6,4) [anchor=west] {$-g\vc e_3$};
\draw[black, thick] plot [smooth cycle] coordinates {(0,1) (1,1.5) (2,3) (3,5) (0,5.5) (-3,5) (-2,3) (-1,1.5)};
\draw[<->] (0,0) -- (0,1);
\node at (-1,3) {$\ms$};
\node at (-3,1) {$\mf$};
\node at (-.15,.5) [anchor=west] {$h$};
\node[fill=white] at (2,2) [anchor=west] {$x_3=h+r^{1+\alpha}$};
\draw[->] (2,2) -- (.75,2) -- (.75,1.4);
\draw[dashed] (-1,0) -- (-1,1.5);
\draw[dashed] (1,0) -- (1,1.5);
\node at (1,0) [anchor=north] {$2r_0$};
\draw[dotted] (-.5,0) -- (-.5,1.1);
\draw[dotted] (.5,0) -- (.5,1.1);
\node at (-.5,0) [anchor=north] {$-r_0$};
\end{tikzpicture}
\caption{The body $\ms$ and fluid $\mf$ in the container $\Omega$.}
\label{fig1}
\end{figure}

Our main result regarding collision now reads as follows:

\begin{Theorem} \label{theo1}

Let $\gamma>\frac{3}{2}$, $p\geq 2$, $0<\alpha\leq 1$, and $\Omega,\ \ms\subset\mathbb R^3$ be bounded domains of class $C^{1,\alpha}$.
Let $(\rho, \vartheta, \vu, {\mathbf G})$ be a weak solution to a model-precise version of \eqref{B-NSF} enjoying the bounds \eqref{UnifBds}, let $\bS$ comply with \ref{S1}--\ref{S3}, and assume that \ref{a1}--\ref{a6} are fulfilled. If the solid's mass is large enough, and its initial vertical and rotational velocities are small enough, then the solid touches $\partial \Omega$ in finite time provided
\begin{align}\label{jedna}
\begin{split}
&\alpha<\min \bigg\{ \frac{3-p}{2p-1}, \frac{3(4p\gamma - 3p - 6\gamma)}{p\gamma + 3p + 6\gamma} \bigg\} \quad \text{with}\\
    &\frac32< \gamma \leq 3, \ \frac{6\gamma}{4\gamma-3} < p < 3, \quad \text{or} \quad \gamma>3, \ 2 \leq p < 3.
    \end{split}
\end{align}
\end{Theorem}

\begin{Remark}
The terms ``large enough'' and ``small enough'' should be interpreted in such a way that inequality \eqref{finalIneq} below is satisfied. Specifically, there is a constant $C_0>0$ which is independent of $m$ and $T$ such that collision occurs provided
\begin{align*}
    C_0 \max\{m^{-1/2}, m^{-3/2} \} \bigg(1 + E_0^{\frac12 + \frac{1}{\gamma} + \frac1p} \bigg) < 1.
\end{align*}
One shall also compare these assumptions with the ones given in Section~\ref{sec:Tresca}.
\end{Remark}

\begin{Remark}\label{rem1}
Let us mention a few facts about the constraints in \eqref{jedna}:
\begin{enumerate}[label = (\roman*)]
\item The two expressions inside the minimum stem, as one shall expect, from estimating the diffusive and convective part, respectively.

\item The restriction $p<3$ is due to the diffusive part, see the estimate of $I_4$ in Section~\ref{sec:PfThm}. Moreover, the requirement $p\geq 2$ stems from the convective term, since we need to estimate the square of the velocity in time. Thus, our result as stated above is just valid for shear-thickening fluids. Omitting convection, Theorem~\ref{theo1} still holds provided
\begin{align}\label{aa1}
\gamma>\frac32, \ \frac{\gamma}{\gamma-1}<p<3, \ \alpha<\min \bigg\{ \frac{3-p}{2p-1}, \frac{9(p\gamma - p - \gamma)}{2p\gamma + 3p + 3\gamma} \bigg\},
\end{align}
hence also allowing for shear-thinning fluids if $\gamma>2$.

\item The first condition on $p$ and $\gamma$ in \eqref{jedna} can be equivalently stated as $\frac{3p}{4p-6} < \gamma \leq 3$, $2<p<3$.

\item The first fraction inside the minimum in \eqref{jedna} wins precisely if $\gamma \geq \frac{3p}{5p-9}$, and in \eqref{aa1} if $\gamma \geq \frac{3p}{4p-6}$. This seems to be optimal in the sense that for $p=2$, $\alpha=\frac13$ is a ``borderline value'' for the incompressible case, which would (loosely speaking) correspond to $\gamma=\infty$ (see \cite[Section~3.1]{GerardVaretHillairet2012} for details).
\end{enumerate}
\end{Remark}

The proof of Theorem~\ref{theo1} will be carried out in the next chapter. Specifically, we will construct a special function associated to the solid. Testing the momentum equation against this function and estimating carefully all occurring terms will finally yield the result.

\begin{Remark}\label{rem2}
Another way how to interpret collision is to use so-called streamlines. These are the solutions to the ODE
\begin{align*}
\frac{\rd}{\rd t} \vc X(t, x) = \vu(t, \vc X(t, x)), \quad t>0, \quad \vc X(0, x)=x.
\end{align*}
If no collision occurs, the solutions to this ODE are well-defined, in particular, streamlines cannot concentrate. For instance, this happens if $\vu$ is Lipschitz continuous in the second variable such that the solutions are even unique. If, however, the solid collides with the boundary of its container, then all the streamlines have the same value at the collision point and thus are ill-defined, and the Lipschitz-norm (or some even weaker $W^{1,q}$-norm) of the velocity $\vu$ blows up as $t$ reaches the collision time. We will see how one can prove no-collision results based on this observation in Chapters~\ref{ch:7} and \ref{ch:8}.
\end{Remark}



\chapter{Proof of Theorem~\ref{theo1}}\label{ch3}
The aim of this chapter is to define an appropriate test function for the momentum equation that will ensure collision in finite time. Let  $(\rho, \vartheta, \vu, {\mathbf G})$ be a weak solution of \eqref{B-NSF} satisfying the assumptions \ref{a1}--\ref{a6} in the time interval $(0,T_*)$ before collision. From now on we denote $\ms_h=\ms_h(t)=\ms(0)+(h(t)-h(0)){\mathbf e}_3$ and $\mf_{h}=\mf_h(t)=\Omega\setminus \ms_h(t)$. As mentioned before, the assumption \ref{a2} on $\ms(t)$ especially means that the body rotates at most around the $x_3$-axis.\\

Collision can occur if and only if $\lim_{t\to T_*}h(t)=0$. Note further that $\dist(\ms_h(t),\partial \Omega)= \min\{h(t),d_0\}$ by assumptions \ref{a2} and \ref{a6}.



\section{Test function}\label{sec:31}
To construct our desired function, we will make use of cylindrical coordinates $(r,\theta,x_3)$ with the standard basis $(\mathbf e_r, \mathbf e_\theta, \mathbf e_3)$.
We use the same function as in \cite{GerardVaretHillairet2012} (see also \cite{GerardVaretHillairet2010, GVHW2015}), which is constructed as a function ${\mathbf w}_h$ associated with the solid particle $\ms_h$ frozen at distance $h$. This function will be defined for $h\in (0,\sup_{t\in [0,T_*)}h(t))$. Note that when $h\to 0$, a cusp arises in $\mf_h$, which is contained in a cylindrical domain beneath $\ms$ given by
\begin{align}\label{omegah}
    \Omega_{h,r_0}=\{ x \in \mf_h: 0\leq r<r_0,\ 0\leq  x_3\leq  h+r^{1+\alpha},\ r^2=x_1^2+x_2^2\}.
\end{align}
For the sequel, we fix $h$ as a positive constant and define $\psi(r):= h+r^{1+\alpha}$. Note that the common boundary $\del \Omega_{h,r_0}\cap \del \ms_h$ is precisely given by the set $\{0\leq r\leq r_0,\ x_3=\psi(r)\}$.\\

Let us derive how an appropriate test function inside $\Omega_{h, r_0}$ might look like. In order to get rid of the pressure term, we seek for a function $\mathbf w_h$ which is divergence-free. Additionally, it shall be rigid on $\ms_h$, and comply with its motion. Thus, our test function shall satisfy
\begin{align*}
\mathbf w_h |_{\ms_h} = \mathbf e_3, \quad \mathbf w_h |_{\del \Omega} = 0, \quad \div \mathbf w_h = 0.
\end{align*}
An easy function satisfying all this is given by $\mathbf w_h = \nabla \times (\phi_h \mathbf e_\theta)$ for some function $\phi_h(r,x_3)$ to be determined. The solenoidality of $\mathbf w_h$ is thus obvious. In cylindrical coordinates, we write $\mathbf w_h$ as
\begin{align}\label{def:wh}
\mathbf w_h=-\del_3\phi_h \mathbf e_r+\frac1r \del_r(r\phi_h)\mathbf e_3.
\end{align}
The boundary conditions on $\mathbf w_h$ translate for $\phi_h$ into
\begin{align*}
\del_3 \phi_h(r, 0) = 0, && \frac1r \del_r (r \phi_h)(r, 0)=0,\\
\del_3 \phi_h(r, \psi(r)) = 0, && \frac1r \del_r (r \phi_h)(r, \psi(r)) = 1.
\end{align*}
Further, considering the energy
\begin{align*}
\mathcal{E}=\int_{\mf_h} |\nabla \mathbf w_h|^2 \dd x
\end{align*}
and anticipating that most of it stems from the vertical motion, that is, from the derivative in $x_3$-direction, we get
\begin{align*}
\mathcal{E}\sim \int_{\mf_h} |\del_3^2 \phi_h|^2 \dd x.
\end{align*}
The Euler-Lagrange equation for the functional $\mathcal{E}$ thus reads $\del_3^4 \phi_h(r, x_3)=0$, meaning $\phi_h(r, x_3) = a(r)x_3^3 + b(r)x_3^2 + c(r)x_3 + d(r)$ for some functions $a,b,c,d$ to find. A simple calculation now leads to the general form
\begin{align*}
\phi_h(r, x_3) = -\frac32 \bigg( \frac{\kappa_1}{r} - r \bigg) \bigg(\frac{x_3}{\psi(r)} \bigg)^2 + \bigg( \frac{\kappa_1}{r} - r \bigg) \bigg(\frac{x_3}{\psi(r)}\bigg)^3 + \frac{\kappa_2}{r}, \quad \kappa_1, \kappa_2 \in \R.
\end{align*}
In order to get a smooth function $\phi_h$ for all values of $r$ and $x_3$, we choose $\kappa_1=\kappa_2=0$ to infer
\begin{align*}
\phi_h(r, x_3) = \frac{r}{2} \Phi \bigg( \frac{x_3}{\psi(r)} \bigg), \quad \Phi(t) = t^2(3-2t).
\end{align*}
Thus, inside $\Omega_{h, r_0}$, the so constructed function will take advantage of the precise form of the solid. Extending $\phi_h$ in a proper way to the whole of $\Omega$, we thus can define a proper test function $\mathbf w_h$.\\

To achieve this, we use a similar extension as in \cite{GerardVaretHillairet2010}: define smooth functions $\chi, \eta$ satisfying
\begin{align}
&\chi=1 \text{ on } (-r_0,r_0)^2 \times (0,r_0), && \chi=0 \text{ on } \Omega\setminus \big( (-2r_0, 2r_0)^2 \times (0, 2r_0) \big)\label{chi}\\
&\eta=1 \text{ on } \mathcal{N}_{d_0/2}, && \eta=0 \text{ on } \Omega \setminus \mathcal{N}_{d_0}, 
\end{align}
where $d_0>0$ is as in \ref{a6}, and $\mathcal{N}_\delta$ is a $\delta$-neighborhood of $\ms(0)$. With a slight abuse of the notations above, set
\begin{align}\label{phih}
\phi_h(r,x_3)= \frac{r}{2} \begin{cases}
    1 & \text{on } \ms_h,\\
    (1-\chi(r,x_3))\eta(r,x_3-h+h(0)) + \chi(r,x_3)\Phi\left(\frac{x_3}{\psi(r)}\right) & \text{on } \Omega \setminus \ms_h,
\end{cases}
\end{align}
and $\mathbf w_h = \nabla \times (\phi_h \mathbf e_\theta)$.
Observe that the function $\mathbf w_h$ satisfies
\begin{align*}
\mathbf w_h|_{\partial \ms_h}=\mathbf e_3,\quad \mathbf w_h|_{\del \Omega}=0,\quad \div\mathbf w_h=0.
\end{align*}
Indeed, the divergence-free condition is obvious from the definition of $\mathbf w_h$. Further, since $\phi_h = r/2$ on $\ms_h$, we have $\mathbf w_h = \mathbf e_3$ there. Moreover, by definition of $\chi$ and $\eta$, we have $\phi_h=0$ on $\del \Omega \setminus \big( (-2 r_0, 2 r_0)^2 \times \{ 0 \} \big)$ as long as $r_0$ and $h$ are so small that $h+r_0^{1+\alpha} \leq d_0 < r_0$. Lastly, $\phi_h=0$ on $\del \Omega \cap \big( (-r_0, r_0)^2 \times \{0\} \big)$ by definition of $\chi$ and $\Phi(0)=0$, and in the annulus $\big( (-2r_0, 2r_0)^2 \setminus (-r_0, r_0)^2 \big) \times \{ 0 \}$ we use also $\eta(r,h(0))=0$ for $r>\mathfrak{r}_0$ for some $\mathfrak{r}_0\in (d_0, r_0)$ to finally conclude $\mathbf w_h|_{\del\Omega}=0$, provided $h$ is sufficiently close to zero.

We summarize further properties in the following Lemma:
\begin{Lemma}\label{BdsWh}
It holds $\mathbf w_h\in C_c^\infty(\Omega)$ and
\begin{align}\label{est1:wh}
\|\del_h\mathbf w_h\|_{L^\infty(\Omega\setminus \Omega_{h,r_0})} + \|\mathbf w_h\|_{W^{1,\infty}(\Omega\setminus \Omega_{h,r_0})}\lesssim 1.
\end{align}
Moreover,
\begin{align*}
\|\mathbf w_h\|_{L^q(\Omega_{h,r_0})} &\lesssim 1 \text{ for any } q<1+\frac3\alpha,\\
\|\del_h \mathbf w_h\|_{L^q(\Omega_{h,r_0})} + \|\nabla \mathbf w_h\|_{L^q(\Omega_{h,r_0})} &\lesssim 1 \text{ for any } q<\frac{3+\alpha}{1+2\alpha}.
\end{align*}
\end{Lemma}
\begin{proof}
We know from the definition of $\mathbf w_h$ in \eqref{def:wh} that $\mathbf w_h\in C_c^\infty(\Omega)$. Moreover, $\mathbf w_h$ is bounded outside the bounded region $\Omega_{h, r_0}$, so the first inequality \eqref{est1:wh} is obvious.

Due to the property \eqref{chi} of $\chi$, the function $\phi_h$ (see \eqref{phih}) in $\Omega_{h,r_0}$ becomes
\begin{equation*}
\phi_h(r,x_3)=\frac{r}{2}\Phi\bigg(\frac{x_3}{\psi(r)}\bigg)\quad \mbox{in}\quad \Omega_{h,r_0}.
\end{equation*}
By definition \eqref{def:wh} of $\mathbf w_h$, we have
\begin{align*}
\mathbf w_h = -\frac{r}{2}\Phi' \bigg(\frac{x_3}{\psi}\bigg)\frac1\psi \mathbf e_r + \Phi\bigg(\frac{x_3}{\psi}\bigg)\mathbf e_3 - \frac{r}{2}\Phi' \bigg(\frac{x_3}{\psi}\bigg)\frac{x_3 \del_r\psi}{\psi^2}\mathbf e_3 \quad \mbox{in}\quad \Omega_{h,r_0}.
\end{align*}
Further, $x_3\leq \psi$ in $\Omega_{h,r_0}$. Hence,
\begin{align}\label{bdsPhi}
    |\Phi| + |\Phi'| + |\Phi''|\lesssim 1,
\end{align}
leading to
\begin{align*}
|\mathbf w_h|\lesssim 1+\frac{r}{\psi}(1+\del_r \psi).
\end{align*}
Due to legibility, we will not write the argument of $\Phi$ in the sequel. Then, we obtain successively
\begin{align*}
|\del_r \mathbf w_h| &\lesssim \Phi' \frac1\psi + \frac{r}{2\psi}\Phi''\frac{x_3\del_r \psi}{\psi^2} + \frac{r}{2}\Phi'\frac{\del_r\psi}{\psi^2}\\
&\quad + \Phi'\frac{x_3\del_r\psi}{\psi^2} + \frac{r}{2}\Phi''\cdot \bigg(\frac{x_3\del_r\psi}{\psi^2}\bigg)^2+\frac{r}{2}\Phi'\frac{x_3\del_r^2\psi}{\psi^2} + r\Phi'\frac{x_3(\del_r\psi)^2}{\psi^3},\\
|\del_3\mathbf w_h| &\lesssim \frac{r}{2}\Phi'' \frac{1}{\psi^2} + \Phi'\frac1\psi + \frac{r}{2}\Phi''\frac{x_3\del_r\psi}{\psi^3}+\frac{r}{2}\Phi'\frac{\del_r\psi}{\psi^2},\\
|\del_h \mathbf w_h| &\lesssim \frac{r}{2}\Phi''\frac{x_3}{\psi^3}+\frac{r}{2}\Phi'\frac{1}{\psi^2}+\Phi'\frac{x_3}{\psi^2}+\frac{r \del_r\psi}{2}\Phi''\cdot \bigg(\frac{x_3}{\psi^2}\bigg)^2+\Phi'\frac{x_3 r\del_r\psi}{\psi^3}.
\end{align*}
Using again $x_3\leq \psi$ and the bounds \eqref{bdsPhi}, we have
\begin{align*}
    |\nabla\mathbf w_h| &\lesssim |\del_r\mathbf w_h|+|\del_3\mathbf w_h|+\bigg|\frac{\mathbf w_h\cdot\mathbf e_r}{r}\bigg| \lesssim \frac1\psi + \frac{r}{\psi^2} +\frac{r\del_r\psi}{\psi^2} + \frac{\del_r\psi}{\psi} + \frac{r(\del_r\psi)^2}{\psi^2} + \frac{r\del_r^2\psi}{\psi},\\
    |\del_h \mathbf w_h| &\lesssim \frac1\psi+\frac{r}{\psi^2}+\frac{r\del_r\psi}{\psi^2}.
\end{align*}
Note that these bounds hold independently of the specific form of $\psi$. In our setting, $\psi(r)=h+r^{1+\alpha}$. Thus, the proof of the remaining estimates on $\mathbf w_h$, $\nabla\mathbf w_h$ and $\del_h\mathbf w_h$ are based on the following result: we have
\begin{align}\label{Lem:est}
    \int_0^{r_0} \frac{r^q}{(h+r^{1+\alpha})^s} \dd r \leq \int_0^{r_0} r^{q-s(1+\alpha)} \dd r \lesssim 1 \ \forall \ (\alpha,q,s)\in (0,\infty)^3 \ \mbox{  satisfying  } \ q+1>s(1+\alpha);
\end{align}
see also Lemma~\ref{lem:i1} for a generalization to other values of $\alpha, q, s$. Since we don't need this refinement here, we give the details later on.\\

Using the estimate \eqref{Lem:est}, we get
\begin{align*}
\int_{\Omega_{h,r_0}} |\mathbf w_h|^q \dd x &\lesssim 1 + \int_0^{r_0} \int_0^\psi \frac{r^{q+1}}{\psi^q}+\frac{r^{(1+\alpha)q+1}}{\psi^q} \dd x_3 \dd r \\
&\lesssim 1+ \int_0^{r_0} \frac{r^{q+1}}{\psi^{q-1}}+\frac{r^{(1+\alpha)q+1}}{\psi^{q-1}} \dd r \lesssim 1\\
\Leftrightarrow \ q+2 &> (q-1)(1+\alpha) \ \mbox{ and } \ q(1+\alpha)+2>(q-1)(1+\alpha)\\
\Leftrightarrow \ \alpha (q-1) &< 3.
\end{align*}
Using now $r\del_r\psi\lesssim \psi$ and $r\del_r^2\psi\lesssim \del_r\psi$, we further have
\begin{align*}
|\nabla \mathbf w_h| \lesssim |\del_r\mathbf w_h|+|\del_3\mathbf w_h|+\Big|\frac{\mathbf w_h \cdot \mathbf e_r}{r}\Big|\lesssim \frac1\psi+\frac{r}{\psi^2}+\frac{\del_r \psi}{\psi}, \qquad |\del_h \mathbf w_h| \lesssim \frac1\psi+\frac{r}{\psi^2}.
\end{align*}
In particular, it is enough to estimate $\nabla\mathbf w_h$, since the most restrictive term is $r/\psi^2$. Hence, we obtain
\begin{align*}
&\int_{\Omega_{h,r_0}} |\nabla\mathbf w_h|^q \dd x \lesssim \int_0^{r_0} \int_0^\psi \frac{r}{\psi^q} + \frac{r^{q+1}}{\psi^{2q}} + \frac{r (\del_r\psi)^q}{\psi^q} \dd x_3 \dd r \lesssim \int_0^{r_0} \frac{r}{\psi^{q-1}} + \frac{r^{q+1}}{\psi^{2q-1}} + \frac{r^{\alpha q+1}}{\psi^{q-1}} \dd r \lesssim 1\\
&\Leftrightarrow \ 2 > (q-1)(1+\alpha) \ \mbox{ and } \ q+2>(2q-1)(1+\alpha) \ \mbox{ and } \ \alpha q+2>(q-1)(1+\alpha)\\
&\Leftrightarrow \ q < \frac{3+\alpha}{1+2\alpha}.
\end{align*}
\end{proof}

\begin{Remark}\label{rem:fits}
The condition $\alpha (q-1)<3$ coming from $\mathbf w_h$ is consistent with the results of \cite[Theorem~3.2]{Starovoitov2003}, where the author showed that collision is forbidden as long as $\alpha(q-1) \geq 3$. Especially, for shapes of class $C^{1,1}$ like balls, this states that no collision can occur as long as $q \geq 4$, which fits the assumptions made in \cite{FeireislHillairetNecasova2008} and \cite{Necasova2009}. Moreover, the difference $q - \frac{2+\alpha}{1+2\alpha}$ occurs in the incompressible two-dimensional setting in \cite[Theorem~3.2]{FilippasTersenov2021} as an optimal value for the solid to move vertically. Our fraction $\frac{3+\alpha}{1+2\alpha}$ thus seems to be a three-dimensional counterpart to that; see also the work \cite{FilippasTersenov2024} for the 3D case, where precisely this value occurs. Furthermore, we shall compare this with the value of $\beta$ occurring in Theorem~\ref{thm:Starov} later on.
\end{Remark}



\section{Estimates near the collision -- Proof of Theorem \ref{theo1}}\label{sec:PfThm}
Let $0<T<T_*$ and let $\zeta\in C^1([0,T))$ with $0\leq \zeta\leq 1$, $\zeta'\leq 0$, $\zeta(T)=0$, and $\zeta=1$ near $t=0$. For instance, setting $\tilde{\zeta}(t) = \exp[T^{-2} - (T^2-t^2)^{-1}]$, the function
\begin{align}\label{zeta}
\zeta(t) = \begin{cases}
1 & \text{if } 0 \leq t < (1-\frac1k)T,\\
\tilde{\zeta}(kt-(k-1)T) & \text{else}
\end{cases}
\end{align}
for some $k\geq 2$ will do the job. We take $\zeta(t){\mathbf w}_{h(t)}$ as test function in the weak formulation of the momentum equation \eqref{B-NSF}$_2$ with right-hand side ${\mathbf f}=-g{\mathbf e}_3$, $g>0$. Recalling $\div\mathbf w_h=0$ and $\del_t \mathbf w_{h(t)}=\dot{h}(t) \del_h\mathbf w_{h(t)}$, we get
\begin{align} \label{e1}
\begin{split}
&\int_0^T\zeta\int_\Omega \rho \vu\otimes \vu:\bD({\mathbf w}_h) \dd x \dd t + \int^T_0\zeta'\int_\Omega \rho \vu\cdot {\mathbf w}_h \dd x \dd t\\
&\quad + \int^T_0\zeta \dot{h}\int_\Omega \rho \vu\cdot {\del_h \mathbf w}_h \dd x \dd t - \int^T_0\zeta\int_\Omega {\bS} :\bD({ \mathbf w}_h) \dd x \dd t\\
= &\int_0^T\zeta \int_\Omega \rho g \mathbf e_3\cdot\mathbf w_h \dd x \dd t - \int_\Omega \mathbf m_0\cdot \mathbf w_h \dd x\\
= &\int_0^T\zeta \int_{\ms_h} \rho g \mathbf e_3\cdot\mathbf w_h \dd x \dd t + \int_0^T\zeta \int_{\mf_h} \rho g \mathbf e_3\cdot\mathbf w_h \dd x \dd t - \int_\Omega \mathbf m_0\cdot \mathbf w_h \dd x.
\end{split}
\end{align}
Observe that we have $\mathbf w_h=\mathbf e_3$ on $\ms_h$, so for a sequence $\zeta_k\to 1$ in $L^1([0,T))$,
\begin{align*}
    &\int^T_0\zeta_k \int_{\ms_h} \rho g{\mathbf e}_3\cdot {\mathbf w}_h \dd x \dd t = \int_0^T\zeta_k \int_{\ms_h} \rho_\ms g \to mgT.
\end{align*}
An example of such a sequence is precisely given by \eqref{zeta}. In particular, for a proper choice of $\zeta$, it follows that
\begin{multline}\label{mom}
\frac12 mgT \leq \int_0^T\zeta\int_\Omega \rho \vu\otimes \vu:\bD({\mathbf w}_h) \dd x \dd t + \int^T_0\zeta'\int_\Omega \rho \vu\cdot {\mathbf w}_h \dd x \dd t + \int^T_0\zeta \dot{h}\int_\Omega \rho \vu\cdot {\del_h \mathbf w}_h \dd x \dd t\\
\quad - \int^T_0\zeta\int_\Omega {\bS} :\bD({ \mathbf w}_h) \dd x \dd t
- \int_0^T\zeta \int_{\mf_h} \rho g \mathbf e_3\cdot\mathbf w_h \dd x \dd t + \int_\Omega \mathbf m_0\cdot \mathbf w_h \dd x
 = \sum_{j=1}^6 I_j.
\end{multline}

We will estimate each $I_j$ separately, and set our focus on the explicit dependence on $T$ and $m$. For the latter purpose, we split each density dependent integral into its fluid and solid part $I_j^f$ and $I_j^s$, respectively.\\

$\bullet$ For $I_2^f$, we have by $|\zeta'| = - \zeta' \geq 0$, $\zeta(T)=0$, and $\zeta(0)= 1$
\begin{align*}
    |I_2^f| &\leq -\int_0^T\zeta'\int_{\mf_h} \rho |\vu| |\mathbf w_h| \dd x \dd t = -\int_0^T\zeta'\int_{\mf_h} \sqrt{\rho} \sqrt{\rho} |\vu| |\mathbf w_h| \dd x \dd t\\
    &\leq -\int_0^T\zeta' \|\sqrt{\rho}\|_{L^{2\gamma}({\mf_h})} \|\sqrt{\rho}\vu\|_{L^2({\mf_h})} \|\mathbf w_h\|_{L^\frac{2\gamma}{\gamma-1}({\mf_h})} \dd t\\
    &\leq \|\rho\|_{L^\infty(0,T;L^\gamma({\mf(\cdot)}))}^\frac12 \|\rho |\vu|^2\|_{L^\infty(0,T;L^1(\Omega))}^\frac12 \|\mathbf w_h\|_{L^\infty(0,T;L^\frac{2\gamma}{\gamma-1}(\mf(\cdot)))} \zeta(0)\lesssim (E_0+1)^{\frac{1}{2\gamma}+\frac12},
\end{align*}
where we have used the estimate \eqref{UnifBds} and Lemma \ref{BdsWh} under the condition 
\begin{align*}
    \frac{2\gamma}{\gamma-1}<1+\frac3\alpha \Leftrightarrow \alpha < \frac{3(\gamma-1)}{\gamma+1}.
\end{align*}

$\bullet$ For $I_2^s$, recall that the solid rotates at most around the $x_3$-axis, hence $\omega(t) = \pm |\omega(t)| \vc e_3$. Further, due to $\vu |_{\ms_h} = \dot{\vc G}(t) + \omega \times (x - \vc G(t))$, $\vc G(t) = \vc G(0) + (h(t)-h(0))\vc e_3$, $\rho |_{\ms_h} = \rho_\ms>0$, and $\vc w_h |_{\ms_h} = \vc e_3$, we have
\begin{align*}
\int_{\ms_h} \rho \vu \cdot \vc w_h \dd x = \rho_\ms \int_{\ms_h} \big[ \dot{h} \vc e_3 \pm |\omega|\vc e_3 \times (x - \vc G(0) - (h-h(0))\vc e_3) \big] \cdot \vc e_3 \dd x = m \dot{h}.
\end{align*}
Moreover, from the bounds \eqref{UnifBds}, we infer
\begin{align}\label{est:h}
\sup_{t\in (0,T)} |\dot h|^2 = \sup_{t \in (0,T)} \frac2m \int_{\ms_h} \rho_\ms |\dot h|^2 \dd x \leq \sup_{t \in (0,T)} \frac2m \int_{\ms_h} \rho_\ms |\vc u|^2 \dd x \lesssim \frac2m (E_0+1).
\end{align}
Hence, by the choice of $\zeta$ such that $|\zeta'|=-\zeta'$ and $\zeta(0)=1+\zeta(T)=1$, we get
\begin{align*}
    |I_2^s| \lesssim -\int_0^T \zeta' m |\dot h| \dd t \lesssim \sqrt{m}(E_0 + 1)^\frac12.
\end{align*}

$\bullet$ For $I_3$, observe that $I_3^s=0$ due to $\del_h \mathbf w_h|_{\ms_h} = \del_h \mathbf e_3=0$. Next, by Sobolev embedding \eqref{SobEmb} and \eqref{UnifBds},
\begin{align*}
    \|\vu\|_{L^p(0,T;L^{p^\ast}(\Omega))} \lesssim \|\vu\|_{L^p(0,T;W_0^{1,p}(\Omega))} \lesssim (E_0+1)^\frac1p,
\end{align*}
where we set $p^\ast = 3p/(3-p)$. Thus,
\begin{align*}
    |I_3|&=|I_3^f| \leq \int_0^T\zeta |\dot{h}(t)|\, \|\rho\|_{L^\infty(0,T;L^\gamma(\mf(\cdot)))} \|\vu\|_{L^p(0,T;L^{p^\ast}(\Omega))} \|\del_h \mathbf w_h\|_{L^\frac{p^\ast \gamma}{p^\ast (\gamma-1)-\gamma}(\mf(\cdot))} \dd t\\
    &\lesssim (E_0+1)^{\frac{1}{\gamma}+\frac1p} \|\dot{h}\|_{L^\infty(0,T)} \|\zeta\|_{L^1(0,T)}\lesssim \sqrt{\frac1m} (E_0+1)^{\frac1\gamma+ \frac1p + \frac12} T,
\end{align*}
where we have used the estimates \eqref{UnifBds}, \eqref{est:h}, and Lemma \ref{BdsWh} under the condition
\begin{align*}
    \frac{p^\ast \gamma}{p^\ast (\gamma-1) - \gamma} < \frac{3+\alpha}{1+2\alpha} \Leftrightarrow \alpha < \frac{2p^\ast \gamma - 3p^\ast - 3\gamma}{p^\ast \gamma + p^\ast + \gamma} = \frac{9(p\gamma - p - \gamma)}{2p\gamma + 3p + 3\gamma}.
\end{align*}

$\bullet$ Regarding $I_4$, using that $\bS \in L^{p'}((0,T) \times \Omega)$ is bounded by $c_1>0$ (see \eqref{Sp}), we calculate
\begin{align*}
    |I_4| &\lesssim \int_0^T \zeta \|\bS\|_{L^{p'}(\Omega)} \|\nabla \mathbf w_h\|_{L^p(\Omega)} \dd t \leq \|\zeta\|_{L^\frac{p}{p-1}(0,T)} \|\bS\|_{L^{p'}((0,T)\times \Omega)} \|\nabla\mathbf w_h\|_{L^\infty(0,T;L^p(\Omega))}\\
    &\lesssim (E_0+1)^\frac{1}{p'} T^\frac{1}{p'},
\end{align*}
where we have used Lemma \ref{BdsWh} under the condition
\begin{align*}
    p<\frac{3+\alpha}{1+2\alpha} \Leftrightarrow \alpha< \frac{3-p}{2p-1}.
\end{align*}

$\bullet$ For $I_5=I_5^f$,
\begin{align*}
    |I_5| &\leq g \int_0^T\zeta \|\rho\|_{L^\gamma(\mf_h)}\|\mathbf w_h\|_{L^\frac{\gamma}{\gamma-1}(\Omega)} \leq g \|\zeta\|_{L^1(0,T)} \|\rho\|_{L^\infty(0,T;L^\gamma(\mf(\cdot)))} \|\mathbf w_h\|_{L^\infty(0,T;L^\frac{\gamma}{\gamma-1}(\Omega))}\\
    &\leq g (E_0+1)^\frac1\gamma T,
\end{align*}
by using Lemma \ref{BdsWh} under the condition
\begin{align*}
    \frac{\gamma}{\gamma-1}<1+\frac3\alpha \Leftrightarrow \alpha< 3(\gamma-1).
\end{align*}

$\bullet$ Similar to $I_2^f$, we have for $I_6^f$ the estimate
\begin{align*}
    |I_6^f| \leq \|\mathbf m_0\|_{L^\frac{2\gamma}{\gamma+1}(\mf(0))}\|\mathbf w_h\|_{L^\infty(0,T;L^\frac{2\gamma}{\gamma-1}(\Omega))} \lesssim \bigg\|\frac{|\mathbf m_0|^2}{\rho_0}\bigg\|_{L^1(\mf(0))}^\frac12 \|\rho_0\|_{L^\gamma(\mf(0))}^\frac12 \lesssim (E_0+1)^{\frac12+\frac{1}{2\gamma}}.
\end{align*}

$\bullet$ For $I_6^s$, where $\mathbf w_h=\mathbf e_3$, $\mathbf m_0=(\rho\vu)(0)=\rho_\ms (\dot{h} \mathbf e_3+\omega\times (x-h\mathbf e_3))$, and $\omega = \pm |\omega| \vc e_3$, we have similarly to $I_2^s$ that
\begin{align*}
    |I_6^s| = \bigg|\int_{\ms(0)} \mathbf m_0\cdot \mathbf e_3 \dd x\bigg| = \bigg|\int_{\ms(0)} \rho_\ms \dot{h} \dd x\bigg| \leq m \|\dot{h}\|_{L^\infty(0,T)} \lesssim \sqrt{m}(E_0+1)^\frac12.
\end{align*}

$\bullet$ Let us turn to $I_1$. Due to $\mathbf w_h|_{\ms_h}=\mathbf e_3$, we see that $I_1^s=0$ since $\bD(\mathbf w_h)=0$ there. Hence, we calculate
\begin{align*}
    |I_1|&= |I_1^f| \lesssim \int_0^T \zeta \|\rho\|_{L^\gamma(\mf_h))} \|\vu\|_{L^{p^\ast}(\Omega)}^2 \|\nabla\mathbf w_h\|_{L^\frac{p^\ast \gamma}{p^\ast (\gamma-1) - 2\gamma}(\Omega)} \dd t\\
    &\lesssim \|\rho\|_{L^\infty(0,T;L^\gamma(\mf_h))} \|\nabla\mathbf w_h\|_{L^\infty(0,T;L^\frac{p^\ast \gamma}{p^\ast (\gamma-1) - 2\gamma}(\Omega))} \int_0^T\zeta \|\nabla\vu \|_{L^p(\Omega)}^2 \dd t\\
    &\lesssim (E_0+1)^\frac1\gamma \|\zeta \|_{L^\frac{p}{p-2}(0,T)} \|\nabla\vu\|_{L^p((0,T)\times\Omega)}^2 \lesssim (E_0+1)^{\frac1\gamma+\frac2p} T^{1-\frac2p},
\end{align*}
by using the estimate \eqref{UnifBds} and Lemma \ref{BdsWh} under the condition
\begin{align*}
    \frac{p^\ast \gamma}{p^\ast (\gamma-1) - 2\gamma}<\frac{3+\alpha}{1+2\alpha} \Leftrightarrow \alpha<\frac{2p^\ast \gamma - 3p^\ast - 6\gamma}{p^\ast \gamma + p^\ast + 2\gamma} =
\frac{3(4p \gamma - 3p - 6\gamma)}{p\gamma + 3p + 6\gamma}. 
\end{align*}
Let us emphasize that this term is the only place where the assumption $p\geq 2$ is needed.\\

Collecting all the requirements made above, we infer
\begin{align*}
\gamma>\frac32, \ \ 2 \leq p < 3, \ \ p\gamma>p+\gamma, \ \ 4p\gamma>3p+6\gamma,
\end{align*}
which translates into
\begin{align*}
\frac32 < \gamma \leq 3, \ \frac{6\gamma}{4\gamma-3} < p < 3, \quad \text{or} \quad \gamma>3, \ 2 \leq p < 3.
\end{align*}

Note further that for any $\gamma\geq \frac32$ and any $\frac{\gamma}{\gamma-1}<p<3$,
\begin{align*}
    \frac{3(4p \gamma - 3p - 6\gamma)}{p\gamma + 3p + 6\gamma}\leq \frac{9(p\gamma - p - \gamma)}{2p\gamma + 3p + 3\gamma} \leq \frac{3(\gamma-1)}{\gamma+1} \leq 3(\gamma-1),
\end{align*}
and that all estimates are independent of the choice of $\zeta$. Hence, we can take a sequence $\zeta_k\to 1$ in $L^r([0,T))$ for some suitable $r>1$ without changing the bounds obtained (again, \eqref{zeta} is a suitable choice). In turn, collecting all the estimates above, we finally arrive at
\begin{align*}
    \frac12 mgT \leq C_0 (1+\sqrt{m}+\sqrt{m}^{-1}) \bigg(1 + (E_0+1)^{\frac12+\frac{1}{2\gamma}} + (E_0+1)^\frac12 + (E_0+1)^\frac1p\\
    \quad + (E_0+1)^{\frac12 + \frac1\gamma + \frac1p} + g(E_0+1)^\frac1\gamma + (E_0+1)^{\frac1\gamma+\frac2p} \bigg) (1 + T^\frac{1}{p'} + T^\frac{1}{p} + T^{1-\frac2p} + T),
\end{align*}
which, after dividing by $\frac12 mg$, noticing that due to $p \geq 2$ the largest exponent is $\frac12 + \frac{1}{\gamma} + \frac1p$, and using Young's inequality \eqref{Young} on several terms, leads to
\begin{align}\label{infinalIneq}
    T \leq C_0 \max\{m^{-1/2}, m^{-3/2} \} \bigg(1 + E_0^{\frac12 + \frac{1}{\gamma} + \frac1p} \bigg) (1 + T),
\end{align}
where $C_0$ only depends on $p, \gamma, g, \alpha$, the bounds on $\mathbf w_h$ obtained in Lemma \ref{BdsWh}, and the implicit constant appearing in \eqref{UnifBds}, provided
\begin{align*}
	&\alpha<\min \bigg\{ \frac{3-p}{2p-1}, \frac{3(4p\gamma - 3p - 6\gamma)}{p\gamma + 3p + 6\gamma} \bigg\} \quad \text{with}\\
    &\frac32< \gamma \leq 3, \ \frac{6\gamma}{4\gamma-3} < p < 3, \quad \text{or} \quad \gamma>3, \ 2 \leq p < 3.
\end{align*}

Recalling the definition of $E_0$ from \eqref{eq:E0} as
\begin{align*}
    E_0 &= \int_{\mf(0)}\bigg(\frac{|\mathbf m_0|^2}{2\rho_0} +  P(\rho_0, \vartheta_0) \bigg) \dd x + \frac{m}{2}|\mathbf V_0|^2 + \frac12 \mathbb J(0) \omega_0\cdot\omega_0,\\
    \mathbb J(0) &= \int_{\ms_0 - \vc G_0}\rho_\ms\Big(|x|^2{\mathbb I}-x\otimes x\Big) \dd x,
\end{align*}

we see that collision can occur only if the solid's mass in \eqref{infinalIneq} is large enough, meaning in fact its density is very high. Heuristically, this shall be clear: a light object would swim rather than sink. Since $E_0$ depends on the solid's mass, we require the solid initially to have low vertical and rotational speed. More precisely, choosing $\mathbf V_0$ and $\omega_0$ such that $|\mathbf V_0|, |\omega_0|=\mathcal{O}(m^{-\frac12})$, and choosing $m$ high enough such that
\begin{align}\label{finalIneq}
    C_0 \max\{m^{-1/2}, m^{-3/2} \} \bigg(1 + E_0^{\frac12 + \frac{1}{\gamma} + \frac1p} \bigg) < 1,
\end{align}
the solid touches the boundary of $\Omega$ in finite time, ending the proof of Theorem \ref{theo1}.

\begin{Remark}
We see that if, by change, the constant $C_0<1$ small enough, then we can get rid of the assumption on the smallness of $\vc V_0$ and $\omega_0$ by also choosing $m<1$. Indeed, in this case $\max\{m^{-1/2}, m^{-3/2} \} = m^{-3/2}$ and $E_0 \lesssim 1$. Hence, for appropriate values $m<1$ and $C_0 m^{-3/2} < 1$, inequality \eqref{finalIneq} can still be valid.
\end{Remark}

\begin{Remark}\label{rem3}
For the case $d=2$, although the construction of $\vc w_h$ is slightly different (see Section~\ref{sec:62}), the same proof of Lemma~\ref{BdsWh} shows
\begin{align*}
\|\vc w_h\|_{L^q(\Omega_{h, r_0})} &\lesssim 1 \text{ for any } q < 1 + \frac{2}{\alpha}, \\
\|\del_h \vc w_h\|_{L^q(\Omega_{h, r_0})} + \|\nabla \vc w_h\|_{L^q(\Omega_{h, r_0})} &\lesssim 1 \text{ for any } q < \frac{2+\alpha}{1+2\alpha},
\end{align*}
the range of $q$ we shall compare with the one from Remark~\ref{rem:fits} and Theorem~\ref{thm:Starov}. Regarding the estimate of $I_4$, we thus find again $|I_4| \lesssim (E_0 + 1)^\frac{1}{p'} T^\frac{1}{p'}$ provided
\begin{align*}
p < \frac{2+\alpha}{1+2\alpha} \Leftrightarrow \alpha < \frac{2-p}{2p-1}, \quad 1<p<2.
\end{align*}
The estimate for the convective term $I_1$, however, still needs $p \geq 2$ to comply with the square integrability of $\nabla \vu$ in time. Hence, for the two-dimensional Navier-Stokes equations, we cannot conclude that collision happens. On the other hand, if we drop the convective term, the estimates for the fluid parts of the integrals remain the same provided
\begin{align*}
&\text{for } I_2^f: \quad \frac{2\gamma}{\gamma-1} < 1 + \frac{2}{\alpha} \Leftrightarrow \alpha < \frac{2(\gamma-1)}{\gamma+1},\\
&\text{for } I_5: \quad \frac{\gamma}{\gamma-1} < 1 + \frac{2}{\alpha} \Leftrightarrow \alpha < 2 (\gamma-1),\\
&\text{for } I_3: \quad \frac{p^\ast}{p^\ast (\gamma-1) - \gamma} < \frac{2+\alpha}{1+2\alpha},
\end{align*}
where now $p^\ast = 2p/(2-p)$. This last condition is equivalent to
\begin{align*}
\alpha < \frac{p^\ast \gamma - 2 p^\ast - 2 \gamma}{p^\ast \gamma + p^\ast + \gamma} = \frac{4(p \gamma - p - \gamma)}{p\gamma+2p+2\gamma}, \quad p > \frac{\gamma}{\gamma-1}.
\end{align*}
Hence, we find that collision for the two-dimensional compressible Stokes equations happens provided
\begin{align*}
\frac{\gamma}{\gamma-1}<p<2, \quad \gamma>2, \quad \alpha < \min\left\{ \frac{2-p}{2p-1}, \frac{4(p \gamma - p - \gamma)}{p\gamma+2p+2\gamma} \right\}.
\end{align*}
Note that this corresponds purely to the case of \emph{shear-thinning} fluids, in contrast to the three-dimensional case.
\end{Remark}



\chapter{Specific model examples}\label{ch:Ex}
This chapter is devoted to investigate some precise examples of (Navier-)Stokes and Navier-Stokes-Fourier equations as well as stress tensors $\bS$ fulfilling the requirements \ref{S1}--\ref{S3}. We start with the easiest case of incompressible Stokes equations, and will end with a model of compressible, heat conducting, non-Newtonian fluids. Again, we strongly remark that the question of \emph{existence} of weak and/or strong solutions is just known in some special cases, and sometimes even just for fluids without immersed bodies. However, as the system of fluid-structure can be written as a single system of pure-fluid type, for which existence results where obtained, it is not unreasonable to assume that such a solution with the desired properties exists. We will give some references to available existence results at appropriate places. Moreover, the fact that collision might occur shows that the solutions, if available, have just \emph{finite time of existence}; in particular, they cannot be prolonged after the time $T_\ast>0$ where collision happens, which mathematically results in a blow-up of some $L^q$-norms of the velocity's gradient (see also Remark~\ref{rem2}).

\section{Incompressible Stokes}
Probably the easiest model of fluid flow around a rigid object travelling through the fluid is given by the following set of linear, incompressible Stokes equations:
\begin{align}\label{Stokes}
\begin{cases}
\div \vu = 0 & \text{in } \mf(t),\\
\del_t\vu - \div \bS + \nabla p = -g \mathbf e_3 & \text{in } \mf(t),\\
\vu = 0 & \text{on } \del \Omega,\\
\vu = \dot{\mathbf G}(t) + \omega(t)\times (x-\mathbf G(t)) & \text{on } \del \ms(t),\\
m \ddot{\mathbf G} = -\int_{\del \ms} (\bS - p \Id)\mathbf n \dd \sigma - \int_{\ms} \rho_\ms g \mathbf e_3 \dd x,\\
\frac{\mathrm{d}}{\mathrm{d} t}(\mathbb{J} \omega) = - \int_{\del \ms} (x-\mathbf G) \times (\bS - p\Id)\mathbf n \dd \sigma - \int_{\ms} (x-\mathbf G) \times \rho_\ms g \mathbf e_3 \dd x,\\
\vu(0)=\vu_0, \ \mathbf G(0)=\mathbf G_0, \ \dot{\mathbf G}(0) = \mathbf V_0, \ \omega(0)=\omega_0 & \text{in } \mf(0),
\end{cases}
\end{align}
where $(\rho_\ms, m, g) \in (0,\infty)^3$, $\mathbb{J} \in \R^{3\times 3}$, and $\vu_0, \vc G_0, \vc V_0, \omega_0$ are as before. The stress tensor is given by Newton's rheological law
\begin{align*}
\bS = 2\mu \bD(\vu) = \mu \big( \nabla \vu + \nabla^T \vu \big), \quad \mu>0,
\end{align*}
thus the term $\div \bS$ can also be written in the more common form $\mu \Delta \vu$, and we set the (constant) fluid's density to be equal to $1$. As easily seen, $\bS$ fulfils all the requirements stated in \ref{S1}--\ref{S3}. Indeed, continuity \ref{S1} is obvious from the definition of $\bS$. Regarding monotonicity \ref{S2}, for all $\mathbb{M}, \mathbb{N} \in \R_{\rm sym}^{3 \times 3}$ we have
\begin{align*}
[\bS(\mathbb{M}) - \bS(\mathbb{N})] : (\mathbb{M} - \mathbb{N}) = 2\mu |\mathbb{M}-\mathbb{N}|^2 \geq 0.
\end{align*}
Setting $\mathbb{N}=0$ in the above, we conclude the growth condition \ref{S3} by choosing $c_0=c_1=2\mu >0$, $\delta=0$, and $p=2$.\\

The definition of weak solutions for system \eqref{Stokes} is classical. For completeness, we state it here.
\begin{Definition}\label{def:Stokes}
Let $\vu_0 \in L^2(\Omega)$ such that $\div \vu_0=0$ and $\bD(\vu_0) = 0$ on $\ms$. We say that $\vu \in L^\infty(0,T;L^2(\Omega)) \cap L^2(0,T;W_0^{1,2}(\Omega))$ with $\div \vu = 0$ is a finite energy weak solution to \eqref{Stokes} if:
\begin{itemize}
\item The weak formulation of the momentum equation holds:
\begin{align*}
&\int_\Omega \vu(\tau) \cdot \phi(\tau) \dd x - \int_0^\tau \int_\Omega \vu \cdot \del_t \phi \dd x \dd t + \int_0^\tau \int_\Omega \bS : \nabla \phi \dd x \dd t \\
&= \int_\Omega \vu_0 \cdot \phi(0) \dd x + \int_0^\tau \int_\Omega \mathbf f \cdot \phi \dd x \dd t 
\end{align*}
with $\mathbf f = -g \mathbf e_3$, for any $\phi \in C_c^\infty([0,T) \times \Omega)$ such that $\div \phi = 0$ and $\bD(\phi)=0$ in a neighborhood of $\ms$;
\item The energy inequality
\begin{align}\label{EISt}
\int_\Omega \frac12 |\vu|^2(\tau) \dd x + \int_0^\tau \int_\Omega \bS:\nabla \vu \dd x \dd t \leq \int_\Omega \frac12 |\vu_0|^2 \dd x + \int_0^\tau \int_\Omega \mathbf f \cdot \vu \dd x \dd t
\end{align}
holds for almost any $\tau \in [0, T]$, where $\mathbf f = -g \mathbf e_3$.
\end{itemize}
\end{Definition}
We note that the energy \emph{equality} formally follows from multiplying the momentum equation by the solution $\vu$, and indeed holds for smooth (classical) solutions. That it is replaced in \eqref{EISt} by an \emph{inequality} is due to the fact that merely \emph{weak} solutions can dissipate more energy than expected. Another point of view is that norms are just weakly lower semi-continuous rather than continuous, thus the former equality changes into inequality.\\

By the classical theory for Stokes and Navier-Stokes equations (see \cite{Feireisl2003, Ladyzhenskaya1969, Ladyzhenskaya1975}), we can state the following
\begin{Theorem}
Let the initial datum $\vu_0 \in L^2(\Omega)$ with $\div \vu_0=0$ and $\bD(\vu_0)=0$ on $\ms$. Then there exists a weak solution to system \eqref{Stokes} in the sense of Definition~\ref{def:Stokes}.
\end{Theorem}

From inequality \eqref{EISt}, we immediately get the desired bounds on the velocity. Indeed, using that for gravity $g\mathbf e_3 = g \nabla [x\mapsto x_3]$, we integrate by parts and use the zero boundary conditions of $\vu$ on $\del \Omega$ to obtain
\begin{align}\label{consEner}
\int_0^\tau \int_\Omega \mathbf f \cdot \vu \dd x \dd t = -\int_0^\tau \int_\Omega g \mathbf e_3 \cdot \vu \dd x \dd t = \int_0^\tau \int_\Omega g x_3 \div \vu \dd x \dd t = 0
\end{align}
by the solenoidality of $\vu$. Hence, using further Korn's and Poincar\'e's inequality \eqref{Korn}--\eqref{Poinc}, we end up with
\begin{align*}
\|\vu\|_{L^\infty(0,T;L^2(\Omega))}^2 + \|\vu\|_{L^2(0,T;W_0^{1,2}(\Omega))}^2 \leq C(\Omega, \mu) E_0, \quad E_0 = \int_\Omega \frac12 |\vu_0|^2 \dd x.
\end{align*}

In view of the above, Theorem~\ref{theo1} is applicable and we ensure collision as long as $\alpha < \frac13$.

\begin{Remark}
As already stated, collision can be proven even for the case $\alpha \geq \frac13$ by means of analyzing the drag force and the corresponding ODE for the distance $h(t)$, see Chapter~\ref{ch:6} and \cite[Section~3.1]{GerardVaretHillairet2012}.
\end{Remark}

\begin{Remark}\label{rem:simply}
To get a lean notation in the following sections, we emphasize that the equations \eqref{Stokes}$_3$--\eqref{Stokes}$_6$ will not change in any model under consideration, so we skip their occurrence later on. Moreover, in the case of compressible and/or heat conducting fluids, the initial conditions \eqref{Stokes}$_7$ are completed with initial conditions for the new unknowns, so we also omit them in the sequel.
\end{Remark}

\section{Incompressible non-Newtonian Navier-Stokes}
Another well-known model is the non-linear Navier-Stokes system as investigated in \cite{BGMSG2012, DieningRuzickaWolf2010, Ladyzhenskaya1969, Wolf2007} (for the pure fluid system). With regard to Remark~\ref{rem:simply}, we modify the momentum equation \eqref{Stokes}$_2$ in the following way:
\begin{align}\label{INSNN}
\del_t \vu + \div(\vu \otimes \vu) - \div \bS + \nabla p = -g\mathbf e_3 \quad \text{in } \mf(t),
\end{align}
where the stress tensor $\bS=\bS(\bD(\vu))$ now is just assumed to satisfy the assumptions \ref{S1}--\ref{S3} with $\vartheta \equiv \text{constant}$. All other equations in \eqref{Stokes} remain unchanged. Also, the notion of weak solutions as given in Definition~\ref{def:Stokes} is the same with obvious changes in the weak momentum formulation due to the presence of the convective term, and the energy inequality \eqref{EISt} still remains valid due to
\begin{align*}
\int_\Omega \div(\vu \otimes \vu) \cdot \vu \dd x = \int_\Omega \frac12 \div(|\vu|^2 \vu) \dd x = \int_{\del \Omega} \frac12 |\vu|^2 \vu \cdot \vc n \dd x = 0.
\end{align*}
Using the growth condition \ref{S3} on $\bS$, we obtain in the same way as before the uniform bounds
\begin{align*}
\|\vu\|_{L^\infty(0,T;L^2(\Omega))}^2 + \|\vu\|_{L^p(0,T;W_0^{1,p}(\Omega))}^p \leq C(\Omega, \delta, c_0, c_1, p) (E_0 + 1).
\end{align*}
Hence, Theorem~\ref{theo1} yields collision as long as $2\leq p < 3$ and $\alpha < \frac{3-p}{2p-1}$. We recall that it seems reasonable that $\alpha \to 0$ as $p \to 3$, since the lower tip of the solid shall be ``sharper'' for a shear-thickening fluid to ``cut'' through it. Moreover, neglecting the convective part $\div(\vu \otimes \vu)$ to get a non-Newtonian incompressible Stokes system, we can handle all $p \in (1,3)$ and all $\alpha<\frac{3-p}{2p-1}$, thus also allowing for shear-thinning fluids.

\section{Incompressible non-Newtonian Navier-Stokes-Fourier}
Regarding heat conducting fluids, we have to insert in the model equations for the temperature. Moreover, the energy \emph{inequality} satisfied for weak solutions is now replaced by an energy \emph{equality}, together with an additional \emph{entropy inequality}. The new fluid's system reads
\begin{align*}
\begin{cases}
\div \vu = 0 & \text{in } \mf(t),\\
\del_t \vu + \div(\vu \otimes \vu) - \div\bS + \nabla p = -g\mathbf e_3 & \text{in } \mf(t),\\
\del_t \vartheta + \vu \cdot \nabla \vartheta + \div \mathbf q = \frac{1}{\vartheta} \Big( \bS : \bD(\vu) - \frac{\vc q(\vartheta, \nabla \vartheta) \cdot \nabla \vartheta}{\vartheta} \Big) & \text{in } \mf(t),\\
\big[ \mathbf q(\vartheta, \nabla \vartheta) - \mathbf q(\vartheta_s, \nabla \vartheta_s) \big] \cdot \mathbf n = 0 & \text{on } \del \ms(t),\\
\mathbf q \cdot \mathbf n = 0 & \text{on } \del \Omega,
\end{cases}
\end{align*}
where the heat flow vector is given by Fourier's law
\begin{align*}
\mathbf q(\vartheta, \nabla \vartheta) = -\kappa(\vartheta) \nabla \vartheta,
\end{align*}
and the heat conductivity $\kappa$ is assumed to be a continuous function of the temperature satisfying $\kappa(\vartheta) \sim 1+\vartheta^\beta$ for some $\beta>1$. Such models where investigated in \cite{Necasova2009} for the case of so-called Boussinesq approximation, where $g \vc e_3$ is replaced by $g \vartheta \vc e_3$, and collision was ruled out as the growth parameter $p \geq 4$, as well as in \cite{Roubivcek2009} and \cite{CHSWW2018}. In the latter reference, the constructed solutions satisfy the energy inequality
\begin{align*}
\frac{\rd}{\rd t} \int_\Omega \frac12 |\vu|^2 + \vartheta \dd x \leq 0,
\end{align*}
from which we infer
\begin{align*}
\|\vu\|_{L^\infty(0,T;L^2(\Omega))}^2 + \|\vartheta\|_{L^\infty(0,T;L^1(\Omega))} \leq E_0.
\end{align*}
Note that this does not provide us with any information on the gradient of the velocity. However, the entropy inequality given by
\begin{align*}
\frac{\rd}{\rd t} \int_\Omega \frac1\vartheta \bigg(\bS:\bD(\vu) - \frac{\mathbf q\cdot\nabla\vartheta}{\vartheta} \bigg) \dd x \dd t \leq 0
\end{align*}
forces
\begin{align*}
\|\vu\|_{L^p(0,T;W_0^{1,p}(\Omega))}^p + \|\nabla \log \vartheta\|_{L^2((0,T)\times \Omega)}^2 + \|\nabla \vartheta^\frac{\beta}{2}\|_{L^2((0,T)\times \Omega)}^2 \lesssim E_0.
\end{align*}
Recalling that the temperature estimates are not important for our analysis, we again arrive at the required estimate \eqref{UnifBds}.

\section{Compressible Navier-Stokes}
Taking into account that the density of the fluid might change over time and also from one position in space to another (as is the case for the prime example of gases\footnote{Note however carefully that even liquids are compressible: water has a compressibility of about $5\cdot 10^{-10} \, {\rm Pa}^{-1}$, which is clearly ``almost zero'' compared to air with  compressibility approximately $1 \, {\rm Pa}^{-1}$.}), we have to modify the fluid's system accordingly. More precisely, it now reads
\begin{align}\label{CNSE}
\begin{cases}
\del_t \rho + \div(\rho \vu) = 0 & \text{in } \mf(t),\\
\del_t(\rho \vu) + \div(\rho \vu \otimes \vu) - \div \bS + \nabla p(\rho) = -\rho g \mathbf e_3 & \text{in } \mf(t),\\
\rho(0)=\rho_0, \ (\rho \vu)(0)=\mathbf m_0 & \text{in } \mf(0),
\end{cases}
\end{align}
where in contrast to the incompressible models in the preceding sections, the pressure is now a function of the unknown density $\rho$. To specify the pressure growth given in \eqref{p-law}, we will assume a barotropic pressure law of the form
\begin{align*}
p\in C([0,\infty)) \cap C^2((0,\infty)),\quad p(0)=0, \quad p'(\rho)>0 \ (\rho>0), \quad \lim_{\rho \to \infty} \frac{p'(\rho)}{\rho^{\gamma-1}} = p_\infty>0,\quad \gamma>\frac{3}{2}.
\end{align*}
The easiest example of such a pressure is given by the usually used barotropic law $p(\rho)=\rho^\gamma$, although other examples are possible. The stress tensor is given by Newton's law
\begin{align*}
\bS = 2\mu \bigg( \bD(\vu) - \frac13 \div \vu \Id \bigg) + \eta \div \vu \Id, \quad \mu>0, \, \eta \geq 0,
\end{align*}
and fulfils requirements \ref{S1}--\ref{S3} the same way as in the incompressible case \eqref{Stokes}. Since the former part of $\bS$ is trace-free, the parameter $\mu$ is called \emph{shear viscosity}, whereas $\eta$ is commonly known as the \emph{bulk viscosity}. The definition of weak solutions is also similar to the one for incompressible fluids, and existence of such was shown in \cite{Feireisl2003}.
\begin{Definition}\label{def:wkSol}
Let $\gamma>\frac{3}{2}$, $\rho_0 \in L^\gamma(\Omega)$, and $\mathbf m_0 \in L^\frac{2\gamma}{\gamma+1}(\Omega)$, together with the compatibility conditions
\begin{align*}
\rho_0\geq 0, \quad \mathbf m_0 = 0 \text{ whenenver } \rho_0 = 0, \quad \frac{|\mathbf m_0|^2}{\rho_0} \in L^1(\Omega).
\end{align*}
We call a couple $(\rho, \vu)$ a finite energy weak solution of system \eqref{CNSE} if:
\begin{itemize}
\item The solution belongs to the regularity class
\begin{gather*}
\vu \in L^2(0,T;W_0^{1,2}(\Omega)), \quad \rho \in L^\infty(0,T;L^\gamma(\Omega)), \quad \rho |\vu|^2 \in L^\infty(0,T;L^1(\Omega)),\\
\bD(\vu) = 0 \text{ on } \ms, \quad \rho \geq 0 \text{ a.e.~in } (0,T) \times \Omega, \quad \rho = \rho_\ms \text{ on } \ms;
\end{gather*}
\item The weak formulation of the continuity equation holds:
\begin{align*}
\int_0^\tau \int_\Omega \rho (\del_t \psi + \vu \cdot \nabla \psi) \dd x \dd t = \int_\Omega \rho(\tau) \psi(\tau) \dd x - \int_\Omega \rho_0 \psi(0) \dd x
\end{align*}
for any $\psi \in C_c^\infty([0,T) \times \overline{\Omega})$;
\item The weak formulation of the momentum equation holds:
\begin{align*}
&\int_0^\tau \int_\Omega \rho \vu \cdot \del_t \phi + \rho \vu \otimes \vu : \nabla \phi + p(\rho) \div \phi - \bS:\nabla \phi + \rho \mathbf f \cdot \phi \dd x \dd t\\
&= \int_\Omega (\rho \vu)(\tau) \cdot \phi(\tau) \dd x - \int_\Omega \mathbf m_0 \cdot \phi(0) \dd x
\end{align*}
with $\mathbf f = -g \mathbf e_3$, for any $\phi \in C_c^\infty([0,T) \times \Omega)$ such that $\bD(\phi)=0$ in a neighborhood of $\ms$;
\item The energy inequality
\begin{align}\label{CEISt}
\left[ \int_\Omega \frac12 \rho |\vu|^2 + P(\rho) \dd x \right]_{t=0}^{t=\tau} + \int_0^\tau \int_\Omega \bS:\nabla \vu \dd x \dd t \leq \int_0^\tau \int_\Omega \rho \mathbf f \cdot \vu \dd x \dd t
\end{align}
holds for almost every $\tau \in [0,T]$, where $\mathbf f = -g\mathbf e_3$, $(\rho |\vu|^2)(0) := |\vc m_0|^2/\rho_0$, and the pressure potential $P$ is determined by
\begin{align}\label{PressPot}
\rho P'(\rho) - P(\rho) = p(\rho), \quad P''(\rho)=p'(\rho)/\rho.
\end{align}
\end{itemize}
\end{Definition}

From the definition of $P$ in \eqref{PressPot}, it follows immediately that $P$ essentially behaves as $p$; indeed, in the case where $p(\rho)=\rho^\gamma$, we have $P(\rho)=\rho^\gamma/(\gamma-1)$. Moreover, from the definition of $p(\rho)$, we see that there is an absolute constant $C_P(\gamma)>0$ such that
\begin{align*}
C_P^{-1} P(\rho) \leq \rho^\gamma \leq C_P P(\rho) \quad \forall \rho \geq 0.
\end{align*}
Regarding the force term, we use the same trick as before to rewrite $\mathbf e_3 = \nabla [x\mapsto x_3]$ and use the continuity equation to infer
\begin{align*}
\int_0^\tau \int_\Omega g\mathbf e_3 \cdot \rho \vu \dd x \dd t &= -\int_0^\tau \int_\Omega g x_3 \div(\rho \vu) \dd x \dd t = \int_0^\tau \int_\Omega g x_3 \del_t \rho \dd x \dd t\\
&= \int_\Omega g x_3 (\rho(\tau) - \rho(0)) \dd x \leq 2\|\rho\|_{L^\infty(0,T;L^\gamma(\Omega))} \|g x_3\|_{L^1(0,T;L^{\gamma'}(\Omega))}\\
&\leq \frac{1}{2C_P}\|\rho\|_{L^\infty(0,T;L^\gamma(\Omega))}^\gamma + C(\Omega, \gamma, g) \leq \frac12 \|P(\rho)\|_{L^\infty(0,T;L^1(\Omega))} +  C(\Omega, \gamma, g).
\end{align*}
Thus, using Gr\"onwall's inequality \eqref{Gronwall}, the energy inequality \eqref{CEISt} implies
\begin{align*}
\|\rho\|_{L^\infty(0,T;L^\gamma(\Omega))}^\gamma + \|\vu\|_{L^2(0,T;W_0^{1,2}(\Omega))}^2 + \|\rho |\vu|^2\|_{L^\infty(0,T;L^1(\Omega))} \leq C(\Omega, \mu, \eta, \gamma, g) (E_0 + 1),
\end{align*}
which is precisely inequality \eqref{UnifBds} for constant temperature and $p=2$.

\section{Incompressible inhomogeneous Navier-Stokes}
The model \eqref{CNSE} takes the full compressibility of the gas into account. If one wants to focus on deep water, this model is in some sense ``too accurate''. One can thus investigate an ``in-between-case'' of density dependent, but incompressible flow. The equations \eqref{CNSE} stay the same with the additional constraint that $\div \vu = 0$. Note that this forces the pressure $p$ to be, like in the purely incompressible case, the Lagrange multiplier for the incompressibility condition and hence, no constitutive relation for $p$ is needed. Definition~\ref{def:wkSol} of weak solutions modifies in the following way:
\begin{Definition}
Let $\rho_0 \in L^\infty(\Omega)$ with $\rho_0 \geq 0$, and $\vu_0 \in L^2(\Omega)$ with $\div \vu_0 = 0$ and $\bD(\vu_0)=0$ on $\ms$. We call a couple $(\rho, \vu)$ a finite energy weak solution of system \eqref{CNSE}, supplemented with $\div \vu=0$, if:
\begin{itemize}
\item The solution belongs to the regularity class
\begin{gather*}
\vu \in L^\infty(0,T; L^2(\Omega)) \cap L^2(0,T;W_0^{1,2}(\Omega)), \quad \rho \in L^\infty(0,T;L^\infty(\Omega)),\\
\bD(\vu) = 0 \text{ on } \ms, \quad \rho \geq 0 \text{ a.e.~in } (0,T) \times \Omega, \quad \rho = \rho_\ms \text{ on } \ms,\\
\|\rho\|_{L^p(0,T; L^p(\Omega))} = \|\rho_0\|_{L^p(\Omega)} \quad \text{for any } p \in [1,\infty),\\
\|\rho(t)\|_{L^\infty(\Omega)} = \|\rho_0\|_{L^\infty(\Omega)} \text{ for a.e. } t \in [0,T];
\end{gather*}
\item The weak formulation of the continuity equation holds:
\begin{align*}
\int_0^\tau \int_\Omega \rho (\del_t \psi + \vu \cdot \nabla \psi) \dd x \dd t = \int_\Omega \rho(\tau) \psi(\tau) \dd x - \int_\Omega \rho_0 \psi(0) \dd x
\end{align*}
for any $\psi \in C_c^\infty([0,T) \times \overline{\Omega})$;
\item The weak formulation of the momentum equation holds:
\begin{align*}
&\int_0^\tau \int_\Omega \rho \vu \cdot \del_t \phi + \rho \vu \otimes \vu : \nabla \phi - \bS:\nabla \phi + \rho \mathbf f \cdot \phi \dd x \dd t\\
&= \int_\Omega (\rho \vu)(\tau) \cdot \phi(\tau) \dd x - \int_\Omega \mathbf m_0 \cdot \phi(0) \dd x
\end{align*}
with $\mathbf f = -g \mathbf e_3$, for any $\phi \in C_c^\infty([0,T) \times \Omega)$ such that $\div \phi=0$ and $\bD(\phi)=0$ in a neighborhood of $\ms$;
\item The energy inequality
\begin{align}
\left[ \int_\Omega \frac12 \rho |\vu|^2 \dd x \right]_{t=0}^{t=\tau} + \int_0^\tau \int_\Omega \bS:\nabla \vu \dd x \dd t \leq \int_0^\tau \int_\Omega \rho \mathbf f \cdot \vu \dd x \dd t
\end{align}
holds for almost every $\tau \in [0,T]$, where $\mathbf f = -g\mathbf e_3$.
\end{itemize}
\end{Definition}
Existence of weak solutions to the above problem in 2D and no vacuum $\rho_0 \geq \underline{\rho} > 0$, even for several bodies, has been shown in \cite[Theorem~2.1]{SanmartinStarovoitovTucsnak2002}, where the authors also investigate collision (see Theorem~2.2 in there); note, however, that the boundary of the bodies is assumed to be of class $C^2$. For the speed of collision, if it happens at time $T_* \in [0,T]$, they derive for the height function $h(t) = \dist(\ms, \del \Omega)$ that
\begin{align*}
\lim_{t \to T_*} h(t) |t - T_*|^{-\eta} = 0, \qquad \eta=2.
\end{align*}
Indeed, the exponent $\eta=2$ here is not surprising: we remark that this leads to $\beta = \frac34$ in Theorem~\ref{thm:Starov} below. This fraction will occur later on again, see Remark~\ref{rem:beta34}, and has to do with the fact that solids of class $C^2$ roughly should behave like balls, i.e., the H\"older exponent is $\alpha=1$. For $\alpha \in [0,1)$, although an existence results seems to be missing in the literature, our collision result would then achieved in the same way as before, namely by noting that the force term is estimated as
\begin{align*}
\int_0^\tau \int_\Omega \rho \vc f \cdot \vu \dd x \dd t &= \int_\Omega g x_3 (\rho_0 - \rho(\tau)) \dd x \leq 2 \|g x_3\|_{L^1(\Omega)} \|\rho_0\|_{L^\infty(\Omega)},
\end{align*}
where we used that $\|\rho(\tau)\|_{L^\infty(\Omega)} = \|\rho_0\|_{L^\infty(\Omega)}$. Hence, we end up with
\begin{align*}
\|\rho\|_{L^\infty(0,T; L^\infty(\Omega))} + \|\vu\|_{L^2(0,T; W_0^{1,2}(\Omega))}^2 + \|\rho |\vu|^2\|_{L^\infty(0,T; L^1(\Omega))} \leq C(\Omega, \mu, \eta, g) (E_0 + 1),
\end{align*}
which is inequality \eqref{UnifBds} again upon replacing $P(\rho_0, \vartheta_0)$ by $\|\rho_0\|_{L^\infty(\Omega)}$ in \eqref{eq:E0}. Note especially that fits the computations of the foregoing section and thus yields Theorem~\ref{theo1} with $\gamma=\infty$.

\section{Compressible non-Newtonian Navier-Stokes-Fourier}
We close this chapter in putting together all the systems above to get a model for heat conducting, compressible, non-Newtonian fluids. The equations read
\begin{align*}
\begin{cases}
\del_t \rho + \div(\rho \vu) = 0 & \text{in } \mf(t),\\
\del_t(\rho \vu) + \div(\rho \vu \otimes \vu) - \div \bS + \nabla p(\rho, \vartheta) = -\rho g \mathbf e_3 & \text{in } \mf(t),\\
\del_t(\rho s) + \div(\rho s \vu) + \div \frac{\vc q}{\vartheta} = \sigma & \text{in } \mf(t).
\end{cases}
\end{align*}
Here, $s=s(\rho, \vartheta)$ is the specific entropy, which is related to the internal energy $e=e(\rho, \vartheta)$, the pressure $p=p(\rho, \vartheta)$, the density $\rho$, and the temperature $\vartheta$ through Gibbs' relation
\begin{align}\label{Gibbs}
\vartheta Ds = De + p D \bigg( \frac{1}{\rho} \bigg).
\end{align}
Further, the entropy production rate $\sigma$ fulfils
\begin{align*}
\sigma \geq \frac{1}{\vartheta} \bigg( \bS:\nabla\vu - \frac{\mathbf q \cdot \nabla \vartheta}{\vartheta} \bigg)
\end{align*}
in the sense of measures. 
Moreover, we will impose some constitutive relations, the precise motivation behind can be found in \cite[Section~1.4]{FeireislNovotny2009singlim} (for $\gamma = \frac{5}{3}$):
\begin{align}\label{constRel}
\begin{aligned}
p(\rho,\vartheta) &= p_{\rm m}(\rho, \vartheta) + p_{\rm rad}(\vartheta), & p_{\rm m}(\rho,\vartheta) &= \vartheta^\frac{\gamma}{\gamma-1} \mathfrak{P}\bigg( \frac{\rho}{\vartheta^\frac{1}{\gamma-1}} \bigg), & p_{\rm rad}(\vartheta) &= \frac{a}{3}\vartheta^4,\\
e(\rho,\vartheta) &= e_{\rm m}(\rho, \vartheta) + e_{\rm rad}(\vartheta), & e_{\rm m}(\rho,\vartheta) &= \frac{1}{\gamma-1} \frac{\vartheta^\frac{\gamma}{\gamma-1}}{\rho} \mathfrak{P}\bigg( \frac{\rho}{\vartheta^\frac{1}{\gamma-1}} \bigg), & e_{\rm rad}(\vartheta) &= \frac{a}{\rho}\vartheta^4,\\
s(\rho,\vartheta) &= s_{\rm m}(\rho, \vartheta) + s_{\rm rad}(\vartheta), & s_{\rm m}(\rho,\vartheta) &= \mathfrak{S}\bigg( \frac{\rho}{\vartheta^\frac{1}{\gamma-1}} \bigg), & s_{\rm rad}(\vartheta) &= \frac{4a}{3} \frac{\vartheta^3}{\rho},
\end{aligned}
\end{align}
where $a>0$ is the Stefan-Boltzmann constant, $\mathfrak{P}\in C([0,\infty)) \cap C^2((0,\infty))$ satisfies
\begin{align*}
\mathfrak{P}(0)=0, \quad \mathfrak{P}'(Z)>0 \ (Z>0), \quad 0<\frac{\gamma \mathfrak{P}(Z) - \mathfrak{P}'(Z)Z}{Z} \leq c \ (Z\geq 0),
\end{align*}
and
\begin{align*}
\mathfrak{S}'(Z)=-\frac{1}{\gamma-1} \frac{\gamma \mathfrak{P}(Z) - \mathfrak{P}'(Z)Z}{Z^2} < 0.
\end{align*}
It follows that the function $Z \mapsto \mathfrak{P}(Z)/Z^\gamma$ is decreasing, and we assume
\begin{align*}
\lim_{Z\to \infty} \frac{\mathfrak{P}(Z)}{Z^\gamma} = p_\infty>0.
\end{align*}
Note in particular that $p(\rho, \vartheta)$ complies with the growth assumption \eqref{p-law}.\\

Similar to the incompressible heat-conducting case, we have inequalities for the energy and entropy, finally resulting in uniform bounds
\begin{align*}
\|\vu\|_{L^p(0,T;W_0^{1,p}(\Omega))}^p + \|\rho\|_{L^\infty(0,T;L^\gamma(\Omega))}^\gamma + \|\rho |\vu|^2\|_{L^\infty(0,T;L^1(\Omega))} \lesssim E_0 + 1,
\end{align*}
where the implicit constant depends on the data. Again, these bounds enable us to conclude.



\chapter{Newtonian flow with temperature-growing viscosities}\label{ch:4}
In this short chapter, we investigate a different model for viscosity that does not fit into the assumptions \ref{S1}--\ref{S3}: viscosities that can grow to infinity as the temperature does. The version of what we present here was earlier given in \cite{NecasovaOschmann2024}. To fix the setting, let
\begin{align*}
\bS=2\mu(\vartheta) \bigg( \bD(\vu) - \frac13 \div \vu \Id \bigg) +\eta(\vartheta) \div \vu \Id,
\end{align*}
where the viscosity coefficients $\mu,\eta$ are assumed to be continuous functions on $(0,\infty)$, $\mu$ is moreover Lipschitz continuous, and they satisfy
\begin{align*}
1+\vartheta &\lesssim \mu(\vartheta), \quad |\mu'|\lesssim 1, && 0 \leq \eta(\vartheta) \lesssim 1+\vartheta.
\end{align*}
Note that this means we consider a Newtonian fluid with growing viscosities that are \emph{not} uniformly bounded in the temperature variable, thus not fulfilling \ref{S3}.

The equations governing the fluid's motion are now given by
\begin{align}\label{NSF}
\begin{cases}
\del_t \rho + \div(\rho \vu) = 0 & \text{in } \mf,\\
\del_t(\rho \vu) + \div(\rho \vu \otimes \vu) - \div \bS + \nabla p(\rho, \vartheta) = \rho \vc f & \text{in } \mf,\\
m \ddot{\vc G}(t) = -\int_{\del \ms} (\bS - p\Id) \vc n \dd \sigma + \int_\ms \rho_\ms \vc f \dd x & \text{in } \mf,\\
\frac{\rd}{\rd t}(\mathbb J \omega) = -\int_{\del \ms} (x-\vc G) \times (\bS - p\Id)\vc n \dd \sigma + \int_\ms (x-\vc G) \times \rho_\ms \vc f \dd x & \text{in } \mf,\\
\del_t(\rho s) + \div(\rho s \vu) + \div \frac{\vc q}{\vartheta} = \frac{1}{\vartheta} \big( \bS : \nabla \vu - \frac{\vc q \cdot \nabla \vartheta}{\vartheta} \big) & \text{in } \mf,\\
\vu = \dot{\vc G}(t) + \omega(t) \times (x - \vc G(t)) & \text{on } \del \ms,\\
\vu = 0 & \text{on } \del \Omega,\\
\vc q \cdot \vc n = 0 & \text{on } \del \Omega.
\end{cases}
\end{align}
Here, the pressure is given by a combination of adiabatic pressure law, Boyle-Mariott law, and radiation pressure arising from the Stefan-Boltzmann law as \[ p(\rho, \vartheta) = p_\infty \rho^\gamma + c_v (\gamma-1) \rho \vartheta + \frac{a}{3} \vartheta^4,\] the heat flow vector $\vc q = \vc q(\vartheta, \nabla \vartheta)$ is given by Fourier's law
\begin{align*}
\vc q(\vartheta, \nabla \vartheta) = - \kappa(\vartheta) \nabla \vartheta
\end{align*}
with the heat capacity coefficient satisfying
\begin{align*}
\kappa(\vartheta) \sim 1+\vartheta^\beta \ \text{for some} \ \beta>1,
\end{align*}
and the specific entropy $s=s(\rho, \vartheta)$ is connected to the pressure $p(\rho, \vartheta)$ and the internal energy $e(\rho, \vartheta)$ of the fluid through Gibbs' relation \eqref{Gibbs}. Note that this relation determines the internal energy and specific entropy as
\begin{align*}
e(\rho, \vartheta) = \frac{p_\infty}{\gamma-1} \rho^{\gamma-1} + c_v \vartheta + \frac{a}{\rho} \vartheta^4, &&
s(\rho, \vartheta) = \log \bigg( \frac{\vartheta}{\rho^{\gamma-1}} \bigg)^{c_v} + \frac{4a}{3} \frac{\vartheta^3}{\rho},
\end{align*}
where $c_v>0$ is the specific heat capacity at constant volume. In the language of \eqref{constRel}, this is equivalent to the choice
\begin{align*}
\mathfrak{P}(Z) = p_\infty Z^\gamma + c_v(\gamma-1) Z, && \mathfrak{S}(Z) = -c_v(\gamma-1) \log Z.
\end{align*}
Note that both $\mathfrak{P}$ and $\mathfrak{S}$ satisfy all the assumptions of the previous chapter.

Denoting now $\vartheta_s$ the solid's temperature, we extend the temperature similarly to the velocity and density as
\begin{align*}
\vartheta = \begin{cases}
\vartheta & \text{in } \mf,\\
\vartheta_s & \text{in } \ms,
\end{cases}
\end{align*}
and we consider the continuity of the heat flux $\vc q(\vartheta, \nabla\vartheta)\cdot \vc n = \vc q(\vartheta_s, \nabla \vartheta_s) \cdot \vc n$ on $\del \ms$. Moreover, for simplicity we assume that the heat capacity coefficient of the solid is the same as the fluid's one (this can be generalized, see \cite[Equation~(4.23)]{Brezina2008}).\\

Noticing that the existence proof presented in \cite{Brezina2008} also works for any $\beta>2$ instead of the considered $\beta=3$, in such case we have the uniform bound
\begin{align*}
\|\vartheta^\frac{\beta}{2}\|_{L^2(0,T;W^{1,2}(\Omega))}^2 \lesssim E_0+1,
\end{align*}
where here
\begin{align*}
E_0 = \int_{\mf(0)} \frac{|\vc m_0|^2}{2 \rho_0} + \rho_0 e(\rho_0, \vartheta_0) \dd x + \frac{m}{2} |\vc V_0|^2 + \mathbb{J}(0)\omega_0 \cdot \omega_0.
\end{align*}
Thanks to Sobolev embedding, this yields
\begin{align*}
\vartheta^\frac{\beta}{2} \in L^2(0,T;L^6(\Omega)), \quad \text{that is,} \quad \vartheta \in L^{\beta}(0,T;L^{3\beta}(\Omega)),
\end{align*}
in turn,
\begin{align*}
\|\vartheta\|_{L^\beta(0,T;L^{3\beta}(\Omega))}^\beta \lesssim E_0+1.
\end{align*}
Accordingly, the estimate for the stress tensor changes into
\begin{align*}
&\left| \int_0^T \zeta \int_\Omega \bS:\nabla \mathbf w_h \dd x \dd t \right| \lesssim \int_0^T \zeta \|\vartheta\|_{L^{3\beta}(\Omega)} \|\nabla \vu\|_{L^2(\Omega)} \|\nabla \mathbf w_h\|_{L^\frac{6\beta}{3\beta-2}(\Omega)} \dd t \\
&\leq \|\zeta\|_{L^\frac{2\beta}{\beta-2}(0,T)} \|\vartheta\|_{L^\beta(0,T;L^{3\beta}(\Omega))} \|\nabla \vu\|_{L^2((0,T)\times \Omega)} \|\nabla\mathbf w_h\|_{L^\infty(0,T;L^\frac{\beta}{3\beta-2}(\Omega))}\\
    &\lesssim (E_0+1)^{\frac{1}{\beta} + \frac12} T^{\frac12-\frac{1}{\beta}},
\end{align*}
provided
\begin{align*}
\frac{6\beta}{3\beta-2} < \frac{3+\alpha}{1+2\alpha} \Leftrightarrow \alpha < \frac{3(\beta-2)}{9\beta+2},
\end{align*}
while all the other estimates stay the same. Hence, repeating the arguments from Section~\ref{sec:PfThm}, we find that collision occurs provided
\begin{align*}
\gamma>3, \ \beta>2, \ \alpha < \bigg\{ \frac{3(\gamma-3)}{4\gamma+3}, \frac{3(\beta-2)}{9\beta+2} \bigg\}.
\end{align*}

As can be easily seen, the same arguments can be used for temperature-dependent non-Newtonian fluids, provided the stress tensor decomposes like
\begin{align*}
\bS(\vartheta, \mathbb{M}) = \mu(\vartheta) \tilde{\bS}(\mathbb{M}) + \eta(\vartheta) |\div \vu|^{p-2} \div \vu \Id
\end{align*}
for some tensor $\tilde{\bS}$ satisfying \ref{S1}--\ref{S3}, and $\mu, \eta$ are as above.

\begin{Remark}
As a matter of fact, all the analyses in this chapter also hold for the incompressible case, which (roughly speaking) corresponds to $\gamma=\infty$. Thus, collision for this type of heat conducting compressible fluids occurs if $\beta>2$ and $\alpha<\frac{3(\beta-2)}{9\beta+2}$. Also here, for constant temperature corresponding to a perfectly heat conducting fluid, we recover the borderline value $\alpha<\frac13$ in the limit $\beta\to\infty$, see Remark~\ref{rem1}.
\end{Remark}



\chapter{Incompressible fluids: A review} \label{ch:6}
\providecommand{\vv}{\mathbf v}
In contrast to the foregoing chapters dealing with a rather general class of fluids in three spatial dimensions, in this and the following chapters we want to recall the known collision results for incompressible Newtonian fluids, for a certain range of the shape value of $\alpha$, and also comparing to results in two space dimensions. To begin with, let us recall the incompressible Navier-Stokes equations for Newtonian fluids as
\begin{align}\label{inc1}
\begin{cases}
\div \vu = 0 & \text{in } \mf(t),\\
\rho_\mf (\del_t \vu + \div(\vu \otimes \vu) ) - \div \bS + \nabla p = \rho_\mf \vc f & \text{in } \mf(t),\\
\vu = 0 & \text{on } \del \Omega,\\
\vu(0)=\vu_0 & \text{in } \mf(0),
\end{cases}
\end{align}
complemented with the fluid-structure interaction terms on $\del \ms(t)$
\begin{align}\label{inc2}
\begin{cases}
m \ddot{\vc G} = -\int_{\del \ms} (\bS - p \Id) \vc n \dd \sigma + \int_\ms \rho_\ms \vc f \dd x,\\
\vu = \dot{\vc G}(t) + \omega(t) (x-\vc G(t))^\perp & \text{ if } d=2,\\
\frac{\rd}{\rd t} (J \omega) = -\int_{\del \ms} (x-\vc G)^\perp \cdot (\bS - p \Id)\vc n \dd \sigma + \int_\ms (x- \vc G)^\perp \cdot \rho_\ms \vc f \dd x & \text{ if } d=2,\\
\vu = \dot{\vc G}(t) + \omega(t) \times (x-\vc G(t)) & \text{ if } d=3,\\
\frac{\rd}{\rd t} (\mathbb{J} \omega) = -\int_{\del \ms} (x-\vc G) \times (\bS - p \Id)\vc n \dd \sigma + \int_\ms (x- \vc G) \times \rho_\ms \vc f \dd x & \text{ if } d=3,\\
\end{cases}
\end{align}
where
\begin{align*}
\bS = 2\mu \bD(\vu) = \mu (\nabla \vu + \nabla^T \vu), \quad \mu>0,
\end{align*}
$\rho_\mf, \rho_\ms>0$ are the (constant) fluid's and solid's density, respectively, and, for $d=2$, the moment of inertia $J=J(t)>0$ is a scalar function. Note that from the form of $\bS$, according to $\div(\nabla^T \vu) = \nabla \div \vu = 0$, we also see that we may write the term $\div \bS$ in a more common way as $\mu \Delta \vu$. For simplicity in notation, without loss of generality we can and will set $\mu=1$ in the sequel. Moreover, we assume that the driving force $\vc f$ is gravity, meaning $\vc f = -g \vc e_d = -g \nabla [x \mapsto x_d]$ for the gravitational constant $g>0$.

\section{Starovoitov's result}
In \cite{Starovoitov2003}, the author considered the incompressible Navier-Stokes equations with a solid of class $C^{1,\alpha}$. The aim of this section is to prove his main result on (no-)collision:
\begin{Theorem}\label{thm:Starov}
Let $\ms \subset \R^d$, $d \in \{2,3\}$, be a compact domain of class $C^{1,\alpha}$, $\alpha \in [0,1]$. If, for some $p \in [1,\infty]$, we have $\vu \in L^\infty(0,T;L^2(\Omega)) \cap L^2(0,T;W^{1,p}(\Omega))$, then the distance function $h(t)=\dist(\del \ms(t), \del \Omega)$ satisfies
\begin{align}\label{ineq1}
|\dot{h}(t)| \leq C h^\beta(t) \|\vu(t)\|_{W^{1,p}(\Omega)}
\end{align}
with $\beta = 2 - \frac{1}{1+\alpha} \big( 1 + \frac{d-1}{p} \big) - \frac1p = \frac{1+2\alpha}{p(1+\alpha)} (p-\frac{d+\alpha}{1+2\alpha})$.\\

Further, if additionally $\vu \in L^q(0,T;W^{1,p}(\Omega))$ for some $q \in [1,\infty]$, then:
\begin{enumerate}
\item[a)] If $h(T_\ast)=0$ for some $T_\ast \in [0,T]$ and $\beta < 1$, that is, $\alpha(p-1) < d$, then collision occurs with rate $\lim_{t \to T_\ast} h(t)|t-T_\ast|^{-\eta} = 0$ for $\eta = \frac{q-1}{q}\frac{1}{1-\beta} = \frac{q-1}{q} \frac{(1+\alpha)p}{d-\alpha (p-1)}$.
\item[b)] If $h(T_\ast)>0$ for some $T_\ast \in [0,T]$ and $\beta \geq 1$, that is, $\alpha(p-1) \geq d$, then $h(t)>0$ for all $t \in [0,T]$.
\end{enumerate}
\end{Theorem}

We shall compare the above result, especially the value of $\beta$, with the bounds of the test function $\vc w_h$ found in Lemma~\ref{BdsWh} (see also Remark~\ref{rem:fits}). Note also that in three dimensions and for $p=2$, the value for $\beta$ in \eqref{ineq1} is $\beta = \frac{3\alpha-1}{2(1+\alpha)}$, and also notice that $\beta<0$ precisely if $\alpha$ is less than our favourite fraction $\alpha < \frac13$. We will see this number again later on. Let us moreover emphasize that this theorem does \emph{not} prove that collision happens. It rather states that \emph{if} at some time $T_\ast$ the distance $h$ vanishes, then it does with zero speed and a certain rate determined by $\eta$. Indeed, as we will see in the next subsections, in dimension two the value for $\alpha$ to let collision happen is quite restricted, whereas in the three-dimensional case and $p=2$ we can allow for \emph{all} values $\alpha \in [0,1)$. Again, $\alpha=1$ is excluded since this corresponds to a ball-shaped object.\\

On the other hand, Starovoitov also showed that collision can happen by giving a precise example: he constructed a solution, the ``remainder'' of which when inserting in the Navier-Stokes equations gives an additional ``singular--in--$W^{-1,2}$'' force, showing that there exists a force such that the solid touches the container's bottom; we will come back to this in the next chapter. At this point, let us put a quote made by P.~Constantin: ``\emph{Be careful with statements that say `there exists a force'. Every function is a solution to Navier-Stokes: there exists a force.}''\footnote{P.~Constantin at \emph{Shocking Developments: New Directions in Compressible and Incompressible Flows: A Conference in Honor of Alexis Vasseur's 50th Birthday}, Leipzig, 26.06.--30.06.2023}

\begin{proof}[Proof of Theorem~\ref{thm:Starov}]
We just consider the case $d=3$. The case $d=2$ follows the same lines with obvious changes in the notations/definitions. To begin, since $\ms$ is rigid, we have $\vu|_{\del \ms} = \dot{\vc G}(t) + \omega(t) \times (x - \vc G(t))$. Then,
\begin{align*}
\int_\ms |\vu|^2 \dd x = \int_\ms |\dot{\vc G} + \omega \times (x-\vc G)|^2 \dd x = \int_\ms |\dot{\vc G}|^2 + |\omega \times (x - \vc G)|^2 \dd x,
\end{align*}
where we used that $\int_\ms \dot{\vc G }\cdot [\omega \times (x - \vc G)] \dd x = \dot{\vc G }\cdot [\omega \times \int_\ms x - \vc G \dd x ] = 0$ by the definition of $\vc G$ as the center of mass. This identity yields immediately a bound for $\dot{\vc G}$. To get a bound on $\omega$ itself, we take a small ball $B = B(\vc G) \subset \ms$ with midpoint $\vc G$ and write
\begin{align*}
\int_\ms |\omega \times (x-\vc G)|^2 \dd x \geq \int_B |\omega \times (x - \vc G)|^2 \dd x = \int_{B(0)} |\omega \times x|^2 \dd x.
\end{align*}
Note that the last integral now is rotationally symmetric \emph{for any rotation}. Thus, we choose to rotate the coordinate system such that $\omega = |\omega| \vc e_1$, yielding
\begin{align*}
\int_{B(0)} |\omega \times x|^2 \dd x = |\omega|^2 \int_{B(0)} |\vc e_1 \times x|^2 \dd x \gtrsim |\omega|^2.
\end{align*}
This finally enables us to conclude that
\begin{align*}
|\dot{\vc G}|^2 + |\omega|^2 \lesssim \int_\ms |\vu|^2 \dd x
\end{align*}
and since $\vu \in L^\infty(0,T;L^2(\Omega))$, we find that both $\dot{\vc G}$ and $\omega$ are bounded in $L^\infty(0,T)$, telling that the motion of $\ms$ and in particular $h$ is Lipschitz.\\

Having this in mind, let $P(t) \in \del \ms(t)$ and $Q \in \del \Omega$ be two points realizing the distance $h(t)=|P(t)-Q|$ (note that $Q$ might also depend on time in general, however, we can choose our coordinate system in such a way that $Q$ is fixed and even $Q=0$). Denote $x' = (x_1, x_2)$. Since $\del \ms$ and $\del \Omega$ are of class $C^{1,\alpha}$, there are constants $k,r>0$ such that the ``parabolic shells''
\begin{align*}
\pi^+(t) &:= \{x \in \R^3: |x'|<r, \ h(t)+k|x'|^{1+\alpha} < x_3 < h(t)+kr^{1+\alpha} \} \subset \ms(t),\\
\pi^-(t) &:= \{x \in \R^3: |x'|<r, \ -kr^{1+\alpha} < x_3 < -k|x'|^{1+\alpha} \} \subset \R^3 \setminus \Omega.
\end{align*}
Set further
\begin{align*}
\mathcal G_r := \{ x\in \R^3: |x'|<r, \ -k|x'|^{1+\alpha} < x_3 < h(t)+k|x'|^{1+\alpha} \}.
\end{align*}
Note that we can split $\del \mathcal G_r = \Gamma_r^+ \cup \Gamma_r^- \cup \Gamma_r^0$, where $\Gamma_r^\pm = \del \mathcal G_r \cap \del \pi^\pm$, and $\Gamma_r^0 =  \{x \in \R^3: |x'|=r, \ -kr^{1+\alpha} \leq x_3 \leq kr^{1+\alpha}\}$.

Extending the velocity by $\vu = 0$ in $\R^3 \setminus \Omega$, we see that by $\div \vu = 0$
\begin{align*}
\int_{\del \mathcal G_r} \vu \cdot \vc n \dd \sigma = \int_{\mathcal G_r} \div \vu \dd x = 0 = \int_{\Gamma_r^-} \vu \cdot \vc n \dd \sigma.
\end{align*}

Furthermore,
\begin{align*}
\int_{\Gamma_r^+} \vu \cdot \vc n \dd \sigma = \int_{\del \pi^+(t)} \vu \cdot \vc n \dd \sigma - \int_{A_k} \vu \cdot \vc n \dd \sigma,
\end{align*}
where $A_k = \{x \in \R^3: |x'| \leq r, \ x_3 = h(t)+kr^{1+\alpha} \}$ is the ``upper part'' of $\pi^+(t)$. Since $\vu = \dot {\vc G} + \omega \times (x-\vc G)$ is rigid on $\ms$, by Gau\ss' theorem, the integral over $\del \pi^+(t)$ vanishes. Moreover, by the same token, $\vu$ is the same at every hight, meaning we can change integration over $A_k$ by integration over $A_{k=0}$. Then, denoting $\vu_P$ and $\vc n_P$ the velocity and the outward normal at the point $P\in \del \ms$, respectively, we have $\vu |_{A_{k=0}} = \vu_P$ and $\vc n |_{A_{k=0}} = - \vc n_P$, leading to
\begin{align*}
\int_{\Gamma_r^+} \vu \cdot \vc n \dd \sigma = \int_{\{|x'|\leq r, \ x_3 = h(t)\}} \vu_P \cdot \vc n_P \dd \sigma = \pi r^2 (\vu_P \cdot \vc n_P).
\end{align*}
Altogether, the above calculations, together with $\Gamma_r^+ = \del \mathcal G_r \setminus ( \Gamma_r^0 \cup \Gamma_r^-)$ and $|\vu \cdot \vc n| \leq |\vu|$, imply
\begin{align*}
\pi r^2 |\vu_P \cdot \vc n_P| \leq \int_{\Gamma_r^0} |\vu| \dd \sigma.
\end{align*}
Integrating this last inequality with respect to $r$ from 0 to some $\rho \in (0,r)$, we find
\begin{align*}
\frac{\pi}{3}\rho^3 |\vu_P\cdot \vc n_P| \leq \int_{\mathcal G_\rho} |\vu| \dd x.
\end{align*}
We further estimate by the use of the Poincar\'e inequality \eqref{Poinc}
\begin{align*}
\int_{\mathcal G_\rho} |\vu| \dd x &\leq |\mathcal G_\rho|^{1-\frac1p} \|\vu\|_{L^p(\mathcal G_\rho)} \lesssim (h+2k\rho^{1+\alpha})|\mathcal G_\rho|^{1-\frac1p} \|\nabla \vu\|_{L^p(\mathcal G_\rho)}\\
&\leq (h+2k\rho^{1+\alpha})|\mathcal G_\rho|^{1-\frac1p} \|\nabla \vu\|_{L^p(\Omega)}.
\end{align*}
Seeing that $|\mathcal G_\rho| \lesssim \rho^2 (h+2k\rho^{1+\alpha})$ yields
\begin{align*}
|\vu_P\cdot\vc n_P| \lesssim \rho^{-1-\frac2p}(h+2k\rho^{1+\alpha})^{2-\frac1p} \|\vu\|_{W_0^{1,p}(\Omega)}.
\end{align*}
We may now take $\rho = c h^{1/(1+\alpha)}$, where $c>0$ is such that $\rho \in (0,r)$; in particular, the choice $c=r \diam(\Omega)^{-\frac{1}{1+\alpha}}$ is allowed. We finally get
\begin{align*}
|\vu_P\cdot\vc n_P|\lesssim h^\beta \|\vu\|_{W_0^{1,p}(\Omega)} \ \text{ with } \ \beta = 2-\frac{1}{1+\alpha} \left(1+\frac2p \right)-\frac1p.
\end{align*}

It remains to show that $|\vu_P \cdot \vc n_P|$ is an upper bound for the time derivative $\dot h$. We will indeed show that $\dot h = - \vu_P \cdot \vc n_P$, meaning the speed of the distance change is decreasing and happens in the normal direction of the solid's movement towards the container's bottom, which one might intuitively expect. To this end, define a function $\mathfrak{y}: [0,T] \times \del \Omega \to \R$ such that $y=x+\mathfrak{y}(t,x)\vc \nu_x \in \del \ms(t)$, where $\vc \nu_x$ is the internal normal on $\del \Omega$ at $x\in \del \Omega$. Note that $\mathfrak{y}$ is well defined on some neighborhood $U \subset \del \Omega$ of $Q$ by regularity of $\del \ms$ and $\del \Omega$. Moreover, we have
\begin{align*}
\vu_y = \dot y = \frac{\del y}{\del t} = \dot{\mathfrak{y}} \nu_x \ \Rightarrow \ \dot{\mathfrak{y}} = (\vu_y \cdot \vc n_y)(\nu_x \cdot \vc n_y)^{-1}
\end{align*}
for a.e.~$t\in [0,T]$ and any $x \in U$, where $\vc n_y$ is the outward normal at $y=x+\mathfrak{y}(t,x)\nu_x \in \del \ms$. We see that for $x=Q$ (and so $y=P$), we get $\nu_x \cdot \vc n_y = -1$ and hence
\begin{align*}
\dot{\mathfrak{y}}(t,Q) = - \vu_P \cdot \vc n_P.
\end{align*}
Seeing finally that by definition $h(t)=\mathfrak{y}(t,Q)$, we obtain \eqref{ineq1}.\\

To show the last assertions, we integrate \eqref{ineq1} from $s$ to $t$ to obtain
\begin{align}\label{ineq2}
H_\beta(s) - C \bigg| \int_s^t \|\vu(\tau)\|_{W^{1,p}(\Omega)} \dd \tau \bigg| \leq H_\beta(t) \leq H_\beta(s) + C \bigg| \int_s^t \|\vu(\tau)\|_{W^{1,p}(\Omega)} \dd \tau \bigg|,
\end{align}
where
\begin{align*}
H_\beta = \begin{cases}
\frac{1}{1-\beta} h^{1-\beta} & \text{if } \beta \neq 1,\\
\log h & \text{if } \beta=1.
\end{cases}
\end{align*}
Assertion a) now follows by setting $s=T_\ast$ for which $H_\beta(T_\ast)=\frac{1}{1-\beta}h^{1-\beta}(T_\ast)=0$, and applying H\"older's inequality \eqref{Holder}, together with $\|\vu(\tau)\|_{W^{1,p}(\Omega)} \in L^q(0,T)$, to get (for $T_\ast < t$)
\begin{align*}
h^{1-\beta}(t) \leq C \|\vu\|_{L^q(T_\ast,t;W^{1,p}(\Omega))} |t-T_\ast|^{1-\frac1q}.
\end{align*}
Seeing that $\|\vu\|_{L^q(T_\ast,t;W^{1,p}(\Omega))} \to 0$ as $t \to T_\ast$, we have $\lim_{t\to T_\ast} h(t)|t-T_\ast|^{-\eta} = 0$ as wished. The case $T_\ast>t$ follows the same lines.\\

The second statement b) follows similarly. Indeed, if $\beta=1$, then \eqref{ineq2} forces
\begin{align*}
h(t) \geq h(T_\ast) \exp\bigg(- C \bigg| \int_{T_\ast}^t \|\vu(\tau)\|_{W^{1,p}(\Omega)} \dd \tau \bigg| \bigg)
\end{align*}
for all $t \in [0,T]$ and hence $h(t)>0$ for any $t \in [0,T]$. For the case $\beta>1$, we calculate
\begin{align*}
h^{1-\beta}(t) \leq h^{1-\beta}(T_\ast) + C\|\vu\|_{L^q(0,T;W^{1,p}(\Omega))} |t-T_\ast|^{1-\frac1q}.
\end{align*}
Since the right-hand side of this inequality is finite for any $t\in [0,T]$, this together with $\beta>1$ means that $h$ can never vanish, thus showing the result.
\end{proof}

\section{Finer estimates and a wider class of obstacles}\label{sec:62}
After determining the rate of collision, if it happens, let us answer the question \emph{whether} there is a configuration such that the solid collides with its container. The (no-)collision results we present here were previously obtained for the linear Stokes equations in \cite{GerardVaretHillairet2012} and for the case $d=3$, and in \cite{GerardVaretHillairet2010} for $d=2$ (see also \cite{Hillairet2007} for the case of a ball-shaped obstacle). Therefore, we shall also concentrate on this case here.

\subsection{Preliminaries}

All the following techniques are similar to the proof of Theorem~\ref{theo1}, in the sense that we have to construct an appropriate test function $\vc w_h$. For this reason, we give here some new estimates, and apply them later in the special cases.\\

Indeed, the same arguments/heuristics made in Section~\ref{sec:31} yield that an appropriate test function for $d=2$ is given by $\vc w_h = \nabla^\perp \phi_h$, where $\phi_h$ in the region beneath $\ms$ is now given by
\begin{align*}
\phi_h(x_1, x_2) &= x_1 \Phi \bigg(\frac{x_2}{\psi_h(x_1)} \bigg), \quad \Phi(t) = t^2(3-2t),\\
\psi_h(x_1) &= h + |x_1|^{1+\alpha} \ \text{ or } \ \psi_h(x_1) = 1+h-\sqrt{1-x_1^2}.
\end{align*}
Note especially that in 2D, we have $r=|x_1|$, so $\phi_h$ has the same structure as in 3D. Extending $\phi_h$ in a proper way to the whole of $\Omega$ similar as before, we have a test function $\vc w_h$ as wished. More precisely, with the notations as in \eqref{chi} (and obvious adaptations for $d=2$), we set
\begin{align}\label{phih2}
\phi_h(x_1,x_2)= x_1 \begin{cases}
    1 & \text{on } \ms_h,\\
    (1-\chi(x_1,x_2))\eta(x_1,x_2-h+h(0)) + \chi(x_1,x_2)\Phi\left(\frac{x_2}{\psi_h(x_1)}\right) & \text{on } \Omega \setminus \ms_h,
\end{cases}
\end{align}
and $\mathbf w_h = \nabla^\perp \phi_h$.\\

To start, we generalize estimate \eqref{Lem:est} to all values of $(\alpha, q, s) \in (0, \infty)^3$:

\begin{Lemma}\label{lem:i1}
The integral
\begin{align*}
\int_{-r_0}^{r_0} \frac{r^q}{(h+r^{1+\alpha})^s} \dd r
\end{align*}
behaves like
\begin{itemize}
\item[i)] $h^{\frac{q+1}{1+\alpha} - s}$, if $q+1<s(1+\alpha)$;
\item[ii)] $\log h$, if $q+1=s(1+\alpha)$;
\item[iii)] $1$, if $q+1>s(1+\alpha)$.
\end{itemize}
\end{Lemma}
\begin{proof}
This is straightforward calculation and we leave the details to the reader as an exercise.
\end{proof}

To lean notation, we will make the following agreement: for $d=2$, we set $\vc w_h = \nabla^\perp \phi_h$, where $\phi_h$ is as in \eqref{phih2}. If $d=3$, we set $\vc w_h = \nabla \times (\phi_h \vc e_\theta)$, where this time $\phi_h$ is as in \eqref{phih}. With this convention, we have:
\begin{Lemma}\label{lem:intEst}
For all $\alpha \in [0,1]$, the function $\vc w_h$ satisfies:
\begin{align*}
\|\vc w_h\|_{L^2(\Omega)} &\lesssim 1,\\
1 \lesssim h^\frac{3\alpha - (d-2)}{2(1+\alpha)} \|\nabla \vc w_h\|_{L^2(\Omega)} &\lesssim 1,\\
\|\nabla \vc w_h\|_{L^\infty(\Omega \setminus \Omega_{h, r_0})} &\lesssim 1.
\end{align*}
Moreover, for $r = |x_1|$ if $d=2$, and $r = \sqrt{x_1^2 + x_2^2}$ if $d=3$,
\begin{align*}
\sup_{r < r_0} |\psi_h(r)|^\frac32 \left( \int_0^{\psi_h(r)} |\nabla \vc w_h(r, x_d)|^2 \dd x_d \right)^\frac12 &\lesssim 1,\\
\int_{-r_0}^{r_0} \int_0^{\psi_h(r)} \psi_h^2(r) |\del_h \vc w_h|^2 \dd x &\lesssim 1.
\end{align*}
\end{Lemma}
\begin{proof}
The proof is follows the same lines as the one for Lemma~\ref{BdsWh} and using Lemma~\ref{lem:i1}, once seen that for the ball case $\frac{1}{2} r^2 \leq 1-\sqrt{1-r^2} \leq r^2$ for $r \leq 1$, that $\dd x = \dd x_1 \dd x_2$ if $d=2$, and $\dd x = r \dd r \dd \theta \dd x_3$ if $d=3$. In the latter case, we even can replace the integral over $\theta$ by the supremum over $\theta \in (0, 2\pi)$.
\end{proof}

\begin{Remark}
Compared to the calculations done in Remark~\ref{rem3}, Lemma~\ref{lem:intEst} gives the correct behavior of $\nabla \vc w_h$ in $L^p(\Omega)$ for $p=2$ instead of $1<p<2$.
\end{Remark}

Let $T_\ast \in (0, \infty]$ be the maximal existence time of the solution $(\rho, \vu)$ to \eqref{inc1}--\eqref{inc2}. To show the (no-)collision result, we would like to test the momentum equation by $\vc w_h$ and integrate by parts. As already noticed, the drag $\mathcal{D}_h$ is the main driving force to establish collision, or to prevent from it. In particular, we would need to shift the Laplace from the fluid's velocity $\vu$ to the test function $\vc w_h$; however, this causes quite strong singularities. To catch them, we will introduce a pressure $q_h$, which we then can insert in the weak formulation of the momentum equation without changes. Note that here the incompressibility condition of the fluid comes into play; without it, we would get an additional term $\int_\Omega q_h \div \vu \dd x$, which we do not know how to handle.

\begin{Lemma}\label{lem:singRemove}
There exists $q_h \in C^\infty(0,T_\ast; C(\overline{\Omega}))$ such that for any $\varphi \in W_0^{1,2}(\Omega)$ with $\varphi |_{\del \ms} = \vc e_d$, we have
\begin{align}\label{int1}
\int_{\mf(t)} |(\Delta \vc w_h - \nabla q_h) \cdot \varphi | \dd x \lesssim \|\varphi\|_{W_0^{1,2}(\Omega)}.
\end{align}
\end{Lemma}
\begin{proof}
\textbf{The case $d=2$.} Recalling $\vc w_h = \nabla^\perp \phi_h$, we set
\begin{align*}
q_h = \del_{12}^2 \phi_h - \int_0^{x_1} \del_{222}^3 \phi_h(t, x_2) \dd t.
\end{align*}
Although this definition seems quite artificial at first glance, it is easily obtained by computing $\Delta \vc w_h$ and ``removing'' all terms that contain a singular part in $h$; particularly, we want to have as less derivatives in $x_2$-direction as possible. Indeed, we calculate
\begin{align*}
(-\Delta \vc w_h)_1 &= \Delta \del_2 \phi_h = \del_{112}^3 \phi_h + \del_{222}^3 \phi_h,\\
(\Delta \vc w_h)_2 &= \Delta \del_1 \phi_h = \del_{111}^3 \phi_h + \del_{122}^3 \phi_h.
\end{align*}
Since $\phi_h$ and in turn $q_h$ are smooth outside $\Omega_{h, r_0}$, we just have to focus on this inner part. Note that inside $\Omega_{h, r_0}$, we have
\begin{align*}
\phi_h(x_1, x_2) = x_1 \Phi \left( \frac{x_2}{\psi_h(x_1)} \right), \ \Phi(t) = t^2(3 - 2t).
\end{align*}
Hence,
\begin{align*}
q_h = \del_{12}^2 \phi_h + 12 \int_0^{x_1} \frac{t}{\psi_h(t)^3} \dd t \ \text{in} \ \Omega_{h, r_0},
\end{align*}
and thus $\del_1 q_h = \del_{112}^3 \phi_h - \del_{222}^3 \phi_h$ and $\del_2 q_h = \del_{122}^3 \phi_h$, which precisely cancel the terms of $x_2$-order two and three occurring in $(\Delta \vc w_h)_1$ and $(\Delta \vc w_h)_2$. Thus, we see
\begin{align*}
\Delta \vc w_h - \nabla q_h = \begin{pmatrix}
-2 \del_{112}^3 \phi_h\\
\del_{111}^3 \phi_h
\end{pmatrix} \ \text{in} \ \Omega_{h, r_0}.
\end{align*}

From the definition of $\vc w_h$, we have $\vc w_h \in C^\infty(\Omega \setminus \Omega_{h, r_0})$. Moreover, we see that the most singular term in $\Delta \vc w_h - \nabla q_h$ is $\del_{112}^3 \phi_h$; indeed, we have for $\psi_h(x_1) = h + |x_1|^{1+\alpha}$
\begin{align}\label{derPhiAlpha}
|\del_{112}^3 \phi_h| \lesssim x_2 \frac{\psi_h'}{\psi_h^3} + x_2^2 \frac{\psi_h'}{\psi_h^4}, && |\del_{111}^3 \phi_h| \lesssim x_2^2 \frac{\psi_h''}{\psi_h^3} + x_2^3 \frac{\psi_h''}{\psi_h^4}
\end{align}
and thus, by $\psi_h'(x_1) \sim x_1^\alpha$,
\begin{align}\label{int12}
\begin{split}
\int_{\Omega_{h, r_0}} |\del_{112}^3 \phi_h|^p \dd x &\lesssim \int_{-r_0}^{r_0} \int_0^{\psi_h(x_1)} x_2^p \frac{|\psi_h'|^p}{\psi_h^{3p}} + x_2^{2p} \frac{|\psi_h'|^p}{\psi_h^{4p}} \dd x_2 \dd x_1 \lesssim \int_0^{r_0} \frac{x_1^{\alpha p}}{(h+x_1^{1+\alpha})^{2p-1}},\\
\int_{\Omega_{h, r_0}} |\del_{111}^3 \phi_h|^p \dd x &\lesssim \int_{-r_0}^{r_0} \int_0^{\psi_h(x_1)} x_2^{2p} \frac{|\psi_h''|^p}{\psi_h^{3p}} + x_2^{3p} \frac{|\psi_h''|^p}{\psi_h^{4p}} \dd x_2 \dd x_1 \lesssim \int_0^{r_0} \frac{x_1^{(\alpha-1) p}}{(h+x_1^{1+\alpha})^{p-1}}.
\end{split}
\end{align}

In the case $\psi_h(x_1) = 1+h-\sqrt{1-x_1^2}$, we first check that for any $k \in \mathbb N$, we have
\begin{align*}
|\psi_h^{(2k+1)}(x_1)| \lesssim |x_1|, \qquad |\psi_h^{(2k)}(x_1)| \lesssim 1 \qquad \forall |x_1| \leq r_0 < 1.
\end{align*}
Consequently, the estimates \eqref{derPhiAlpha} have to be replaced by
\begin{align*}
\begin{aligned}
|\del_{112}^3 \phi_h| &\lesssim x_2 \bigg( \frac{|x_1|}{\psi_h^3} + \frac{|x_1|^3}{\psi_h^4} \bigg) + x_2^2 \bigg( \frac{|x_1|}{\psi_h^4} + \frac{|x_1|^3}{\psi_h^5} \bigg) \\
|\del_{111}^3 \phi_h| &\lesssim x_2^2 \bigg( \frac{1}{\psi^3} + \frac{x_1^2}{\psi^4} + \frac{x_1^4}{\psi^5} \bigg) + x_2^3 \bigg( \frac{1}{\psi^4} + \frac{x_1^2}{\psi^5} + \frac{x_1^4}{\psi^6} \bigg).
\end{aligned}
\end{align*}
These bounds precisely correspond to $\alpha=1$, hence we do not have to distinguish between $\alpha=1$ and $\alpha<1$ in the sequel.\\

For the exponents occurring in \eqref{int12}, we have
\begin{align*}
\alpha p + 1 &< (2p-1)(1+\alpha) \Longleftrightarrow p > 1,\\
(\alpha-1)p + 1 &> (p-1)(1+\alpha) \Longleftrightarrow p < \frac{2+\alpha}{2}.
\end{align*}
Hence, using Lemma~\ref{lem:i1}, we see that the second integral in \eqref{int12} is uniformly bounded in $h$ for all $p<(2+\alpha)/2$, whereas the first one is always unbounded. Nonetheless, it is now obvious that $\Delta \vc w_h - \nabla q_h \in L^p(\Omega)$ for any $p>1$ (\emph{without} a uniform bound in $h$), thus for $h>0$ the integral in \eqref{int1} is well defined since by Sobolev embedding \eqref{SobEmb}, $\varphi \in W_0^{1,2}(\Omega) \hookrightarrow L^q(\Omega)$ for any $1 \leq q < \infty$.

To prove the desired inequality \eqref{int1}, by $\vc w_h, q_h \in C^\infty(\Omega \setminus \Omega_{h, r_0})$, we might without loss of generality assume that $\supp \varphi \subset \overline{\Omega}_{h, r_0}$ such that
\begin{align*}
\int_{\mf(t)} (\Delta \vc w_h - \nabla q_h) \cdot \varphi \dd x = \int_{\Omega_{h, r_0}} (\Delta \vc w_h - \nabla q_h) \cdot \varphi \dd x.
\end{align*}
Since $\varphi |_{\del \ms_h} = (0,1)^T$, we integrate by parts to obtain
\begin{align*}
\int_{\Omega_{h, r_0}} (\Delta \vc w_h - \nabla q_h) \cdot \varphi \dd x = -\int_{\del \ms_h \cap \del \Omega_{h, r_0}} \del_{11}^2 \phi_h \vc n_1 \dd \sigma -\int_{\Omega_{h, r_0}} \del_{11}^2 \phi_h (2 \del_2 \varphi_1 - \del_1 \varphi_2) \dd x,
\end{align*}
where we calculate
\begin{align*}
\del_{11}^2 \phi_h = - \Phi'' \left( \frac{x_2}{\psi_h} \right) \frac{x_1 x_2^2 (\psi_h')^2}{\psi_h^4} -\Phi'\left( \frac{x_2}{\psi_h} \right) \left( 2 \frac{x_2 \psi_h'}{\psi_h^2} + \frac{x_1 x_2 \psi_h''}{\psi_h^2} - 2 \frac{x_1 x_2 (\psi_h')^2}{\psi_h^3} \right).
\end{align*}
By Lemma~\ref{lem:i1}, we can check that $\del_{11}^2 \phi_h$ is bounded in $L^2(\Omega_{h, r_0})$ uniformly in $h$. Furthermore, by definition of $\Phi$, we have $\Phi'(1)=0$ and $\Phi''(1)=-6$ such that, together with $x_2 = \psi_h$ on $\del \ms_h$,
\begin{align*}
\del_{11}^2 \phi_h = \frac{6 x_1 x_2^2 (\psi_h')^2}{\psi_h^2} = \frac{6 x_1 (\psi_h')^2}{\psi_h^2}.
\end{align*}
Lastly, as $\del \ms_h \cap \del \Omega_{h, r_0}$ is parametrized by the curve $\koppa: (-r_0, r_0) \ni x_1 \mapsto (x_1, \psi_h(x_1))^T$, we find
\begin{align*}
\vc n = \frac{\koppa'(x_1)^\perp}{|\koppa'(x_1)|} = \frac{1}{\sqrt{1 + (\psi_h')^2}} (- \psi_h', 1)^T,
\end{align*}
hence
\begin{align*}
\left| \int_{\del \ms_h} \del_{11}^2 \phi_h \vc n_1 \dd \sigma \right| \leq \int_{-r_0}^{r_0} \left| \frac{6 x_1 (\psi_h'(x_1))^2}{(\psi_h(x_1))^2} \ \frac{\psi_h'(x_1)}{\sqrt{1+(\psi_h'(x_1))^2}} \right| \dd x_1.
\end{align*}
Note that the last integral hides a factor $\|\varphi_2\|_{L^\infty(\del \Omega_{h, r_0})}$ by $\varphi|_{\del \ms_h} = (0,1)^T$ and thus $1 = \|\varphi_2\|_{L^\infty(\del \Omega_{h, r_0})}$. Moreover, by standard trace inequality, $\|\varphi_2\|_{L^\infty(\del \Omega_{h, r_0})} \lesssim \|\varphi \|_{W_0^{1,2}(\Omega)}$. Hence, for $\psi_h(x_1) = h + |x_1|^{1+\alpha}$, we have $\psi_h'(x_1) \sim x_1^\alpha$ and this integral may be estimated again uniformly in $h$ with the help of Lemma~\ref{lem:i1}. The same conclusion holds for $\psi_h(x_1) = 1 + h - \sqrt{1 - x_1^2}$ as $\psi_h(x_1) \sim x_1^2$ for $|x_1| \leq r_0 < 1$. Consequently, we arrive at
\begin{align*}
\left| \int_{\Omega_{h, r_0}} (\Delta \vc w_h - \nabla q_h) \cdot \varphi \dd x \right| \lesssim \|\varphi_2\|_{L^\infty(\del \Omega_{h, r_0})} + \|\nabla \varphi\|_{L^2(\Omega_{h, r_0})} \lesssim \|\varphi\|_{W_0^{1,2}(\Omega)}.
\end{align*}

\textbf{The case $d=3$.} Very similar, we proceed for the three-dimensional case. Recall that $\vc w_h = \nabla \times (\phi_h \vc e_\theta)$. Calculating as before $\Delta \vc w_h$, where in 3D and cylindrical coordinates
\begin{align*}
\nabla = \vc e_r \del_r + r^{-1} \vc e_\theta \del_\theta + \vc e_3 \del_3, && \Delta = r^{-1} \del_r (r \del_r) + r^{-2} \del_{\theta \theta}^2 + \del_{33}^2,
\end{align*}
we find that we may simply choose $q_h$ as
\begin{align*}
q_h(r, x_3) = \del_{r 3}^2 \phi_h(r,x_3) - \int_0^r \del_{333}^3 \phi_h(t, x_3) \dd t.
\end{align*}
Similar calculations as for the 2D case show the result. We leave the details to the reader (see, for instance, \cite[Section~3.4]{HillairetTakahashi2009}).
\end{proof}

Before coming to the proof of Theorem~\ref{thm:ParInc}, we need the following proposition to control the drag force and the remainder arising from the weak formulation of the momentum equation.

\begin{Proposition}\label{prop:Remainder}
Let $\vc w_h, q_h$ be as above and $(\rho, \vu)$ be the solution to equations \eqref{inc1}--\eqref{inc2}. Define further
\begin{align*}
n(h) &:= \int_{\del \mf_h} (\nabla \vc w_h - q_h \Id)\vc n \cdot \vc e_d \dd \sigma,\\
R(t) &:= \int_0^t \int_\Omega \rho \vu \cdot \del_t \vc w_h + \rho \vu \otimes \vu : \bD(\vc w_h) \dd x \dd s \\
&\quad + \int_\Omega \vc m_0 \cdot \vc w_{h(0)} \dd x - \int_\Omega \rho(t) \vu(t) \cdot \vc w_{h(t)} \dd x\\
&\quad + \int_0^t \int_{\mf(s)} (\Delta \vc w_{h(s)} - \nabla q_{h(s)}) \cdot \vu(s) \dd x \dd s.
\end{align*}
Then, for $h>0$ small enough, all $\alpha \in [0,1]$, and all $0 \leq t < T_\ast$,
\begin{align*}
1 \lesssim h^\beta n(h) &\lesssim 1, \quad \beta = \frac{3\alpha-(d-2)}{1+\alpha},\\
|R(t)| &\lesssim 1+ \sqrt{t}.
\end{align*}
\end{Proposition}
\begin{proof}
We start with the estimate for $n(h)$. Using $\div \vc w_h = 0$, integration by parts gives
\begin{align*}
n(h) = 2 \int_\Omega |\bD(\vc w_h)|^2 \dd x + \int_{\mf_h} (\Delta \vc w_h - \nabla q_h) \cdot \vc w_h \dd x.
\end{align*}
By Lemma~\ref{lem:intEst}, we have $\|\nabla \vc w_h\|_{L^2(\Omega)}^2 \sim h^{(d-2-3\alpha)/(1+\alpha)}$. Combining this with Lemma~\ref{lem:singRemove} yields
\begin{align*}
\bigg| \int_{\mf_h} (\Delta \vc w_h - \nabla q_h) \cdot \vc w_h \dd x \bigg| \lesssim \|\nabla \vc w_h\|_{L^2(\Omega)} \lesssim 1 + \delta \|\nabla \vc w_h\|_{L^2(\Omega)}^2 \lesssim 1 + \delta h^\frac{d-2-3\alpha}{1+\alpha}
\end{align*}
for $\delta>0$ small enough (see \eqref{Young}). In particular,
\begin{align*}
n(h) \lesssim \|\nabla \vc w_h\|_{L^2(\Omega)}^2 + \bigg| \int_{\mf_h} (\Delta \vc w_h - \nabla q_h) \cdot \vc w_h \dd x \bigg| \lesssim 1 + h^\frac{d-2-3\alpha}{1+\alpha} \lesssim h^{-\beta}
\end{align*}
as wished. Additionally, for $h<1$,
\begin{align*}
n(h) \gtrsim \|\nabla \vc w_h\|_{L^2(\Omega)}^2 - \bigg| \int_{\mf_h} (\Delta \vc w_h - \nabla q_h) \cdot \vc w_h \dd x \bigg| \gtrsim (1-\delta) h^\frac{d-2-3\alpha}{1+\alpha} - 1 \gtrsim h^\frac{d-2-3\alpha}{1+\alpha}.
\end{align*}

To estimate $R(t)$, we control each term separately. Recalling again Lemma~\ref{lem:singRemove}, integrating this inequality with respect to $t$, and using H\"older's inequality \eqref{Holder}, we get
\begin{align*}
\int_0^t \bigg| \int_{\mf(s)} (\Delta \vc w_{h(s)} - \nabla q_{h(s)}) \cdot \vu(s) \dd x \bigg| \dd s \lesssim \|\vu\|_{L^2(0,t; W^{1,2}(\Omega))} \sqrt{t} \leq C(\|\vu_0\|_{L^2(\Omega)}) \sqrt{t},
\end{align*}
the last estimate coming from basic energy estimates for the velocity $\vu$, and the fact that the gravitational force $g \vc e_2$ is conservative (compare this with \eqref{consEner}); thus, the constant does not depend on the time $t$. Similarly,
\begin{align*}
\bigg| \int_\Omega \vc m_0 \cdot \vc w_{h(0)} \dd x - \int_\Omega \rho(t) \vu(t) \cdot \vc w_{h(t)} \dd x \bigg| \lesssim \|\vu\|_{L^\infty(0,t; L^2(\Omega))} \|\vc w_{h(\cdot)} \|_{L^\infty(0,t; L^2(\Omega))} \leq C(\|\vu_0\|_{L^2(\Omega)}).
\end{align*}

Before we come to the term involving $\del_t \vc w_h$, we state the following general result: for any $(r, \vc v) \in L^\infty(\Omega) \times W_0^{1,2}(\Omega)$ and any $\vc \phi \in W_0^{1,2}(\Omega)$,
\begin{align*}
\bigg| \int_\Omega r \vc v \cdot \phi \dd x \bigg| \lesssim \|r\|_{L^\infty(\Omega)} \|\nabla \vc v\|_{L^2(\Omega)} \bigg( \|\phi\|_{L^2(\Omega \setminus \Omega_{h, r_0})} + \bigg( \int_{-r_0}^{r_0} \int_0^{\psi_h(x_1)} |\psi_h(x_1)|^2 |\phi(x)|^2 \dd x \bigg)^\frac12 \bigg).
\end{align*}
The proof of this inequality follows easily from splitting $\Omega = (\Omega \setminus \Omega_{h, r_0}) \cup \Omega_{h, r_0}$, using H\"older's and Poincar\'e's inequality \eqref{Holder} and \eqref{Poinc} for the first part, and Hardy's inequality \eqref{Hardy} for the second. We leave the details to the reader.\\

Using the above inequality for $\phi = \del_t \vc w_h = \dot{h} \del_h \vc w_h$, combined with Lemma~\ref{lem:intEst} and the fact that $\del_h \vc w_h$ is bounded outside the region $\Omega_{h, r_0}$ by smoothness of $\vc w_h$ there, we arrive at
\begin{align*}
&\bigg| \int_0^t \int_\Omega \rho \vc u \cdot \del_t \vc w_h \dd x \dd t \bigg| \lesssim \|\dot h\|_{L^\infty ([0,T_\ast))} \int_0^t \|\nabla \vu(s)\|_{L^2(\Omega)} \bigg( \|\del_h \vc w_{h(s)}\|_{L^2(\Omega \setminus \Omega_{h(s), r_0})} \\
&\quad + \bigg( \int_{-r_0}^{r_0} \int_0^{\psi_h(x_1)} |\psi_h(x_1)|^2 |\del_h \vc w_{h(s)}(x)|^2 \dd x \bigg)^\frac12 \bigg) \dd s \leq C(\|\vu_0\|_{L^2(\Omega)}) \sqrt{t}.
\end{align*}
To finally deal with the convective term, we use similarly as before a general result: for any $(r, \vc v) \in L^\infty(\Omega) \times W_0^{1,2}(\Omega)$ and any $\phi \in W_0^{1,2}(\Omega)$, we have
\begin{align}\label{i1}
\begin{split}
\bigg| \int_\Omega r \vc v \otimes \vc v : \bD(\phi) \dd x \bigg| &\lesssim \|r\|_{L^\infty(\Omega)} \|\nabla \vc v\|_{L^2(\Omega)}^2 \bigg( \|\bD(\phi)\|_{L^\infty(\Omega \setminus \Omega_{h, r_0})} \\
&\quad + \sup_{x_1 \in (-r_0, r_0)} |\psi_h(x_1)|^\frac32 \bigg( \int_0^{\psi_h(x_1)} |\nabla \phi(x)|^2 \dd x_1 \bigg)^\frac12 \bigg).
\end{split}
\end{align}
As before, the first part of this inequality follows from H\"older's and Poincar\'e's inequality \eqref{Holder} and \eqref{Poinc}. The second part uses a refined Poincar\'e inequality
\begin{align*}
\bigg( \int_0^l |\vc v|^p \dd x \bigg)^\frac1p \lesssim l^{\frac12 + \frac1p} \bigg( \int_0^l |\del_x \vc v|^2 \dd x \bigg)^\frac12 \ \forall p \geq 2 \ \forall \vc v \in W^{1,2}(0,l), \ \vc v(0)=0
\end{align*}
that can be obtained via a scaling argument. We use this refined inequality for $l=\psi_h(x_1)$ and $p=4$ to conclude \eqref{i1}. Finally, we proceed as before, using again Lemma~\ref{lem:intEst} to get
\begin{align*}
\bigg| \int_\Omega \rho \vu \otimes \vu : \bD(\vc w_h) \dd x \bigg| \leq C(\|\vu_0\|_{L^2(\Omega)}).
\end{align*}
Putting together the bounds obtained, we finish the proof of the proposition.
\end{proof}

\subsection{The (no-)collision results}

We will prove the (no-)collision results in a general setting, meaning for both two and three dimensions, as well as for both parabolic and ball-shaped solids. Our main theorem reads:
\begin{Theorem}\label{thm:ParInc}
Let $d \in \{2,3\}$, and $\ms \subset \R^d$ be of class $C^{1,\alpha}$ for some $\alpha \in [0,1]$ such that $\dist(\ms(0), \del \Omega)>0$. Moreover, assume that the solid's density $\rho_\ms$ is larger than the fluid's one $\rho_\mf$, meaning $\rho_\ms > \rho_\mf$.
\begin{enumerate}
\item If $\alpha < \frac{d-1}{2}$, then the solid collides with $\del \Omega$ in finite time.
\item If $\alpha \geq \frac{d-1}{2}$, then the solid stays away from $ \del \Omega$ for all times.
\end{enumerate}
\end{Theorem}
The proof of the above theorem we present here is the one given in \cite{Hillairet2007} for the case $d=2$, but the same ideas also work for $d=3$, see \cite{HillairetTakahashi2009}, where a ball-shaped solid is considered, and \cite{GerardVaretHillairet2012} for an energy-based consideration for parabolic shapes. Moreover, combined with Starovoitov's results, we can state the following
\begin{Corollary}
Under the assumptions of Theorem~\ref{thm:ParInc}, for $2\alpha < d-1$ such that collision happens, the solid collides with rate
\begin{align*}
\lim_{t \to T_\ast} h(t) |t-T_\ast|^{-\eta} = 0, \qquad \eta = \frac{1+\alpha}{d-\alpha}.
\end{align*}
\end{Corollary}

\begin{Remark} In contrast to Theorem~\ref{thm:ParInc}, Starovoitov gave an example of a moving ball such that collision appears, which we will review in Chapter~\ref{ch:7}. As it seems at first glance to be in contradiction with the above result such a ball shall behave like a ``parabola'' with ${\alpha=1}$, we see that Starovoitov's example needs an additional ``singular'' force $\vc f \in L^2(0,T;\widehat{W}^{-1,2}(\Omega))$ adapted to the constructed velocity to ensure collision, where $\widehat{W}^{-1,2}(\Omega)$ is the dual to $\{ \vu \in W_0^{1,2}(\Omega): {\div \vu = 0} \}$. On the other hand, Theorem~\ref{thm:ParInc} is valid for arbitrary forces in $(L^p + \nabla L^p)(\Omega)$, $1<p\leq \infty$, sufficiently regular in time.
\end{Remark}

With the help of Lemmata~\ref{lem:intEst}, \ref{lem:singRemove}, and Proposition~\ref{prop:Remainder}, we are ready to prove Theorem~\ref{thm:ParInc}.
\begin{proof}[Proof of Theorem~\ref{thm:ParInc}]
By the regularity achieved for $\vc w_h$, we are allowed to use it as a test function for the momentum equation. This gives
\begin{align*}
&\int_0^t \int_\Omega \rho \vu \cdot \del_s \vc w_{h(s)} + \rho \vu \otimes \vu : \nabla \vc w_{h(s)} - \nabla \vu : \nabla \vc w_{h(s)} - \rho g \vc e_d \cdot \vc w_{h(s)} \dd x \dd s\\
&= \int_\Omega (\rho \vu \cdot \vc w_h)(t) \dd x - \int_\Omega \vc m_0 \cdot \vc w_{h(0)} \dd x.
\end{align*}
Since $\rho = \rho_\mf \chi_{\mf(t)} + \rho_\ms \chi_{\ms(t)}$ and $\div \vc w_h = 0$, we get
\begin{align*}
\int_\Omega \rho g \vc e_d \cdot \vc w_{h(t)} \dd x &= \int_{\ms(t)} \rho_\ms g \vc e_d \cdot \vc e_d + \int_{\mf(t)} \rho_\mf g \nabla [x \mapsto x_d] \cdot \vc w_{h(t)}\\
&= \rho_\ms g |\ms(0)| + \rho_\mf g \int_{\del \mf(t) \setminus \del \Omega} x_d \vc e_d \cdot \vc n \dd \sigma = \rho_\ms g |\ms(0)| - \rho_\mf g \int_{\del \ms(t)} x_d \vc e_d \cdot \vc n \dd \sigma \\
&= \rho_\ms g |\ms(0)| - \rho_\mf g \int_{\ms(t)} \div(x_d \vc e_d) \dd x = g |\ms(0)| (\rho_\ms - \rho_\mf),
\end{align*}
where we have used that $\vc w_h |_{\del \Omega} = 0$ and $\vc n |_{\del \mf(t) \setminus \del \Omega} = -\vc n|_{\del \ms(t)}$. Further, we write
\begin{align*}
2 \int_\Omega \bD(\vu)(t) : \bD(\vc w_{h(t)}) \dd x &= \dot{h}(t) \int_{\del \mf(t)} (\nabla \vc w_{h(t)} - q_{h(t)} \Id)\vc n \cdot \vc e_d \dd \sigma - \int_{\mf(t)} (\Delta \vc w_{h(t)} - \nabla q_{h(t)}) \cdot \vu \dd x\\
&= \dot{h}(t) n(h) - \int_{\mf(t)} (\Delta \vc w_{h(t)} - \nabla q_{h(t)}) \cdot \vu \dd x,
\end{align*}
where $n(h)$ is as in Proposition~\ref{prop:Remainder}. Note that here we have used $\div \vc u = 0$ in order to ``smuggle in'' the pressure $q_{h(t)}$. Denoting
\begin{align*}
N(h(t)) := \int_0^t n(h(s)) \dd s,
\end{align*}
and recalling the definition of $R(t)$ as in Proposition~\ref{prop:Remainder}, we can write the weak formulation of the momentum equation as
\begin{align*}
N(h(t)) + (\rho_\ms - \rho_\mf) g |\ms(0)| t &= R(t).
\end{align*}
The proof is now easily finished: if $\beta \geq 1$ in Proposition~\ref{prop:Remainder}, then
\begin{align*}
N(h(t)) \gtrsim (\rho_\mf - \rho_\ms)t - (1+\sqrt{t}).
\end{align*}
Moreover, by the same token for $h>0$ small enough,
\begin{align*}
N(h(t)) \lesssim \begin{cases}
|\log h| & \text{if } \beta=1,\\
h^{1-\beta} & \text{if } \beta>1.
\end{cases}
\end{align*}
This yields
\begin{align*}
|\log h|(t) \lesssim (\rho_\ms - \rho_\mf) t + \sqrt{t} + 1 < \infty \quad \forall \ t<T_\ast,
\end{align*}
in particular,
\begin{align*}
h(t) \geq C \exp\{-(\rho_\ms - \rho_\mf) t - \sqrt{t} \}.
\end{align*}
Recalling $\rho_\ms>\rho_\mf$, this means that $h$ cannot vanish in finite time and no collision occurs. Especially, the maximal existence time of the solution $(\rho, \vu)$ is $T_\ast = \infty$. Note that $\beta \geq 1$ is satisfied if $d=2$ and $\alpha \geq 1/2$, or if $d=3$ and $\alpha=1$, which fits precisely the cases of a ``blunt'' parabola and a ball-shaped obstacle, respectively.

If $\beta<1$ which can be just the case for $d=2$ and $\alpha<1/2$, then $n \in L^1(0,T_\ast)$ and $N$ as its primitive is continuous. Since $h(t)$ is bounded below and above, we deduce
\begin{align}\label{ineq3}
- \infty < \inf_{t \in (0, T_\ast)} N(h(t)) \lesssim (\rho_\mf - \rho_\ms) t + \sqrt{t} + 1.
\end{align}
If now $T_\ast = \infty$, we can send $t \to \infty$ on the right-hand side of the above inequality. But $\rho_\ms>\rho_\mf$, thus the right-hand side goes to $-\infty$, which is a contradiction. Eventually, $T_\ast < \infty$, meaning that $h$ vanishes in finite time and collision occurs. The proof is finished.
\end{proof}

\subsection{Concluding remarks}\label{sec:conc}
A similar but more sophisticated estimation of the drag force using energy considerations was given in \cite[Section~3.1]{GerardVaretHillairet2012}, although just for the incompressible Stokes equations. To begin, the drag force on the solid $\ms$ is defined as the component of the total force the fluid applies on the solid which is parallel to the flow velocity, that is,
\begin{align*}
\mathcal{D}_h = \int_{\del \ms_h} \vu \cdot (\bD(\vu)-p\Id)\mathbf n \dd \sigma.
\end{align*}
Next, by the standard existence theory of Stokes equations, we can show the following
\begin{Lemma}
There exists a unique solution $\vv_h \in W_0^{1,2}(\Omega)$ to
\begin{align*}
\begin{cases}
\div \vv_h = 0 & \text{in } \mf_h,\\
-\Delta \vv_h + \nabla p_h = 0 & \text{in } \mf_h,\\
\vv_h = \mathbf e_3 & \text{on } \ms_h,\\
\vv_h = 0 & \text{on } \del \Omega.
\end{cases}
\end{align*}
\end{Lemma}
Since $\vv_h |_{\del \ms_h} = \mathbf e_3$, $\vv_h |_{\del \Omega}=0$, and $\div \vv_h=0$, we may write
\begin{align}\label{drag}
\begin{split}
\mathcal{D}_h &= \int_{\del \ms_h \cup \del \Omega} \vv_h \cdot (\bD(\vv_h)-p_h\Id)\mathbf n \dd \sigma = \int_{\mf_h} \div(\bD(\vv_h) \vv_h - p_h\vv_h) \dd x\\
 &= \int_{\mf_h} (\Delta \vv_h- \nabla p_h) \cdot \vv_h + \bD(\vv_h):\nabla \vv_h - p_h \div \vv_h \dd x = \int_{\mf_h} |\bD(\vv_h)|^2 \dd x.
\end{split}
\end{align}
Thus, the drag force is completely determined by the behavior of the symmetric gradient of $\vv_h$. The test functions $\vc w_h$ constructed above now ``almost'' solve this Stokes problem for $\vv_h$ in the sense that $-\Delta \vc w_h + \nabla q_h$ does not vanish, but is controlled as seen by Lemma~\ref{lem:singRemove}. Hence, the drag force calculated with $\vc w_h$ instead of $\vv_h$ shall behave similarly. Indeed, the outcomes of \cite{GerardVaretHillairet2012} roughly read
\begin{align*}
\mathcal{D}_h(\vc w_h) \sim \begin{cases}
h^\frac{1-3\alpha}{1+\alpha} & \text{if } \alpha>\frac13,\\
|\log h| & \text{if } \alpha=\frac13,\\
1 & \text{if } \alpha<\frac13,
\end{cases}
\end{align*}
which fits our calculations done before. Note moreover that this means the convective term $\rho_\mf \div(\vu \otimes \vu)$ does not play a role in the question whether or not collision occurs. Heuristically, this is again easy to explain: by Theorem~\ref{thm:Starov}, the solid touches the ground with zero speed, meaning the fluid behaves as a creeping (also called Stokes) flow. Such flows are usually modelled by the Navier-Stokes equations with negligible convective term (if the fluid's velocity is small, then the ``quadratic'' convective term is even smaller), which gives rise precisely to the Stokes equations.\\

As for compressible fluids, there arises another difficulty. Applying the same technique as before, introducing the functions $(\vc w_h, q_h)$ as test functions, and integrating by parts in the term $\int_\Omega \bD(\vu) : \bD(\vc w_h) \dd x$, we get an additional term
\begin{align*}
\int_{\mf_h} q_h \div \vu \dd x,
\end{align*}
which we cannot control although the pressure $q_h$ is explicitly given.
Moreover, we cannot apply the same technique as for incompressible fluids since there might be regions with $\rho_\mf = 0$ (vacuum), but also regions with $\rho_\mf > \rho_\ms$, hence, naively, inequality \eqref{ineq3} tells us nothing.\\

Lastly, let us also mention Starovoitov's work \cite{Starovoitov2003}, where the author gives a detailed analysis of the solid's behavior \emph{during} contact. The main outcomes of this work are that if $d=2$, the solid sticks to the boundary of its container as long as $\alpha \geq \frac23$ (in fact, for non-Newtonian fluids with $\vu \in L^p(0,T;W^{1,p}(\Omega))$, this holds as long as $(2p-1)\alpha \geq 2$). If $d=3$, the solid still can rotate around an axis orthogonal to the boundary at the point of collision as long as $\alpha=1$ (non-Newtonian: $(2p-1) \alpha \geq 3$). In particular, if the solid touches the boundary in more than one point, then also in three dimensions, it sticks motionless on $\del \Omega$. The same restrictions occur in \cite[Theorem~3.2]{FilippasTersenov2021} for $d=2$, and \cite[Theorem~1.1]{FilippasTersenov2024} for $d=3$ as optimal values in the sense that for the reversed inequalities, there exist $\Omega$ and a vector field $\vu \in L^p(0,T; W_0^{1,p}(\Omega))$ such that the solid $\ms$ still moves (surprisingly, when dropping the assumption $\div \vu = 0$, the condition for $d=2$ has to be replaced by $(p-1)\alpha \geq 2$, see \cite[Theorem~2.1]{FilippasTersenov2021}). We will come back to this in the next chapter.



\chapter{Special forces, non-uniqueness, and no-collision with controls}\label{ch:7}
In this chapter, we investigate the following three problems for ball-shaped solids:
\begin{enumerate}
\item Give a specific example of a solution $\vu$ and a driving force $\vc f$ for incompressible Navier-Stokes equations such that collision for a ball-shaped solid happens.
\item Show that, for a specific force $\vc f$, the solution to the Navier-Stokes equations is not unique.
\item Show that collision in a compressible fluid is forbidden for additional controls, or higher regularity of the solution.
\end{enumerate}

Recall that for incompressible fluids and gravity, we just proved that collision for a ball-shaped body is \emph{forbidden}. As we will see in the next section, this does not contradict the outcome for the first problem, since the corresponding driving force will be sufficiently ``bad''.

\section{A singular force}
\providecommand{\g}{{\mathrm g}}
Let us start with the example of a ``singular'' driving force for incompressible Navier-Stokes equations for the two-dimensional case $d=2$, following the presentation of \cite{Starovoitov2003}. We make the assumptions $\Omega = B_R(0)$, $\ms(t) = B_r(\vc G(t))$, $0<r<R$, and $\vc G(t) = (\g(t), 0)$ with $|\g(t)| \leq R-r$ for all $t \in (0,T)$, meaning the solid moves inside $\Omega$ just along the horizontal axis. The idea in constructing a colliding solution will be the following: first, we define a specific velocity field having the desired properties of energy estimates and collision; second, the driving force will be the error coming from inserting the constructed solution in the Navier-Stokes equations. This follows the heuristics of ``construct a function that you want, calculate the derivatives, and call everything that you don't want to have the driving force''.\\

Following these heuristics, we construct a velocity field that will ensure collision in finite time. We will rely on the construction of \cite[Section~2]{Starovoitov2005}, however, we choose here to directly apply polar coordinates instead of first using Cartesian coordinates and then make a second change of variables as in the reference. To this end, let us define a map $F: (0,T) \times \Omega \ni (t, \rho, \theta) \mapsto x \in \Omega$ via
\begin{align*}
x_1 &= F_1(t, \rho, \theta) = \rho \cos \theta + \sigma(t) \big( R - \rho \big),\\
x_2 &= F_2(t, \rho, \theta) = \rho \sin \theta,
\end{align*}
where $\sigma(t) = \g(t) (R-r)^{-1} \in [0, 1)$ such that $\lim_{t \to T_\ast} \sigma(t) = 1$. As it is convenient for us, for $\xi = (\rho \cos \theta, \rho \sin \theta)$, we will not distinguish between $F(t, \rho, \theta)$ and $F(t, \xi) = F(t, \rho \cos \theta, \rho \sin \theta)$. It is easy to see that for any $t \in (0,T)$ we have $F(t, \Omega)=\Omega$ and $F(t, B_r(0)) = \ms(t)$. Moreover, if $\xi \in \del \Omega$ such that $\rho=R$, then $F(t, \xi) = \xi$. Computing the inverse mapping $F^{-1}(t,x)$, we first find
\begin{align*}
(x_1 - \sigma R)^2 + x_2^2 = \rho^2 (1-\sigma^2) - 2 \sigma \rho (x_1 - \sigma R),
\end{align*}
which then yields
\begin{align*}
\rho = \frac{1}{1-\sigma^2} \left( \sigma (x_1 - \sigma R) + \sqrt{(1 + \sigma^2) (x_1 - \sigma R)^2 + x_2^2} \right), &&
\theta = \begin{cases}
\arcsin \frac{x_2}{\rho} & \text{if } x_2 \geq 0,\\
2\pi - \arcsin \frac{x_2}{\rho} & \text{if } x_2<0.
\end{cases}
\end{align*}
Especially, we have $F^{-1}(t, \Omega) = \Omega$ and $F^{-1}(t, \ms(t)) = B_r(0)$. The actions of $F$ and $F^{-1}$ are depicted in Figure~\ref{fig2}.

\begin{figure}[H]\label{fig2}
\centering
\begin{tikzpicture}
\draw[->] (-2,0) -- (2,0);
\node at (0,2) [anchor=west] {$\xi_1$};
\draw[->] (0,-2) -- (0,2);
\node at (2,0) [anchor=west] {$\xi_2$};

\draw (0,0) circle (0.7cm);
\node at (-.5, 1) {$\Omega$};
\draw[rotate=45] (-.7,0) -- (0,0);
\node at (-.2,-.2) [anchor=north] {$r$};
\draw (0,0) circle (1.5cm);
\draw[rotate=135] (-1.5,0) -- (0,0);
\node at (.9,-.9) [anchor=south] {$R$};

\draw[->] (5,0)--(9,0);
\node at (9,0) [anchor=west] {$x_1$};
\draw[->] (7,-2) -- (7,2);
\node at (7,2) [anchor=west] {$x_2$};

\draw (7.5,0) circle (0.7cm);
\node at (6.5, 1) {$\Omega$};
\draw (7,0) circle (1.5cm);
\node at (7.5, -.3) {$\ms$};

\draw[->] (2,1) to [out=30, in=150] (5,1);
\node at (3.5,1.5) [anchor=south] {$x = F(t,\xi)$};
\draw[->] (5,-1) to [out=-150, in=-30] (2,-1);
\node at (3.5,-1.5) [anchor=north] {$\xi = F^{-1}(t,x)$};
\end{tikzpicture}
\caption{The actions of the map $F(t,\xi) = F(t, \rho, \theta)$ and its inverse $F^{-1}(t,x)$.}
\end{figure}

\begin{Remark}
We will use a similar mapping to transform a ball falling over a half-plane to a concentric situation in Chapter~\ref{ch:8} by using complex analysis. In fact, the same results found in the present section can be obtained by mapping the domain $\Omega \setminus \ms = B_R(0) \setminus B_r((\g,0)) \subset \R^2 \sim \mathbb C$ to the concentric domain $B_1(0) \setminus B_{r_0}(0) \subset \mathbb C$ by the conformal function
\begin{align*}
\tilde{F}(z) &= R \frac{z-c}{R^2-cz}, \quad c = \frac{1}{2 \g} (R^2+\g^2-r^2-W), \quad r_0 = \tilde{F}(\g + r) = \frac{2Rr}{R^2+r^2-\g^2+W}, \\
W &= \big( (R+r+\g)(R+r-\g)(R-r+\g)(R-r-\g) \big)^\frac12,
\end{align*}
see \cite[§5.7]{Henrici1988} for the derivation of this mapping.
\end{Remark}

Computing the Jacobi matrices of $F$ and $F^{-1}$, we see
\begin{align*}
\nabla_{(\rho,\theta)} F &= \begin{pmatrix}
\cos \theta - \sigma & - \rho \sin \theta\\
\sin \theta & \rho \cos \theta
\end{pmatrix}, \\
\nabla_x (F^{-1}) &= [\nabla_{(\rho, \theta)} F]^{-1} \circ F^{-1} = \left[ \frac{1}{\rho (1 - \sigma \cos \theta)} \begin{pmatrix}
\rho \cos \theta & \rho \sin \theta\\
- \sin \theta & \cos \theta - \sigma 
\end{pmatrix} \right] \circ F^{-1},
\end{align*}
leading for the Jacobians to
\begin{align*}
J_F &= \det (\nabla_\xi F) = \rho (1 - \sigma(t) \cos \theta),\\
J_{F^{-1}} |_{x = F(t, \xi)} &= J_F^{-1} \circ F^{-1} = \frac{1}{\rho( 1-\sigma(t) \cos \theta )} \circ F^{-1}(t,x).
\end{align*}

Since we consider a two-dimensional domain, it is convenient to define the velocity via a stream function\footnote{The streamlines $\vc X$ defined by $\dot{\vc X} = \vu(t, \vc X)$ are in this context lines where the stream function $\psi(t, \vc X)$ is constant (that is, level sets of $\psi$).} $\psi(t, x)$ associated to the flow, that is, $\vu(t, x) = \nabla^\perp \psi(t, x)$. We still want that
\begin{align}\label{i2}
\vu = 0 \ \text{on} \ \del \Omega, \quad \bD(\vu)=0 \ \text{in} \ \ms.
\end{align}
Let us consider the symmetric domain $B_R(0) \setminus B_r(0) = F(t, \Omega \setminus \ms(t))$, and we search for a function $\tilde \vu$ inside this domain. Then, the function $\tilde \vu(t, \xi)$ shall satisfy
\begin{align*}
\tilde \vu(t, \xi) = 0 \ \text{whenever} \ |\xi|=R, \quad \bD(\tilde \vu)(t, \xi) = 0 \ \text{whenever} \ |\xi| < r.
\end{align*}
As our solid shall move in $x_1$-direction with speed $\dot \g(t)$, we search for $\tilde \vu$ in the form $\tilde \vu(t, \xi)|_{B_r(0)} = \dot \g(t) (1, 0)^T = \dot \g(t) \nabla^\perp [\xi \mapsto \xi_2 = \rho \sin \theta]$. This leads us to the ansatz $\tilde \psi(t, \xi) = \tilde \psi(t, \rho, \theta) = \dot \g(t) \rho \phi(\rho) \sin \theta$ for some smooth function $\phi$ with $\phi(\rho)=1$ for $\rho<r$. The boundary condition for $\tilde \vu$ on $\del \Omega$ then suggests to search $\phi$ with $\phi(R)=0$. In all other parameters, $\phi$ is free to choose; we therefore may choose this function such that
\begin{align*}
&\phi: [0,\infty) \to [0,1] \ \text{is decreasing on} \ [0, \infty),\\
&\phi(\rho) = 1 \ \text{if} \ \rho<r, \quad \phi(R)=\phi'(r)=\phi'(R)=0.
\end{align*}
The example given in \cite{Starovoitov2005} is
\begin{align*}
\phi(\rho) = \begin{cases}
1 & \text{if } 0 \leq \rho < r,\\
(R-r)^{-3} (\rho - R)^2 (2\rho - 3r + R) & \text{if } r \leq \rho \leq R,\\
0 & \text{else},
\end{cases}
\end{align*}
which is an easy function in the sense of applications and numerical experiments since it is just a polynomial of degree three. Especially for $r=1$ and $R=3$, $\phi(\rho) = \frac14 \rho (\rho - 3)^2$ inside $[r, R]$. As another example the function $\zeta$ introduced in Section~\ref{sec:PfThm} may serve.\\

Back to the velocity field in the original domain, we define $\vu(t, x) = \nabla^\perp \psi(t, x)$ with
\begin{align*}
\psi(t, x) = \tilde{\psi}(t, \rho, \theta)|_{(\rho, \theta) = F^{-1}(t, x)} = \dot{\g}(t) [ \rho \phi(\rho) \sin \theta ] \circ F^{-1}(t,x).
\end{align*}
Note especially that $\vu$ satisfies precisely \eqref{i2}. To verify that $\vu$ serves as an appropriate solution to our problem, we have to verify that $\vu \in L^\infty(0,T;L^2(\Omega)) \cap L^2(0,T;W_0^{1,2}(\Omega))$. As for the $L^2$-norm and since $|\nabla^\perp \psi| = |\nabla \psi|$, we first calculate by chain rule
\begin{align}\label{gradpsi}
\begin{split}
\nabla_x \psi(t,x) &= \left( \nabla_{(\rho, \theta)} \tilde{\psi}(t, \rho, \theta) \cdot [\nabla_{(\rho, \theta)} F(t,\rho,\theta)]^{-1} \right) \circ F^{-1}(t,x) \\
&= \left[ \frac{\dot{\g}}{1-\sigma \cos \theta} \big( \rho \phi'(\rho) \sin \theta \cos \theta, \ \phi(\rho) (1 - \sigma \cos \theta) + \rho \phi'(\rho) \sin^2 \theta \big) \right] \circ F^{-1}(t,x).
\end{split}
\end{align}
Thus, as $F(t,\Omega) = \Omega = F^{-1}(t, \Omega)$ and $J_F = \rho (1 - \sigma \cos \theta)$,
\begin{align*}
\|\vu\|_{L^2(\Omega)}^2 &= \int_\Omega |\nabla_x \psi(t,x)|^2 \dd x \\
&= \int_\Omega \Big| \nabla_{(\rho, \theta)} \tilde \psi(t, \rho, \theta) \cdot [\nabla_{(\rho, \theta)} F(t, \rho, \theta)]^{-1} \cdot \sqrt{J_F} \Big|^2 \circ F^{-1}(t,x) \cdot J_F^{-1} \dd x \\
&= \int_0^R \int_0^{2 \pi} \Big| \nabla_{(\rho, \theta)} \tilde \psi(t, \rho, \theta) \cdot [\nabla_{(\rho, \theta)} F(t, \rho, \theta)]^{-1} \Big|^2 \rho (1 - \sigma \cos \theta) \dd \rho \dd \theta.
\end{align*}
With equation \eqref{gradpsi}, we have that the integrand equals
\begin{align*}
\frac{|\dot{\g}|^2}{1-\sigma \cos \theta}
\Big[ \rho^3 (\phi'(\rho))^2  \sin^2 \theta + \rho^2 (\phi^2(\rho))' \sin^2 \theta (1 - \sigma \cos \theta) + \rho (\phi(\rho))^2 (1 - \sigma \cos \theta)^2 \Big],
\end{align*}
leading finally to
\begin{align}\label{i4}
\begin{split}
\|\vu\|_{L^2(\Omega)}^2 &= |\dot{\g}|^2 \int_0^R \int_0^{2\pi} \frac{1}{1-\sigma \cos \theta}
\Big[ \rho^3 (\phi'(\rho))^2  \sin^2 \theta \\
&\qquad + \rho^2 (\phi^2(\rho))' \sin^2 \theta (1 - \sigma \cos \theta) + \rho (\phi(\rho))^2 (1 - \sigma \cos \theta)^2 \Big] \dd \rho \dd \theta \\
&= \mu_1(r,R) \nu_1(\sigma) |\dot \sigma|^2,
\end{split}
\end{align}
where
\begin{align*}
\mu_1(r,R) &= (R-r)^2 \int_r^R \rho^3 (\phi'(\rho))^2 \dd \rho < \infty, \\
\nu_1(\sigma) &= \int_0^{2\pi} \frac{\sin^2 \theta}{1 - \sigma \cos \theta} \dd \theta = \frac{2\pi}{1 + \sqrt{1-\sigma^2}} \in [\pi, 2\pi].
\end{align*}
In turn, we have:
\begin{Lemma}\label{lem1}
The function $\vu \in L^\infty(0,T;L^2(\Omega))$ if and only if $\dot \sigma \in L^\infty(0,T)$.
\end{Lemma}

To verify that $\vu \in L^2(0,T;W_0^{1,2}(\Omega))$, by boundedness of Riesz transform in $L^2(\R^2)$ (see \eqref{Riesz}), it is enough to estimate $\Delta \psi$. First, using again chain rule,
\begin{align*}
\Delta_x \psi &= \div_x (\nabla_x \psi) = \div_x ((\nabla_{(\rho, \theta)} \tilde \psi(t, \rho, \theta) \cdot [\nabla_{(\rho, \theta)} F(t, \rho, \theta)]^{-1}) \circ F^{-1}(t,x)) \\
&= \left( \nabla_{(\rho, \theta)} [\nabla_{(\rho, \theta)} \tilde \psi(t, \rho, \theta) \cdot (\nabla_{(\rho, \theta)} F(t, \rho, \theta))^{-1}] : [ \nabla_{(\rho, \theta)} F(t, \rho, \theta) ]^{-T} \right) \circ F^{-1}(t,x).
\end{align*}

Some quick formal discussion\footnote{Here, we do not have to concentrate on the dependence on $\rho$ since this part will always be integrable due to compact support of $\phi$.} hints in terms of $\theta$ that we have $\big|\nabla_{(\rho, \theta)} [\nabla_{(\rho,\theta)} F(t,\rho,\theta)]^{-1} \big| \sim \big|[\nabla_{(\rho, \theta)} F(t,\rho,\theta)]^{-1} \big|^2$. Hence,
\begin{align*}
|\Delta \psi| \sim \big|\nabla_{(\rho, \theta)} \tilde \psi(t, \rho, \theta)  \cdot [\nabla_{(\rho, \theta)} F(t, \rho, \theta)]^{-1} \big| \cdot \big| [\nabla_{(\rho,\theta)} F(t,\rho,\theta)]^{-1} \big| \sim \sin \theta (1 - \sigma \cos \theta)^{-2},
\end{align*}
and therefore we shall expect
\begin{align*}
\|\Delta \psi\|_{L^2(\Omega)}^2 \lesssim \int_0^{2 \pi} \frac{\sin^2 \theta}{(1 - \sigma \cos \theta)^3} \dd \theta.
\end{align*}

To make this formal guess rigorous, we calculate
\begin{align*}
&\nabla_{(\rho, \theta)} (\nabla_{(\rho, \theta)} \tilde \psi \cdot (\nabla_{(\rho, \theta)} F)^{-1}) \\
&= \nabla_{(\rho,\theta)} \left[ \frac{\dot{\g}}{1-\sigma \cos \theta} ( \rho \phi' \sin \theta \cos \theta, \ \phi (1 - \sigma \cos \theta) + \rho \phi' \sin^2 \theta ) \right] \\
&= \dot{\g} \begin{pmatrix}
\frac{(\phi' + \rho \phi'') \sin \theta \cos \theta}{1-\sigma \cos \theta} & \rho \phi' \frac{\cos^2 \theta - \sin^2 \theta - \sigma \cos^3 \theta}{(1 - \sigma \cos \theta)^2} \\
\phi' + \frac{(\phi' + \rho \phi'') \sin^2 \theta}{1-\sigma \cos \theta} & \rho \phi' \frac{2 \cos \theta \sin \theta - 2 \sigma \cos^2 \theta \sin \theta - \sigma \sin^3 \theta}{(1- \sigma \cos \theta)^2}
\end{pmatrix},
\end{align*}
thus
\begin{align*}
\Delta_x \psi
&= \left[ \frac{\dot{\g}}{1-\sigma \cos \theta} \left( \frac{\rho \phi''(\rho) \sin \theta}{1-\sigma \cos \theta} + 3 \phi'(\rho) \sin \theta + \frac{\sigma^2 \phi'(\rho) \sin^3 \theta}{(1- \sigma \cos \theta)^2} \right) \right] \circ F^{-1}(t,x), \\
|\Delta_x \psi|^2 &= \left[ \frac{|\dot{\g}|^2}{(1-\sigma \cos \theta)^2} \left( \frac{\rho^2 (\phi''(\rho))^2 \sin^2 \theta}{(1-\sigma \cos \theta)^2} + \frac{3 \rho [(\phi'(\rho))^2]' \sin^2 \theta}{1-\sigma \cos \theta} + \frac{\rho [(\phi'(\rho))^2]' \sigma^2 \sin^4 \theta}{(1- \sigma \cos \theta)^3} \right. \right. \\
&\qquad \left. \left. + 9 (\phi'(\rho))^2 \sin^2 \theta + \frac{6 \sigma^2 (\phi'(\rho))^2 \sin^4 \theta}{(1- \sigma \cos \theta)^2} + \frac{\sigma^4 (\phi'(\rho))^2 \sin^6 \theta}{(1- \sigma \cos \theta)^4} \right) \right] \circ F^{-1}(t,x).
\end{align*}
Hence, for $f(\rho) = \rho^2 (\phi''(\rho))^2 + 4 \rho [(\phi'(\rho))^2]' + 16 (\phi'(\rho))^2$, we finally get
\begin{align*}
\|\Delta \psi\|_{L^2(\Omega)}^2 &\leq |\dot{\g}|^2 \int_r^R f(\rho) \dd \rho \int_0^{2\pi} \frac{\sin^2 \theta}{(1-\sigma \cos \theta)^3} + \frac{\sin^2 \theta}{(1-\sigma \cos \theta)^2} + \frac{\sigma^2 \sin^4 \theta}{(1- \sigma \cos \theta)^4} \\
&\qquad + \frac{\sin^2 \theta}{1-\sigma \cos \theta} + \frac{\sigma^2 \sin^4 \theta}{(1- \sigma \cos \theta)^3} + \frac{\sigma^4 \sin^6 \theta}{(1- \sigma \cos \theta)^5} \dd \theta \\
&\lesssim |\dot{\g}|^2 \int_0^{2 \pi} \frac{\sin^2 \theta}{(1-\sigma \cos \theta)^3} + \frac{\sigma^2 \sin^4 \theta}{(1- \sigma \cos \theta)^4} + \frac{\sigma^4 \sin^6 \theta}{(1- \sigma \cos \theta)^5} \dd \theta,
\end{align*}
where the implicit constant just depends on $r$ and $R$ (and $\phi$). Noting that in the above integral, all terms are equally singular as $\sigma \to 1$, we indeed find\footnote{This last integral can be solved in using the complex representations of $\sin \theta$ and $\cos \theta$ via $e^{i \theta}$ and residue theorem (see \cite[§4.7]{Henrici1988}). Similar calculations work for $\nu_1(\sigma)$ in \eqref{i4}.}
\begin{align*}
&\int_0^{2 \pi} \frac{\sin^2 \theta}{(1-\sigma \cos \theta)^3} + \frac{\sigma^2 \sin^4 \theta}{(1- \sigma \cos \theta)^4} + \frac{\sigma^4 \sin^6 \theta}{(1- \sigma \cos \theta)^5} \dd \theta \lesssim \int_0^{2 \pi} \frac{\sin^2 \theta}{(1-\sigma \cos \theta)^3} \dd \theta = \frac{\pi}{(1-\sigma^2)^{3/2}};
\end{align*}
%
in turn,
\begin{align*}
\|\Delta \psi\|_{L^2(\Omega)}^2 \leq \mu_2(r,R) (1-\sigma^2)^{-3/2} |\dot \sigma|^2,
\end{align*}
where $\mu_2(r,R) < \infty$ is a positive constant depending only on $r$ and $R$. Thus, we can state
\begin{Lemma}\label{lem2}
The function $\vu \in L^2(0,T;W_0^{1,2}(\Omega))$ if and only if $(1-\sigma^2)^{-3/2} |\dot \sigma|^2 \in L^1(0,T)$.
\end{Lemma}

\begin{Remark}\label{rem:beta34}
Note that the result of \cite[Lemma~2.2]{Starovoitov2005} reads
\begin{align*}
\|\Delta \psi\|_{L^2(\Omega)}^2 \leq \mu_2(r,R) (1-\sigma)^{-3/2} |\dot \sigma|^2,
\end{align*}
which, compared to our outcome, seems to miss an exponent $2$ in $\sigma$. However, since $\sigma < 1$, we immediately see $(1-\sigma^2)^{-3/2} \leq (1-\sigma)^{-3/2}$; hence, our estimate is just sharper. Moreover, recalling that $h(t) = (R-r)(1-\sigma(t))$, and that Starovoitov's condition reads $h^{-3/4} |\dot h| \in L^2(0,T)$, this fits precisely the estimate \eqref{ineq1} in Theorem~\ref{thm:Starov}, where then with $\alpha=1$ and $p=d=2$, we have $\beta = \frac{1+2\alpha}{p(1+\alpha)} (p-\frac{d+\alpha}{1+2\alpha}) = \frac34$.
\end{Remark}

All that is left to do now is to find a function $\sigma(t)$ with $\sigma(T_\ast)=1$ for some $T_\ast \in (0,T)$ fulfilling the conditions of Lemmata \ref{lem1} and \ref{lem2}. Such a function is easy to find: let $T>0$ and $T_\ast \in (0,T)$ be fixed, then we can take
\begin{align}\label{i3}
\sigma(t) = 1 - \bigg(\frac{t-T_\ast}{T} \bigg)^4.
\end{align}
It is easy to see that $\dot \sigma, (1-\sigma^2)^{-3/2} |\dot \sigma|^2 \in L^\infty(0,T)$ and in particular $\sigma$ satisfies the conditions of Lemmata \ref{lem1} and \ref{lem2}. Hence, for this $\sigma$, the function $\vu$ constructed above is a velocity field satisfying the energy estimate and ensuring collision in finite time.\\

However, we are still missing one point: till now we don't know that $\vu$ is a weak solution to the Stokes problem. To ensure also this last requirement, we need to bound the time derivative $\del_t \vu$. Obviously, it is enough to do this for $\del_t \psi$, yielding
\begin{align*}
\del_t \psi(t,x) &= \del_t [\dot \g(t) \rho \phi(\rho) \sin \theta \circ F^{-1}(t,x)] \\
&= \ddot \g(t) [\rho \phi(\rho) \sin \theta] \circ F^{-1}(t,x) + \dot \g(t) [ \nabla_{(\rho, \theta)} (\rho \phi(\rho) \sin \theta)] \circ F^{-1}(t,x) \cdot \del_t F^{-1}(t, x).
\end{align*}
Together with $\del_t F(t,\rho, \theta) = \dot \sigma(t) (R-\rho, 0)^T$, by the usual change of variables $x = F(t,\rho, \theta)$ we arrive at
\begin{align*}
\|\del_t \psi\|_{L^2(\Omega)} \leq \mu_3(r,R) (|\ddot \sigma| + |\dot \sigma|^2)
\end{align*}
for some constant $\mu_3(r,R)<\infty$. This yields
\begin{Lemma}
If $\ddot \sigma \in L^2(0,T)$, then $\del_t \vu \in L^2(0,T;W^{-1,2}(\Omega))$.
\end{Lemma}
Again, for our function $\ddot \sigma = \frac{\rd^2}{\rd t^2} (1 - (t-T_\ast)^4 T^{-4}) \in L^\infty(0,T)$, hence fulfilling the condition of the lemma. Setting now the force
\begin{align*}
\vc f := \del_t \vu + \div(\vu \otimes \vu) - \div \bS(\nabla \vu) \in L^2(0,T; W^{-1,2}(\Omega)),
\end{align*}
we have proven the following
\begin{Theorem}
If $\sigma(t)$ is such that $\dot \sigma \in L^\infty(0,T)$, $\ddot \sigma \in L^2(0,T)$, and $(1-\sigma^2)^{-3/2} |\dot \sigma|^2 \in L^1(0,T)$, then there exists a force $\vc f \in L^2(0,T;W^{-1,2}(\Omega))$ such that the function $\vu$ constructed above is a weak solution to the Navier-Stokes equations ensuring collision in finite time.
\end{Theorem}

\section{Two colliding solutions}
Having found a solution to our problem, we investigate now the question whether this solution is uniquely determined. To this end, we will assume that there exists $T_\ast \in (0,T)$ such that the ball-shaped body $\ms = B(\vc G(t))$ touches the boundary of $\Omega$. As seen in the last section, this can just happen if the driving force $\vc f$ is ``sufficiently bad'', meaning a distribution rather than a function. In particular, gravity is forbidden in this setting, since collision does not occur for a ball as proven in Chapter~\ref{ch:6}.\\

We will follow the construction from \cite{Starovoitov2005}. Indeed, the idea is similar as in the foregoing section: the domains $\Omega$ and $\ms$ will be the same disks in $\R^2$. We also take the same velocity $\vu = \nabla^\perp \psi(t, x)$ up to time $T_\ast$. From time $T_\ast$ onwards, one solution will simply go ``back in time'', whereas the other one will stick to the boundary of $\Omega$ without any further movement. Note that this last solution is in some sense ``allowed to exist'': indeed, as mentioned in Section~\ref{sec:conc}, our function $\vu \in L^p(0,T;W_0^{1,p}(\Omega))$ with $p=2$ and hence $(2p-1)\alpha \geq 2$ as $\alpha=1$ for a ball.\\

Recalling the form of $\sigma(t)$ in \eqref{i3}, we see that $|\sigma(t)|<1$ if $t \in [0, T_\ast) \cup (T_\ast, T]$ and $\sigma(T_\ast)=1$. Hence, in agreement with the calculations done in the previous section, we have $\vu \in L^\infty(0, T_\ast; L^2(\Omega)) \cap L^2(0,T_\ast; W_0^{1,2}(\Omega))$ and $\del_t \vu \in L^2(0, T_\ast; W^{-1, 2}(\Omega))$. In particular, the function $\vu$ is strongly continuous in time in $L^2(\Omega)$, since $\|\vu(t, \cdot)\|_{L^2(\Omega)}$ is continuous in time by \eqref{i4}. Thus, we can prolong $\vu$ continuously in time.

\paragraph{The first solution.} For $t>T_\ast$, we simply define as before $\vu(t, x) = \nabla^\perp \psi(t, x)$. Since $\sigma$ and hence $\psi$ is symmetric in $(t- T_\ast)$, this is the solution moving ``back in time'' (note that this is precisely the reason for taking the fourth power in the definition of $\sigma$). By construction, we have
\begin{gather*}
\vu \in L^\infty(0,T; L^2(\Omega)) \cap L^2(0,T; W_0^{1,2}(\Omega)), \quad \del_t \vu \in L^2(0,T;W^{-1,2}(\Omega)),\\
\div \vu = 0, \quad \vu|_{\del \Omega} = 0, \quad \bD(\vu)|_{\ms} = 0,
\end{gather*}
and that $\vu$ fulfils the energy inequality. All is left is to find a force $\vc f \in L^2(0,T;W^{-1,2}(\Omega))$ such that the weak form of the momentum equation
\begin{align*}
\int_0^T \int_\Omega \del_t \vu \cdot \phi - (\vu \otimes \vu - \bD(\vu)):\bD(\phi) \dd x \dd t = \int_0^T \int_\Omega \vc f \cdot \phi \dd x \dd t
\end{align*}
holds for any $\phi \in L^2(0,T;W_0^{1,2}(\Omega))$ such that $\div \phi = 0$ and $\bD(\phi)|_{\ms} = 0$.

This is easy as well: by strong continuity of $\vu$ in time, we can view the left-hand side of the above equality as a continuous linear functional on $L^2(0,T;W_0^{1,2}(\Omega))$ applied to $\phi$. In turn, we can find a force $\vc f$ fulfilling the requirements needed; namely, as before,
\begin{align}\label{i6}
\vc f := \del_t \vu + \div(\vu \otimes \vu) - \div \bS(\nabla \vu) \in L^2(0,T;W^{-1,2}(\Omega)).
\end{align}
As $\vc f$ fulfils the weak formulation for \emph{any} $\phi \in L^2(0,T;W_0^{1,2}(\Omega))$, in particular it is an appropriate force for solenoidal $\phi$ with $\bD(\phi)|_{\ms} = 0$. This gives us the force and the first solution to our problem.

\paragraph{The second solution.} In order to finish the proof of non-uniqueness, we need to construct another solution $\vv(t, x)$ to the Navier-Stokes equations for the vector field $\vc f$ constructed above. This solution will collide with $\del \Omega$ and then stick at its position. Note moreover that the body $\ms$ is a part of the solution; hence, for the second solution $\vv$, we also need to make precise what the solid does. Denote by $\mathcal R$ the domain occupied by the body for the velocity field $\vv$.\\

We now proceed as the heuristics may suggest: the part until collision is
\begin{align*}
\vv(t, x) = \vu(t, x), \quad \mathcal{R}(t) = \ms(t), \quad t \in [0, T_\ast].
\end{align*}
As the solid touches $\del \Omega$ at $t=T_\ast$, the extended solution will be defined by
\begin{align*}
\mathcal{R}(t) &= \mathcal{R}(T_\ast) = \mathcal{R}_\ast = \ms(T_\ast), \quad t \in (T_\ast, T],\\
\vv(t, x) &= 0 \ \text{ for } \ t \in (T_\ast, T], \quad x \in \mathcal{R}(t) = \mathcal{R}_\ast.
\end{align*}
We remark that this extension, till now, \emph{is just inside $\mathcal{R}$}. It means that the solid does not move, but on the other hand, the fluid around may still do something. Thus, we search for $\vv$ in the (fixed) domain $\Omega \setminus \mathcal{R}_\ast$ with right-hand side $\vc f$ and $\vv|_{\del (\Omega \setminus \mathcal{R}_\ast)} = 0$. As $\vu(T_\ast, x)=0$ for every $x \in \Omega$, we have
\begin{align}\label{i5}
\int_{T_\ast}^T \int_{\Omega \setminus \mathcal{R}_\ast} \del_t \vv \cdot \phi - (\vv \otimes \vv - \bD(\vv)) : \bD(\phi)  - \vc f \cdot \phi \dd x \dd t = 0
\end{align}
for any solenoidal function $\phi \in L^2(T_\ast, T; W_0^{1,2}(\Omega \setminus \mathcal{R}_\ast))$ with $\bD(\phi)|_{\mathcal{R}_\ast} = 0$. The general existence theory for two-dimensional incompressible Navier-Stokes equations \cite{Ladyzhenskaya1975} now implies that there is a unique function $\vv \in L^\infty(T_\ast, T; L^2(\Omega \setminus \mathcal{R}_\ast)) \cap L^2(T_\ast, T; W_0^{1,2}(\Omega \setminus \mathcal{R}_\ast))$ such that $\div \vv = 0$ and $\del_t \vv \in L^2(T_\ast, T; W^{-1,2}(\Omega \setminus \mathcal{R}_\ast))$ satisfying \eqref{i5}.

In turn, the functions $\vv$ and $\mathcal{R}$ are now defined for all times $t \in [0,T]$ and all $x \in \Omega$. It is left to show that $\vv$ and $\mathcal{R}$ fulfil the momentum equation in its weak form in the whole of $\Omega$. We emphasize that the function $\vv$ inside $\Omega \setminus \mathcal{R}_\ast$ is constructed using test functions from $L^2(T_\ast, T; W_0^{1,2}(\Omega \setminus \mathcal{R}_\ast))$, whereas the momentum equation in the whole of $\Omega$ needs to be valid for all test functions in $L^2(0, T; W_0^{1,2}(\Omega))$.\\

As $\vv = 0$ in $\mathcal{R}(t)$ for any $t \in [T_\ast, T]$, it is enough to check that
\begin{align*}
\int_{T_\ast}^T \int_{\Omega \setminus \mathcal{R}_\ast} \del_t \vv \cdot \phi - (\vv \otimes \vv - \bD(\vv)) : \bD(\phi)  - \vc f \cdot \phi \dd x \dd t = 0
\end{align*}
for any solenoidal $\phi \in L^2(T_\ast, T; W_0^{1,2}(\Omega))$ with $\bD(\phi)|_{\mathcal{R}} = 0$. To this end, we recall a result from \cite[Theorem~2.1]{Starovoitov2003}:
\begin{Lemma}
Let $\mathcal{R} \subset \Omega \subset \R^2$ be connected domains of class $C^{1, \alpha}$ with $\alpha \in [\frac23, 1]$ such that $\dist(\mathcal{R}, \del \Omega) = 0$. If $\phi \in W_0^{1,2}(\Omega)$ is solenoidal and satisfies $\bD(\phi) = 0$ on $\mathcal{R}$, then $\phi = 0$ on $\mathcal{R}$.
\end{Lemma}
The above lemma in combination with \eqref{i5}, the fact that $\vv = \vu$ for any $t \in [0, T_\ast]$, and that $\vu$ is a weak solution to the Navier-Stokes equations with force $\vc f$ given by \eqref{i6} immediately yields the result. Moreover, as $\vv$ obviously differs from $\vu$ when $t>T_\ast$, this finishes the proof of non-uniqueness.

\begin{Remark}
As for the three-dimensional case, we refer to \cite{Sabbagh2018}, where non-uniqueness was proven in a similar way as shown above. The ideas are essentially the same as in \cite{Starovoitov2005}, but the main difference is that the considered geometry consists of a bounded domain with two spherical holes, symmetric to a line on which the obstacle moves, and at a distance such that the body $\ms$ perfectly fits through them. This geometry enables the author to show that the force $\vc f$ from \eqref{i6} is not a mere distribution, but rather an element of $L^2(0,T;L^p(\Omega))$ for some $1 \leq p < 2$.
\end{Remark}

\section{Feedback law and higher regularity of solutions}
We finish this chapter by again focussing on the compressible Navier-Stokes equations as given in Section~\ref{CNSE} in dimension $d=3$, and give three particular examples of modified equations that ensure no collision.

\paragraph{Feedback law.} Following \cite[Section~4]{JNOR2025}, the equations of motion of fluid-structure interaction we consider here are similar as before, the only difference being that we replace the fifth equation in \eqref{Stokes} by
\begin{align*}
m\ddot{\vc G} = -\int_{\del \ms}\big(\bS - p \Id \big)\vc n \dd \sigma + \vc b,
\end{align*}
where $\vc b(t)$ is a feedback control of the form
\begin{equation}\label{feedback}
\vc b(t)= - k_p(\vc G(t) - \vc G_1) - k_d \dot{\vc{G}}(t).
\end{equation}
In control engineering, such a feedback \eqref{feedback} is known as a \emph{proportional-derivative (PD) controller} (see, for instance, \cite{AAKA2021} for a fluid-beam interaction, and the references therein). The feedback $\vc b(t)$ can be thought of being generated by a (massless) spring with spring constant $k_p>0$, and a mechanical damper with damping constant $k_d\geq 0$ connected between the solid's center of mass $\vc G(t)$ and a fixed point $\vc G_1 \in \Omega$. The definition of weak solutions is similar to Definition~\ref{def:wkSol}, the only difference being that the weak formulation of the momentum equation is replaced by 
\begin{align}\label{momentum_weak1}
\begin{split}
&\int_0^\tau \int_\Omega (\rho\vu)\cdot \del_t \phi+ (\rho\vu \otimes \vu) : \bD(\phi) + p(\rho) \div \phi - \bS(\vu) : \bD(\phi) \dd x \dd t\\
&= \int_\Omega \rho(\tau)\vu(\tau) \cdot \phi(\tau) - \vc m_0 \cdot \phi(0) \dd x + \int_0^\tau \vc{b} \cdot \ell_{\phi} \dd t
\end{split}
\end{align}
for any $\phi \in C_c^\infty([0,T) \times \Omega)$ with $\phi(t,x) = \ell_\phi(t) + \omega_\phi(t)(x-\vc G(t))$ near $\ms(t)$, and the energy inequality \eqref{CEISt} is replaced by
\begin{align}\label{StdEI1}
\left[ \int_\Omega \frac12 \rho|\vu|^2 + P(\rho) \dd x \right]_{t=0}^{t=\tau} + \int_0^\tau \int_\Omega \bS: \bD(\vu) \dd x \dd t \leq \int_0^\tau \vc b \cdot \dot{\vc G} \dd t.
\end{align}

The existence of weak solutions can be established by following \cite[Theorem~4.1]{Feireisl2004}, and the existence of strong solutions can be found in \cite[Theorem 1.1]{RoyTakahashi2021}. Combining the energy estimate \eqref{StdEI1} with the feedback law \eqref{feedback}, we obtain the following no-collision result:
\begin{Proposition}\label{energyest}
Let $\vc{G}_1 \in \Omega$ with $\operatorname{dist}(\vc{G}_1,\del \Omega)>1$ and assume that $\vc b$ satisfies the feedback law \eqref{feedback}. Let  $(\rho,\vu,\vc G)$ be a weak solution. Then
\begin{align}\label{energy estimate full system}
\begin{split}
&\int_{\mf(\tau)} \frac{1}{2} \rho |\vu|^2 + P(\rho) \dd x + \frac{m}{2}|\dot{\vc{G}}|^2 + \frac12 \mathbb{J} \omega\cdot\omega + \frac{k_p}{2}|\vc{G}(\tau) - \vc{G}_1|^2 + {k_d}\int_0^\tau |\dot{\vc G}|^2 \dd t \\
&+ \int_0^\tau \int_{\mf(\tau)} \bS:\bD(\vu) \dd x \dd t\\
&\leq \int_{\mf(0)} \frac{|\vc m_0|^2}{2\rho_0} + P(\rho_0) \dd x + \frac{m}{2}|\vc V_0|^2 + \frac12 \mathbb{J} \omega_0 \cdot \omega_0 + \frac{k_p}{2}|\vc G_1|^2.
\end{split}
\end{align}
Moreover, if there exists $\delta \in (0, \dist(\vc G_1, \del \Omega) - 1)$ such that
\begin{align}\label{delta1}
\frac{2}{k_p} \left( \int_{\mf(0)} \frac{|\vc m_0|^2}{2 \rho_0} + P(\rho_0) \dd x + \frac{m}{2}|\vc G_0|^2 + \frac12 \mathbb{J} \omega_0 \cdot \omega_0 + \frac{k_p}{2}|\vc{G}_1-\vc G_0|^2 \right) \leq \delta^2,
\end{align}
then there exists $\eps = \eps(\delta) > 0$ such that
\begin{equation}\label{noc}
\dist(\vc G(t),\partial\Omega) \geq 1+\varepsilon \quad \forall t \geq 0.
\end{equation}
\end{Proposition}
\begin{proof}
As $\vc b(t)= - k_p(\vc G(t) - \vc G_1)-k_d \dot{\vc G}(t)$, we obtain 
\begin{equation}\label{feeden}
-\vc b \cdot \dot{\vc G} = k_p(\vc G(t) - \vc G_1)\cdot \dot{\vc G}(t) + k_d|\dot{\vc G}(t)|^2 = \frac{\rd}{\rd t}\left(\frac{k_p}{2} |\vc G(t)-\vc G_1|^2\right) + k_d|\dot{\vc G}(t)|^2.
\end{equation}
Inserting \eqref{feeden} into the energy inequality \eqref{StdEI1}, and using the extensions of $\rho,\vu$ as in \eqref{extended}, we conclude \eqref{energy estimate full system}.\\

In order to establish \eqref{noc}, we use \eqref{energy estimate full system} and \eqref{delta1} to obtain
\begin{align}\label{energy estimate full system1}
|\vc G(t)-\vc G_1|^2  
\leq \frac{2}{k_p}\left(\int_{\mf(0)} \frac{|\vc m_0|^2}{2\rho_0} + P(\rho_0) \dd x + \frac{m}{2}|\vc{V}_0|^2 + \frac12 \mathbb{J}\omega_0\cdot \omega_0+\frac{k_p}{2}|\vc G_0-\vc G_1|^2\right) \leq \delta^2.
\end{align}
Finally, we conclude \eqref{noc} by
\begin{align*}
\dist(\vc G(t), \del \Omega) \geq \dist(\vc G_1, \del \Omega) - |\vc G(t) - \vc G_1| \geq 1 + \underset{=: \eps > 0}{\underbrace{[\dist(\vc G_1, \del \Omega) - 1 - \delta ]}}.
\end{align*}
\end{proof}

\begin{Remark}
Another way how to interpret \eqref{delta1} is as follows: If the initial center of mass of the solid is close enough to the given point $\vc G_1$, and the spring is stiff enough, meaning $k_p \gg 1$, then, regardless of its mass and the initial fluid's energy, the solid stays close to $\vc G_1$ for all times.
\end{Remark}

\paragraph{Besov spaces.} A little more advanced as the situation above is the fact that if we allow for higher regularity, then we can obtain the no-collision result even \emph{without} the external PD-controller. To this end, for $k \in \mathbb{N}$ and every $0<s<k$, $1\leq p,q <\infty$, we define the \emph{Besov space} $B^s_{q,p}(\Omega)$ by real interpolation of Sobolev spaces as $B^s_{q,p}(\Omega)=(L^q(\Omega),W^{k,q}(\Omega))_{s/k,p}$. Besov spaces are somewhat related to H\"older spaces in the setting of Lebesgue and Sobolev spaces and can measure the ``smoothness'' of a function $f \in L^q(\Omega)$. There exist several equivalent definitions of these spaces; one of them is the following (see \cite{DevoreSharpley1993}):
\begin{Definition}
Let $\Omega \subset \R^d$ be an open set, $1 \leq q<\infty$, and $f \in L^q(\Omega)$. Let $\Id(f) = f$ be the identity operator, and denote for any $h \in \R^d$ by $\tau_h(f,x) = f(x+h)$ the shift operator. Define
\begin{align*}
\bigtriangleup_h^k (f,x,\Omega) &= \begin{cases}
[\tau_h - \mathbb{I}]^k(f,x) & \text{if } x+jh \in \Omega, \ j \in \{0,...,k\},\\
0 & \text{else}\, ;
\end{cases} \\
\omega_{k,q}(f,t,\Omega) &= \sup_{|h| \leq t} \|\bigtriangleup_h^k(f, \cdot, \Omega)\|_{L^q(\Omega)}.
\end{align*}
For $0<s<k$ and $1 \leq p<\infty$, the Besov space $B_{q,p}^s(\Omega)$ is defined as all functions $f \in L^q(\Omega)$ such that
\begin{align*}
|f|_{B_{q,p}^s(\Omega)}^p = \int_0^1 \left| \frac{\omega_{k,q}(f,t,\Omega)}{t^s} \right|^p \frac{{\rm d} t}{t} < \infty.
\end{align*}
It becomes a Banach space if we equip it with the norm $\|\cdot\|_{B_{q,p}^s(\Omega)} = \|\cdot\|_{L^q(\Omega)} + |\cdot|_{B_{q,p}^s(\Omega)}$.
\end{Definition}
We will not give further properties of these spaces since it would go beyond the scope of this chapter, so we refer to \cite{DevoreSharpley1993,Triebel2006} for a detailed presentation of Besov spaces. With this at hand, as in \cite[Theorem 1.2]{HMTT2019}, we can establish the following result for a smooth rigid body $\ms(t)$:
\begin{Theorem}
Let $2<p<\infty$, $3<q<\infty$ satisfy the condition $\frac{1}{p}+\frac{1}{2q}\neq \frac{1}{2}$. Let the initial data satisfy 
\begin{equation*}
\rho_0 \in W^{1,q}(\mf(0)),\quad \vu_0\in B^{2(1-1/p)}_{q,p}(\mf(0)), \quad \inf_{\mf(0)} \rho_0 > 0,
\end{equation*}
\begin{equation*}
\vc{G}_0\in \R^3,\quad \vc{V}_0\in \R^3,\quad \omega_0\in \R^3,
\end{equation*}
\begin{equation*}
\frac{1}{|\mf(0)|}\int_{\mf(0)}\rho_0 = \bar{\rho}>0 ,
\end{equation*}
\begin{equation*}
\vu_0=0 \mbox{ on } \del \Omega,\quad \vu_0=\vc{V}_0 + \omega_0\times (y-\vc{G}_0) \mbox{ on } \del \ms(0).
\end{equation*}
If there exists $\delta > 0$ and $\varepsilon >0$ such that
\begin{align*}
\|(\rho_0-\bar{\rho},\vu_0,\vc{V}_0, \omega_0)\|_{W^{1,q}\times B^{2(1-1/p)}_{q,p} \times \R^3 \times \R^3}\leq \delta,\quad \operatorname{dist}(\ms(0),\del\Omega)\geq \varepsilon >0,
\end{align*}
then
\begin{align*}
\dist(\ms(t),\del\Omega)\geq \frac{\varepsilon}{2}\mbox{ for all } t\in [0,\infty).
\end{align*}
\end{Theorem}

\paragraph{Smoothies.} As the last example in this chapter we mention the result for so-called multi-polar or $k$-fluids as described in \cite{FeireislNecasova2008}. These fluids are characterized by a higher-order stress tensor, hence giving higher regularity to the fluid's velocity and, in turn, to the density. More precisely, for $k \in \mathbb{N}$, the stress tensor takes the form
\begin{align*}
\bS(\nabla \vu) = \sum_{j=0}^{k-1} (-1)^j \Delta^j \bigg[ \mu_j \bigg( \nabla \vu + \nabla^T \vu - \frac23 \div \vu \Id \bigg) + \eta_j \div \vu \Id \bigg],
\end{align*}
where $\mu_j, \eta_j \geq 0$, $\mu_{k-1} > 0$ are the (constant) viscosity coefficients, and $\Delta^j$ is the standard Laplacian applied $j$ times. In this context, classical Newtonian fluids are mono-polar, that is, $k=1$ in the above. Again by the properties of Riesz transform \eqref{Riesz}, one shall expect that the velocity field $\vu$ is very regular; in particular, $\vu \in L^2(0,T;W_0^{k, 2}(\Omega))$. The density then can be shown to be bounded away from zero and infinity as long as the initial density is; hence, these types of fluids behave roughly as incompressible ones, and indeed also $\vu \in L^\infty(0,T;L^2(\Omega))$ and $\rho \in L^\infty((0,T) \times \Omega) \cap C([0,T]; L^1(\Omega))$ is bounded away from zero.

Since $\vu \in L^2(0,T;W_0^{k, 2}(\Omega))$, by Sobolev embedding \eqref{SobEmb} one a priori has $\vu \in L^2(0,T;C_0^1(\overline{\Omega}))$ as long as $k \geq 3$; in particular, the streamlines given by the unique solution of
\begin{align*}
\frac{\rd}{\rd t} \vc X(t, x) = \vu(t, \vc X(t, x)), \quad t>0, \quad \vc X(0, x) = x,
\end{align*}
are well defined and no collision can occur.



\chapter{Other models}\label{ch:8}
Having investigated several (in)compressible systems in two and three dimensions for smooth and non-smooth bodies, we review here the (no-)collision results for several other models:
\begin{itemize}
\item Ideal incompressible fluids: Here, the viscosity is zero, and collision happens with non-zero speed.
\item One-dimensional models: We will show that points never collide.
\item Slip boundary conditions: This is a kind of roughness allowing for collision.
\item Tresca's boundary condition: It is a variant of slip conditions, also allowing the solid to touch the container's wall.
\end{itemize}

Without being exhaustive, we will just focus on the main ideas in the references given for the specific problem. The interested reader might consult them to get detailed versions of the statements and proofs.

\section{Euler fluids: a perfect flow}
Let us start with the description of incompressible, inviscid, and irrotational fluids and the collision result given in \cite{HouotMunnier2008}. In their paper, however, no real PDE is stated; the whole procedure rather relies on local coordinates, the physical principle of least action, and the resulting Lagrangian. For this reason, let us recall how one can get from the common PDE formulation to the Lagrangian one, where we follow the presentation of \cite{FarazmandSerra2018}.\\

First, the incompressible Euler\footnote{after Leonhard Euler (1707--1783)} equations are given by the system
\begin{align}\label{Euler}
\begin{cases}
\div \vu = 0 & \text{in } (0,T) \times \Omega,\\
\rho_\mf (\del_t \vu + \div(\vu \otimes \vu) ) + \nabla p = 0 & \text{in } (0,T) \times \Omega,\\
\vu \cdot \vc n = 0 &\text{on } (0,T) \times \del \Omega,
\end{cases}
\end{align}
together with the usual coupling equations for the solid inside the fluid (with $\bS = 0$). The streamlines satisfy
\begin{align*}
\dot{\vc X}(t, x_0) = \vu(t, \vc X(t, x_0)), \quad \vc X(0, x_0)=x_0,
\end{align*}
for any point $x_0 \in \Omega$, where $\dot{\vc X} = \frac{\rd}{\rd t} \vc X$. Now, as the flow is incompressible (and hence volume preserving), we have
\begin{align*}
\det (\nabla_{x_0} \vc X(t, x_0)) = 1 \quad \forall (t, x_0) \in (0,T) \times \Omega.
\end{align*}
Further, if $\vu$ is a solution to \eqref{Euler}, we get
\begin{align*}
\rho_\mf \ddot{\vc X}(t, x_0) &= \rho_\mf \frac{\rd}{\rd t} \vu(t, \vc X(t,x_0)) = \rho_\mf \big( \del_t \vu(t, \vc X(t, x_0)) + \nabla_x \vu(t, \vc X(t, x_0)) \cdot \dot{\vc X}(t, x_0) \big)\\
&= -\nabla_x p(t, \vc X(t, x_0)).
\end{align*}
This is nothing else than the Euler-Lagrange equations for the Lagrangian
\begin{align*}
\mathcal{L}(t, \vc X, \dot{\vc X}) = \frac12 \rho_\mf |\dot{\vc X}|^2 - p(t,\vc X).
\end{align*}
Indeed, by least action principle, we find
\begin{align}\label{i7}
0 = \frac{\rd}{\rd t} \frac{\del \mathcal{L}}{\del \dot{\vc X}} - \frac{\del \mathcal{L}}{\del {\vc X}} = \rho_\mf \frac{\rd}{\rd t} \dot{\vc X} - \frac{\del}{\del {\vc X}}(-p) = \rho_\mf \ddot{\vc X} + \nabla p.
\end{align}
Recall also that the Lagrangian $\mathcal{L}$ is physically the difference of kinetic energy (here, $\frac12 \rho_\mf |\vu|^2 = \frac12 \rho_\mf |\dot{\vc X}|^2$) and potential energy (here, the pressure $p$). The total energy, the so-called Hamiltonian $\mathcal{H}$, is instead the \emph{sum} of both, meaning $\mathcal{H} = \frac12 \rho_\mf |\dot{\vc X}|^2 + p = \frac12 \rho_\mf |\vu|^2 + p$; note that this totally coincides with the energy balances derived in previous chapters. Lastly, the initial conditions read
\begin{align}\label{i8}
\vc X(0, x_0) = x_0, \quad \dot{\vc X}(0, x_0) = \vu(0, x_0) = \vu_0(x_0).
\end{align}

On the other hand, given a one-parameter family $\vc X(t, x_0)$ and a function $p(t, x)$ satisfying \eqref{i7} and \eqref{i8}, we define the flow as $\vu(t,x)=\vu(t, \vc X(t, x_0)) = \dot{\vc X}(t, x_0)$. Following the lines above, we can show that the so constructed function is indeed a solution to \eqref{Euler}. This shows that the PDE formulation and the Lagrangian formulation coincide (for a more detailed discussion and references, see \cite{FarazmandSerra2018}).

\begin{Remark}
If the force density on the right hand side of \eqref{Euler}$_2$ is not zero, but conservative, meaning of the form $\vc f = \nabla f$ for some function $f=f(t, x)$, the Lagrangian reads
\begin{align*}
\mathcal{L}(t, \vc X, \dot{\vc X}) = \frac12 \rho_\mf |\dot{\vc X}|^2 - p(t, \vc X) + f(t, \vc X),
\end{align*}
which can be verified by the same arguments. In particular, for gravitational force $\vc f = -g \vc e_3 = \nabla[x \mapsto -g x_3]$, we can write the Euler equations as a Lagrangian flow.
\end{Remark}

Let us come to the collision result from \cite[Section~5]{HouotMunnier2008}. The considered configuration is an infinite cylinder of radius $1$ falling in half-space. By symmetry, the whole configuration reduces to a two-dimensional system, where a ball $\ms$ of radius $1$ is falling in a half-plane $\Omega = \R \times (0, \infty)$. As before, we denote by $h>0$ the distance of $\ms$ to $\del \Omega$.\\

For the sake of fixing notations, we set
\begin{align*}
\Omega = \{(x_1, x_2) \in \R^2: x_2>0\}, \quad \ms(t)=B_1(\vc G(t)), \quad \vc G(t) = (0, 1+h(t)), \quad \vu|_{\ms} = \dot{\vc G} = (0, \dot{h}).
\end{align*}
Note that this means we consider a spherical solid moving vertically along the axis $x_1=0$ with speed $\dot{\vc G}$ and zero angular velocity. Moreover, as the fluid flow is irrotational and hence ${\rm curl}(\vu) = 0$, we may find a potential $\phi$ such that $\vu|_{\mf} = \nabla \phi$. Moreover, by the movement of the solid, this potential is linear in $\dot h$; in particular, $\phi = \dot h \varphi$, where $\varphi$ solves
\begin{align}\label{i9}
\begin{cases}
-\Delta \varphi = 0 & \text{in } \ms(t)\\
\nabla \varphi \cdot \vc n = 1+h - x_2 & \text{on } \del \ms(t),\\
\nabla \varphi \cdot \vc n = 0 & \text{on } \del \Omega,
\end{cases}
\end{align}
where we used that $\vc n = -r \vc e_1 + (1+h-x_2) \vc e_2$ on $\del \ms$ (see also Figure~\ref{fig:82} below). In order to verify that collision may occur, we will solve system \eqref{i9} explicitly by using conformal mappings and transforming the whole configuration to concentric circles; this is essentially the same idea as we used before in Chapter~\ref{ch:7}. As it is notationally convenient for us, we will not distinguish between the complex variable $z = z_1 + i z_2 \in \mathbb{C}$ and the vector $\vc z = (z_1, z_2) \in \R^2$. In the same spirit, we have $\nabla_z f(z) = f'(z)$ for a function $f$ defined on $\R^2 \sim \mathbb{C}$.\\

We begin by transforming the whole configuration with two simple lemmata from complex analysis, the proofs of which can be found in \cite[§§ 5.4 and 5.6]{Henrici1988}:
\begin{Lemma}\label{lem:konf1}
Let $h>0$ and
\begin{align*}
k(z)=\frac{z-ia}{z+ia}, \quad a=\sqrt{(1+h)^2-1}.
\end{align*}
Then, $k$ is a conformal map from $\Omega = \R \times (0, \infty)$ to $B_1(0)$. Moreover, for $\ms = B_1(0)$ and $A_\sigma = B_1(0) \setminus \overline{B_\sigma(0)}$ with $\sigma = (1+h+a)^{-1}$, we have the following properties:
\begin{align*}
k(\ms) = B_\sigma(0), \quad k(\Omega \setminus \ms) = A_\sigma, \quad k(\del \ms) = \del B_\sigma(0), \quad k(\del \Omega) = \del B_1(0) \setminus \{1\}.
\end{align*}
As for the inverse mapping, we have
\begin{align*}
k^{-1}(w) = ia \frac{1+w}{1-w}, \quad 1+h=\frac{\sigma^{-1} + \sigma}{2}, \quad a=\frac{\sigma^{-1}-\sigma}{2}.
\end{align*}
\end{Lemma}

\begin{Lemma}\label{lem:konf2}
Let $U, V \subset \R^2$ be open and $l =l_1 + il_2 : U \to V$ be a conformal map.
\begin{enumerate}
\item Let $f \in C^2(V)$ and $F = f \circ l$. Then $F \in C^2(U)$ and for any $w = w_1 + iw_2 \in U$,
\begin{align*}
\Delta F(w) = |l'(w)|^2 \Delta f(l(w)),
\end{align*}
where $l'(w) = \frac{\rd}{\rd w} l(w)$ is the complex derivative of $l$.

\item Let $\gamma \subset U$ be a $C^1$ curve and $\vc n$ be the unit normal on $\gamma$, and let $l:U \to V$ be a conformal map. Similarly, let $\tilde{\gamma} = l(\gamma) \subset V$ and $\tilde{\vc n} = l(\vc n)$. Then, for all $w=w_1+iw_2 \in U$,
\begin{align*}
\nabla F(w_1, w_2) \cdot \vc n = \bigg| \frac{\del l}{\del w_1}(w_1, w_2) \bigg| \ \nabla f(l(w_1, w_2)) \cdot \tilde{\vc n}.
\end{align*}
\end{enumerate}
\end{Lemma}

As the map $k$ from Lemma~\ref{lem:konf1} is conformal, its inverse $l=k^{-1}$ is as well. Thus, applying Lemma~\ref{lem:konf2} with $U=A_\sigma$ and $V=\Omega$ to problem \eqref{i9} and defining $\zeta(z) = \varphi(k^{-1}(z))$, we find
\begin{align*}
\nabla \zeta \cdot \vc n |_{\del B_\sigma(0)} &= \big| \del_{w_1} k^{-1}(w) \big| \nabla \varphi \cdot \vc n |_{\del \ms(t)} = \frac{2a}{|1-w|^2} (1+h-x_2) \circ k^{-1}(w) \\
&= \frac{2a (1+h-a\frac{1-\sigma^2}{|1-w|^2})}{|1-w|^2} = \frac{2a (2\sigma^2 - (1+\sigma^2)w_1)}{\sigma (1+\sigma^2 - 2w_1)^2}.
\end{align*}
Since clearly $\nabla \zeta \cdot \vc n |_{\del B_1(0)} = 0$, the system for $\zeta$ reads\footnote{We remark that in \cite{HouotMunnier2008}, a factor $2$ is missing in the boundary condition on $\del B_\sigma(0)$.}
\begin{align}\label{i10}
\begin{cases}
-\Delta \zeta = 0 & \text{in } A_\sigma,\\
\nabla \zeta \cdot \vc n = \frac{2a(2 \sigma^2 - (1+\sigma^2) w_1)}{\sigma (1+\sigma^2-2w_1)^2} & \text{for } w=w_1 + iw_2 \in \del B_\sigma(0),\\
\nabla \zeta \cdot \vc n = 0 & \text{on } \del B_1(0).
\end{cases}
\end{align}

As $\zeta$ is harmonic, there exists a harmonic complex conjugate $\xi$ and the function $F=\zeta + i \xi$ is holomorphic in $\Omega \setminus \ms$ and defined up to an additive constant (this is because $\zeta$ solves a Neumann boundary problem). Without loss of generality, we may fix this constant to be equal to zero. The boundary conditions for $F$ follow from those for $\zeta$: on $|w|=1$, we have $\vc n = w$ and hence $\nabla F \cdot \vc n = F'(w)w$. On $|w|=\sigma$, the normal is given by $\vc n = -w/\sigma$, thus $-\sigma \nabla F \cdot \vc n = F'(w)w$. Finally, writing $w = w_1 + iw_2$, \eqref{i10}$_2$ and \eqref{i10}$_3$ give us
\begin{align*}
\begin{cases}
\Re[F'(w)w] = -\frac{2a(2\sigma^2 - (1+\sigma^2)w_1)}{(1+\sigma^2-2w_1)^2} & \text{if } |w|=\sigma,\\
\Re[F'(w)w] = 0 & \text{if } |w|=1,
\end{cases}
\end{align*}
where $\Re[w_1+iw_2]=w_1$ is the real part of $w$.\\

Since $F$ is holomorphic in $A_\sigma$, it admits a Laurent series expansion. More precisely, we have:
\begin{Lemma}\label{lem84}
For any $w \in A_\sigma$, the function $F$ has Laurent series expansion
\begin{align*}
F(w) = -2a \sum_{n \geq 1} \frac{\sigma^{2n}}{1-\sigma^{2n}} (w^n + w^{-n}).
\end{align*}
\end{Lemma}
\begin{Remark}
Let us note that the radius of convergence of this series is $\sigma^2 < |w| < \sigma^{-2}$, which by $\sigma<1$ is a superset of $A_\sigma$.
\end{Remark}
\begin{proof}[Proof of Lemma~\ref{lem84}]
We start in writing
\begin{align*}
F(w) = \sum_{n \geq 1} (a_n + ib_n) w^n + (c_n + id_n) w^{-n}.
\end{align*}
This gives for any $w = r e^{i \theta}$
\begin{align*}
\Re[F'(w)w] = \sum_{n \geq 1} n[r^n a_n - r^{-n} c_n] \cos(n \theta) - n [r^n b_n + r^{-n} d_n] \sin(n \theta).
\end{align*}
Accordingly, from the boundary conditions of $F'(w)w$, for $r=1$ we find $a_n=c_n$ and $b_n = -d_n$, hence
\begin{align*}
\Re[F'(w)w] = \sum_{n \geq 1} n [r^n - r^{-n}] (a_n \cos(n\theta) - b_n \sin(n \theta)).
\end{align*}
For $r=\sigma$, by observing $w_1 = \sigma \cos(\theta)$ we define
\begin{align*}
\beta(\theta) = -\frac{2a\sigma (2\sigma-(1+\sigma^2)\cos(\theta))}{(1+\sigma^2-2\sigma \cos(\theta))^2}.
\end{align*}
Thus,
\begin{align*}
\beta(\theta) = \sum_{n \geq 1} n [\sigma^n - \sigma^{-n}] (a_n \cos(n\theta) - b_n \sin(n \theta)).
\end{align*}
As $\beta$ is an even function, we immediately have $b_n=0$ for all $n \geq 1$. Moreover, we find
\begin{align*}
\int_0^{2\pi} \beta(\theta) \cos(n \theta) \dd \theta = n a_n [\sigma^n - \sigma^{-n}] \int_0^{2\pi} \cos^2(n \theta) \dd \theta = \pi n a_n [\sigma^n - \sigma^{-n}].
\end{align*}
Therefore,
\begin{align*}
a_n &= -\frac{\sigma^n}{n(1-\sigma^{2n})} \frac{1}{\pi} \int_0^{2\pi} \beta(\theta) \cos(n\theta) \dd \theta = -\frac{\sigma^n}{n(1-\sigma^{2n})} \Re\bigg[ \frac{1}{\pi} \int_0^{2\pi} \beta(\theta) e^{-in \theta} \dd \theta \bigg]\\
&= -\frac{\sigma^n}{n(1-\sigma^{2n})} \Re[\hat{\beta}(n)],
\end{align*}
where $\hat{\beta}(n)$ is the Fourier coefficient of $\beta$ (note that the seemingly ``missing'' factor 2 is already included in $a_n$, since we consider a one-sided rather than a two-sided series expansion). It is now sufficient to calculate the Fourier coefficients of $\beta$ explicitly. To this end, for $z= e^{i\theta}$, we have
\begin{align*}
\hat{\beta}(n) = \frac{1}{2\pi i} \int_{\del B_1(0)} \frac{-2a\sigma(4\sigma z - (1+\sigma^2)(z^2+1))}{((1+\sigma^2)z-\sigma(z^2+1))^2} \frac{\rd z}{z^n}.
\end{align*}
As we may write
\begin{align*}
\frac{-2a\sigma(4\sigma z - (1+\sigma^2)(z^2+1))}{((1+\sigma^2)z-\sigma(z^2+1))^2} = 2a\sigma \bigg[ \frac{1}{(z-\sigma)^2} + \frac{1}{\sigma^2} \frac{1}{(z - \sigma^{-1})^2} \bigg],
\end{align*}
by virtue of the residue theorem, we get
\begin{align*}
\hat{\beta}(n) = 2a n \sigma^n \quad \forall n \geq 1.
\end{align*}
Plugging this expression in the one for $a_n$, we complete the proof.
\end{proof}

With the series expansion of $F$ at hand, we can compute the kinetic energy of the fluid:
\begin{Lemma}
The kinetic energy of the fluid inside $\Omega \setminus \ms$ is given by $E_f = \frac12 \rho_\mf |\dot h|^2 \mathcal{E}(\sigma)$, where
\begin{align*}
\mathcal{E}(\sigma) = 4 \pi a^2 \sum_{n \geq 1} n \sigma^{2n} \frac{1+\sigma^{2n}}{1-\sigma^{2n}}.
\end{align*}
Moreover, $\mathcal{E}(\sigma)$ is continuous on $(0,1)$ with finite strictly positive limits as $\sigma \to 0$ and $\sigma \to 1$, respectively.
\end{Lemma}
\begin{proof}
We will just show the main steps of this lemma. By change of variables, we have
\begin{align*}
E_f = \frac12 \rho_\mf |\dot h|^2 \int_{\Omega \setminus \ms} |\nabla \varphi|^2 \dd x = \frac12 \rho_\mf |\dot h|^2 \int_{A_\sigma} |\nabla \zeta|^2 \dd w,
\end{align*}
where $\varphi$ is the potential of the velocity, and $\zeta$ is as above. Since $|\nabla \zeta|^2 = |F'(w)|^2$, we get
\begin{align*}
\mathcal{E}(\sigma) = \int_{A_\sigma} |\nabla \zeta|^2 \dd w = \int_\sigma^1 \int_0^{2\pi} |F'(\rho \cos \theta, \rho \sin \theta)|^2 \rho \dd \theta \dd \rho.
\end{align*}
We wish to apply Parseval's theorem, for which we have to proof that for any $\rho \in [\sigma, 1]$, the function $\theta \mapsto F'(\rho \cos \theta, \rho \sin \theta)$ is an element of $L^2(0, 2\pi)$. Indeed, a simple calculation shows
\begin{align*}
|F'(\rho \cos \theta, \rho \sin \theta)| \leq \frac{2a}{(1-\sigma)^3}.
\end{align*}
Thus, Parseval's theorem gives
\begin{align*}
\mathcal{E}(\sigma) = \pi \int_\sigma^1 \sum_{n \geq 1} \bigg[ \frac{2 a n \sigma^{2n}}{1-\sigma^{2n}} \bigg]^2 [\rho^{2n-2} + \rho^{-2n-2}] \rho \dd \rho = 4 \pi a^2 \sum_{n \geq 1} n \sigma^{2n} \frac{1+\sigma^{2n}}{1-\sigma^{2n}}.
\end{align*}

Some simple consequences from the series expansions of $(1-x)^{-1}$ and $(1-x)^{-2}$, and the fact that $a = (\sigma^{-1} - \sigma)/2$ then give another representation for $\mathcal{E}(\sigma)$ as
\begin{align*}
\mathcal{E}(\sigma) = 2\pi \bigg( \sum_{n \geq 1} \sigma^{2n-2} \bigg[ \frac{1-\sigma^2}{1-\sigma^{2n}} \bigg]^2 \bigg) - \pi.
\end{align*}
Lebesgue's theorem now shows that indeed $\mathcal{E}(\sigma)$ is continuous on $(0,1)$ with limit $\lim_{\sigma \to 0} \mathcal{E}(\sigma) = \pi$. As the summand tends to $1/n^2$ as $\sigma \to 1$, we have also that $\mathcal{E}(\sigma)$ is continuous on $[0,1]$ with value $\mathcal{E}(1) = \frac{\pi^3}{3} - \pi$.
\end{proof}

Let us finally show how the calculations above yield the collision result. As $\sigma = \sigma(h)$, with a slight abuse of notation, the Lagrangian of our system under consideration is simply its kinetic energy, given by\footnote{Compared to the Lagrangian discussed previously, here we can neglect the pressure $p$ since it just plays the role of a Lagrange multiplier for the incompressibility condition $\div \vu = 0$.}
\begin{align*}
\mathcal{L} = \frac12 m |\dot h|^2 + \frac12 \rho_\mf |\dot h|^2 \mathcal{E}(h).
\end{align*}
Hence, the least action principle yields
\begin{align*}
0=\frac{\rd}{\rd t} \frac{\del \mathcal{L}}{\del \dot h} - \frac{\del \mathcal{L}}{\del h} = \ddot{h}(m + \rho_\mf\mathcal{E}(h)) + \frac12 \rho_\mf |\dot{h}|^2 \frac{\del \mathcal{E}}{\del h}(h).
\end{align*}
Multiplying by $2\dot h$ gives $\frac{\rd}{\rd t} [|\dot h|^2 (m + \rho_\mf \mathcal{E}(h))] = 0$, this is,
\begin{align}\label{Euler-h-dot}
\dot{h}=\dot{h}_0 \sqrt{\frac{m + \rho_\mf \mathcal{E}(h_0)}{m + \rho_\mf \mathcal{E}(h)}}.
\end{align}
(Note that the negative solution is indeed ruled out, since this would imply $\dot{h}(0)=-\dot{h}_0$.) The above equation ensures that $\dot h$ keeps the same sign as $\dot{h}_0$. If we choose $\dot{h}_0<0$ and $h_0>0$, then this ensures $0 \leq h \leq h_0$. As $1+h=(\sigma^{-1}+\sigma)/2$ and $\mathcal{E}(h)$ is right-continuous as $h \to 0$ (that is, $\sigma \to 1$), there exists $h_0>0$ such that $\mathcal{E}(h) \leq \frac32 \mathcal{E}(0)$ for all $0 \leq h < h_0$. Hence, for any $\dot{h}_0<0$ and any $0 \leq h < h_0$,
\begin{align*}
h(t) \leq h_0 + \dot{h}_0 t \sqrt{\frac{m}{m + \frac32 \rho_\mf \mathcal{E}(0)}}.
\end{align*}
As always $h \geq 0$, this shows that there exists some finite $T=T(m, h_0, \dot{h}_0)<\infty$ such that $h(t) \to 0$ as $t \to T$, meaning collision happens in finite time and, additionally, with non-zero speed.

\begin{Remark} 
We remark that this collision result is in contrast to Theorem~\ref{thm:Starov}, where collision occurs indeed with zero velocity. This shows that viscosity has some strong damping effect on the solid's speed. Further investigation of \eqref{Euler-h-dot} as made in \cite[Section~5.2]{HouotMunnier2008} shows that for Euler fluids, there is also a damping effect, although not that strong. As a matter of fact, for the foregoing situation such that $m = \pi \rho_\ms$, even if the density ratio $\rho_\mf / \rho_\ms \to \infty$, the damping in an Euler fluid is not higher than $1- (\pi^2/3 - 1)^{-1/2} \approx 34 \%$, meaning that the solid collides with a speed that is at least $66 \%$ of its initial velocity.
\end{Remark}

\section{On a line}
After investigation of inviscid fluids, let us come to a much simpler model of one-dimensional flows as described in \cite{VazquezZuazua2006}. Here, we consider $N$ particles (points) on the real line having positions $h_i(t)$ with
\begin{align*}
- \infty < h_1(t) < h_2(t) < ... < h_N(t) < \infty
\end{align*}
for some $0<t<T$. The system under consideration is the 1D Navier-Stokes system
\begin{align}\label{1DNSE}
\begin{cases}
\del_t u + \kappa \del_x u^2 - \del_{xx}^2 u = 0 & \text{for } x \in I_i(t), \ i \in \{0,..., N\}, \ t>0\\
\dot h_i(t) = u(t, h_i(t)) & \text{for } i \in \{1,...,N\}, \ t>0,\\
m_i \ddot{h}_i(t) = [\del_x u](t, h_i(t)) & \text{for } i \in \{1,...,N\}, \ t>0,\\
u(0,x) = u_0(x) & \text{for } x \in \R,\\
h_i(0) = h_{i,0}, \ \dot{h}_i(0) = \dot{h}_{i,0} & \text{for } i \in \{1,...,N\}.
\end{cases}
\end{align}
Here, we denoted by $I_i(t)$ the intervals occupied by the fluid and separated by the particles, given as
\begin{align*}
I_0(t) = (-\infty, h_1(t)), \quad I_i(t) = (h_i(t), h_{i+1}(t)), \ i \in \{1,...,N-1\}, \quad I_N(t) = (h_N(t), \infty);
\end{align*}
see Figure~\ref{fig:1D} for the main notations.

\begin{figure}[H]
\centering
\begin{tikzpicture}
\draw (0,0) -- (10,0);
\foreach \a in {2,4,6,8}
\filldraw (\a,0) circle (1pt);
\foreach \a in {1,2,3,4}
\node at (2*\a, 0) [anchor=north] {$h_{\a}(t)$};
\foreach \a in {0,1,...,4}
	\node at (2*\a+1, 0) [anchor=south] {$I_{\a}(t)$};
\end{tikzpicture}
\caption{The configuration for $N=4$.}
\label{fig:1D}
\end{figure}

The mass of particle $i$ is $m_i>0$, and the real number $\kappa$ is the ratio between convection and diffusion (in processes involving heat, this is also called the Pecl\'et number). Lastly, the jump of a function $f$ is defined as
\begin{align*}
[f](x) = \lim_{s \to 0} [f(x+s) - f(x-s)].
\end{align*}
As written in \cite[Section~1]{VazquezZuazua2006}, ``according to the transmission conditions satisfied at the point mass locations $x = h_i(t), \ i = 1, ... ,N$, the velocities of the fluid and the particles coincide, and each particle is accelerated by the difference of the velocity gradient on both sides of it. Thus, the velocity gradient acts as a pressure.''\\

The main theorem of this section now reads:
\begin{Theorem}[{\cite[Theorem~2.1]{VazquezZuazua2006}}]\label{thm:1D}
Let $u_0 \in W^{1,2}(\R)$, $m_i>0$, and $h_{i,0}, \dot{h}_{i,0} \in \R$ such that
\begin{align*}
-\infty < h_{1,0} < h_{2,0} < ... < h_{N, 0} < \infty,
\end{align*}
together with the compatibility conditions $u_0(h_{i, 0}) = \dot{h}_{i, 0}$ for any $i \in \{1,...,N\}$. Then, there exists a unique global solution $(u, h_1,...,h_N)$ to system \eqref{1DNSE} with
\begin{align*}
u \in C([0,\infty); W^{1,2}(\R)), \quad \del_{xx}^2 u \in L^2((0,T) \times I_i(t)), \quad h_i \in C([0, \infty)), \quad \ddot{h}_i \in L^2(0,T),
\end{align*}
for any finite time $T>0$. Moreover, for any $t>0$,
\begin{align*}
- \infty < h_1(t) < h_2(t) < ... < h_N(t) < \infty,
\end{align*}
meaning there is no collision in finite time.
\end{Theorem}

For the existence proof of a strong solution, we refer to the original paper \cite{VazquezZuazua2006}. Instead, we will just show how the regularity properties of $u$ imply the no-collision result. Indeed, similar to before, the dynamics of the particles are given by
\begin{align*}
\dot{h}_i(t) = u(t, h_i(t)), \quad i \in \{1,...,N\}.
\end{align*}
If two points (say, $h_1$ and $h_2$) collide at time $t=T>0$, then they are both solutions of the Cauchy problem
\begin{align}\label{ODE1}
\begin{cases}
\dot{h}_i(t) = u(t, h_i(t)), & \text{for } 0 \leq t < T, \ i \in \{1, 2\},\\
h_1(T)=h_2(T).
\end{cases}
\end{align}
Hence, to show the result, it is sufficient to prove that \eqref{ODE1} is indeed uniquely solvable backwards in time. This in turn will follow from the fact that $u(t,h)$ is Lipschitz in the second variable.\\

As a first step, we let the reader prove the following version of Gr\"onwall's inequality:
\begin{Lemma}[Gr\"onwall's inequality (differential form, backwards in time)]
Let $f:[0,T] \to \R$ be differentiable and assume there is a function $\alpha(t) \in L^1(0,T)$ such that
\begin{align*}
f'(t) \geq -\alpha(t) f(t).
\end{align*}
Then, we have
\begin{align}\label{GronwallDiff}
f(t) \leq f(T) \exp\left( \int_t^T \alpha(s) \dd s \right)
\end{align}
for almost every $t \in [0,T]$.
\end{Lemma}

Next, we see
\begin{align}\label{viii2}
|\dot{h}_1(t) - \dot{h}_2(t)| = |u(t, h_1(t)) - u(t, h_2(t))| \leq \|\del_x u(t, \cdot)\|_{L^\infty(I_1(t))} |h_1(t) - h_2(t)|.
\end{align}
That means that if $\|\del_x u(t, \cdot)\|_{L^\infty(I_1(t))} \in L^1(0, T)$, then an application of \eqref{GronwallDiff} shows that the solution to \eqref{ODE1} is uniquely determined by the data at time $t=T$. As a matter of fact, we have even more.
\begin{Proposition}\label{prop81}
Under the assumptions of Theorem~\ref{thm:1D}, it holds $\del_x u \in L^2(0,T;L^\infty(\R))$.
\end{Proposition}
\begin{proof}
As a key step in the existence proof, the authors in \cite[Proposition~4.1]{VazquezZuazua2006} show that the solution $u$ to \eqref{1DNSE} satisfies in the time interval $(0,T)$ without collision that, for any $t \in (0,T)$,
\begin{align*}
\int_{\R} |\del_x u(t, x)|^2 \dd x + \sum_{i=0}^N \int_0^T \|\del_{xx}^2 u\|_{L^2(I_i(t))}^2 \dd t + \sum_{i=1}^N \int_0^T m_i |\ddot{h}_i(t)|^2 \dd t \leq C,
\end{align*}
where $C>0$ depends on $T$, $\|u_0\|_{W^{1,2}(\R)}$, and $\sum_{i=1}^N m_i |\dot{h}_{i, 0}|^2$. Moreover, recall that the jump of $\del_x u(t, h_i(t))$ is precisely given by $m_i \ddot{h}_i$ by \eqref{1DNSE}$_3$; hence, we may decompose $\del_x u$ in a regular part $\del_x u_{\rm reg}$ and a finite number of Dirac measures located at positions $h_i(t)$ and having amplitudes $m_i$. Since $\del_{xx}^2 u(t, \cdot) \in L^2(I_i(t))$ for any $i$ and any $t$, meaning $\del_{xx}^2 u_{\rm reg}(t, \cdot) \in L^2(\R)$ for any $t \in (0,T)$, we have by Sobolev embedding \eqref{SobEmb} that $\del_x u_{\rm reg}(t, \cdot) \in L^\infty(\R)$. We may now impose the following interpolation inequality: for any $f \in L^2(\R)$ that can be decomposed in a regular part $f_{\rm reg} \in L^\infty(\R)$ and a finite number of Dirac measures with amplitudes $a_i \in \R$, we have
\begin{align}\label{viii1}
\|f\|_{L^\infty(\R)} \lesssim \|f\|_{L^2(\R)}^\frac12 \|\del_x f_{\rm reg}\|_{L^2(\R)}^\frac12 + \sum_{i=1}^N |a_i|.
\end{align}
Choosing $f = \del_x u$ yields finally
\begin{align*}
\|\del_x u\|_{L^2(0,T;L^\infty(\R))}^2 \lesssim 1 + \int_0^T \|\del_x u\|_{L^2(\R)} \|\del_{xx}^2 u_{\rm reg}\|_{L^2(\R)} \dd t \lesssim 1.
\end{align*}
\end{proof}

\begin{Remark}
Inequality \eqref{viii1} follows by observing
\begin{align*}
f^2(x) = -\int_{-\infty}^x 2 f(t) f_{\rm reg}'(t) \dd t + \sum_{x_i \in (-\infty, x)} [f^2](x_i),
\end{align*}
where $x_i$ is the location of the Dirac measure with amplitude $a_i$. Moreover,
\begin{align*}
|[f^2](x_i)| \leq 2 \|f\|_{L^\infty(\R)} |a_i|.
\end{align*}
It follows
\begin{align*}
\|f\|_{L^\infty(\R)}^2 \leq 2 \|f\|_{L^2(\R)} \|f'_{\rm reg}\|_{L^2(\R)} + 2 \|f\|_{L^\infty(\R)} \sum_{i=1}^N |a_i|.
\end{align*}
Solving for $\|f\|_{L^\infty(\R)}$ gives the desired.
\end{Remark}

Finally, with the help of Proposition~\ref{prop81}, we indeed find $\|\del_x u\|_{L^\infty(I_1(t))} \in L^1(0,T)$. Hence, going back to \eqref{viii2}, we get with Gr\"onwall's inequality \eqref{GronwallDiff}
\begin{align*}
|h_1(t)-h_2(t)|\lesssim |h_1(T)-h_2(T)| \ \quad \forall \, 0 \leq t \leq T.
\end{align*}
As $h_1(T)=h_2(T)$, we have $h_1 \equiv h_2$, thus showing that the ODE \eqref{ODE1} is uniquely solvable; in turn, no collision can occur.



\section{Slip-boundary conditions}
Let us go back to three-dimensional flows. The paradox that a ball does not collide its boundary clearly contradicts our physical intuition and experience. To resolve this issue, one can take a closer look to the physical behavior of fluids on surfaces, or, mathematically speaking, the boundary conditions of the fluid velocity. Indeed, the common no-slip condition $\vu|_{\del \Omega} = 0$, although very close to reality, is just an approximation to a more complex slip-condition
\begin{align}
(\vu - \dot{\vc G}(t))\cdot \vc n|_{\del \ms} &= 0, & (\vu - \dot{\vc G}(t)) \times \vc n|_{\del \ms} &= -\beta_\ms [\bS(\vu)\vc n]\times \vc n|_{\del \ms}, \label{slipBC1}\\
\vu \cdot \vc n|_{\del \Omega} &= 0, & \vu \times \vc n|_{\del \Omega} &= -\beta_\Omega [\bS(\vu)\vc n]\times \vc n|_{\del \Omega}, \label{slipBC2}
\end{align}
where the slip coefficients $\beta_\ms, \beta_\Omega \geq 0$ represent the slip lengths on the solid's and container's boundary, respectively. Typically, those slip lengths are very small of order $10^{-9} - 10^{-6}$\,m (see, e.g., \cite{ChoiWestinBreuer2003, TrethewayMeinhart2002}), so it is convenient to simply set them equal to zero and recover the no-slip conditions. However, as we shall show in the sequel, the fact that they are \emph{not equal} to zero allows for collision even for a ball-shaped solid. In a nutshell, our main result for slip-boundary conditions reads as follows.
\begin{Theorem}\label{thm:slip}
Let $\ms=B(\vc G(t))\subset \R^3$ be the unit ball centered at $\vc G(t)$ with $\dist(\ms(0), \del \Omega)>1$. Let $\vu$ be a solution to the incompressible Stokes equations \eqref{Stokes} with stress tensor $\bS = 2\mu \bD(\vu)$ that complies with the slip-boundary conditions \eqref{slipBC1}--\eqref{slipBC2}. Let moreover the fluid's density $\rho_\mf$ be smaller than the solid's one, meaning $\rho_\ms > \rho_\mf > 0$.
\begin{itemize}
\item If both $\beta_\ms>0$ and $\beta_\Omega>0$, collision occurs in finite time.
\item If at least one of the coefficients $\beta_\ms$ and/or $\beta_\Omega$ vanishes, then the solid stays away from $\del \Omega$ for all times.
\end{itemize}
In other words, a ball can just collide a wall if \emph{all} surfaces are slippery.
\end{Theorem}

The problem of collision with slip boundary conditions was investigated in \cite{GerardVaretHillairet2012, GerardVaretHillairet2014, GVHW2015}, and we follow the presentation of \cite{GVHW2015}. The ideas to prove Theorem~\ref{thm:slip} are basically the same as for no-slip conditions, however, the test function now has to comply with the slip boundary conditions. This will complicate both its form as well as the computations to be done.

\subsection{Construction of test function}
As before, we want to construct a proper test function for the drag force. Due to the slip-conditions, the energy functional now reads
\begin{align*}
\mathcal{E}_h = \int_{\mf_h} |\nabla \vu|^2 \dd x + (1+\beta_\ms^{-1}) \int_{\del \ms_h} |(\vu - \vc e_3) \times \vc n|^2 \dd \sigma + \beta_\Omega^{-1} \int_{\del \Omega} |\vu \times \vc n|^2 \dd \sigma,
\end{align*}
which formally can be obtained by multiplying the momentum equation by $\vu$ and integrating by parts. Again, we search for a function of the form $\vc w_h = \nabla \times (\phi_h \vc e_\theta) = -\del_3 \phi_h \vc e_r + \frac1r \del_r (r \phi_h) \vc e_3$ for some function $\phi_h$ to be determined. As we saw in previous chapters, the convective part does not really play a role in whether or not the solid can collide with its container; hence, in the formulation of the boundary conditions, we can replace $\dot{\vc G}$ simply by $\dot{h} \vc e_3$ as in the linear case for Stokes equations. Skipping the index $h$ for brevity, the Euler-Lagrange equation for $\mathcal{E}$ is again $\del_3^4 \phi(r,x_3)=0$, thus
\begin{align*}
\phi(r, x_3) = a(r) + b(r)x_3 + c(r)x_3^2 + d(r)x_3^3.
\end{align*}
To determine the boundary conditions of $\phi$, from which we will find the functions $a,b,c$, and $d$, let us first note that the normal on $\del \Omega \cap \del \Omega_{h, r_0}$ is simply $\vc n = -\vc e_3$, and, since $\ms$ is a ball of radius one, the normal on $\del \ms \cap \del \Omega_{h, r_0}$ is given by
\begin{align*}
\vc n = -r \vc e_r + \sqrt{1-r^2} \vc e_3.
\end{align*}

\begin{figure}[H]
\centering
\begin{tikzpicture}
\draw[->] (0,-3) -- (0,3);
\node at (0,3) [anchor=west] {$x_3$};
\draw[->] (-3,0) -- (3,0);
\node at (3,0) [anchor=south] {$(r, \theta)$};
\draw (0,0) circle (2cm);
\node at (1.9,0) [anchor=south west] {$1$};
\draw[rotate=45, -Stealth, very thick] (-2,0) -- (0,0);
\node at (-.5,-1.1) [anchor=south] {$\vc n$};
\draw[dashed] (-{2*sqrt(1/2)}, -{2*sqrt(1/2)}) -- (-{2*sqrt(1/2)},0);
\node at (-{2*sqrt(1/2)}, 0) [anchor=south] {$r$};
\draw[dashed] (-{2*sqrt(1/2)}, -{2*sqrt(1/2)}) -- (0,-{2*sqrt(1/2)});
\node at (0,-{2*sqrt(1/2)}) [anchor=south west] {$\sqrt{1-r^2}$};
\draw[-Stealth, very thick] (0,0) -- (-2,0);
\node at (-.5,0) [anchor=south] {$\vc e_r$};
\draw[-Stealth, very thick] (0,0) -- (0,2);
\node at (0,1) [anchor=west] {$\vc e_3$};
\node at (2,-2) [anchor=west] {$\otimes \, \vc e_\theta$};
\end{tikzpicture}
\caption{The normal $\vc n$ on $\del \ms \cap \del \Omega_{h, r_0}$.}
\label{fig:82}
\end{figure}

With this at hand, we compute
\begin{align*}
0 &= -\vc w \cdot \vc n|_{\del \Omega} =  \Big(-\del_3 \phi \vc e_r + \frac1r \del_r (r \phi) \vc e_3 \Big)(r,0) \cdot \vc e_3 = \frac1r \del_r(r \phi)(r,0),\\
0 &= (\vc w - \vc e_3) \cdot \vc n|_{\del \ms} = \Big(-\del_3 \phi \vc e_r + \frac1r \del_r (r \phi) \vc e_3 - \vc e_3 \Big)(r,\psi) \cdot (-r \vc e_r + \sqrt{1-r^2} \vc e_3)\\
&= r \del_3 \phi(r,\psi) + \frac{\sqrt{1-r^2}}{r} \del_r (r\phi)(r,\psi) - \sqrt{1-r^2}.
\end{align*}
Noting that for the ball case, where the shape function is given by $\psi = \psi_h(r) = 1+h-\sqrt{1-r^2}$, we see that
\begin{align*}
\psi' = \frac{r}{\sqrt{1-r^2}}, && \sqrt{1+|\psi'|^2} = \frac{1}{\sqrt{1-r^2}}=\frac{1}{r} \psi'.
\end{align*}
Hence, the impermeability condition on $\del \ms$ reduces to
\begin{align*}
0 = r \del_3 \phi(r,\psi) + \frac{\sqrt{1-r^2}}{r} \del_r (r\phi)(r,\psi) - \sqrt{1-r^2} = \frac{1}{\psi'} \frac{\rd}{\rd r} \bigg[r \phi(r, \psi) - \frac{r^2}{2} \bigg],
\end{align*}
that is,
\begin{align*}
\phi(r,\psi) = \frac{r}{2} + \frac{\kappa}{r}, \quad \kappa \in \R.
\end{align*}
In order to get a smooth function up to the origin, we again choose $\kappa=0$. We moreover note that from the impermeability condition on $\del \ms$, we find
\begin{align*}
0 = \del_3 \phi + r \del_{r 3}^2 \phi - \frac{1}{r^2 \sqrt{1-r^2}} \del_r (r \phi) + \frac{\sqrt{1-r^2}}{r} \del_{r r}^2 (r \phi) + \frac{r}{\sqrt{1-r^2}}.
\end{align*}
Hence, as $\phi(r, \psi) = r/2$, we see that
\begin{align*}
- \frac{1}{r^2 \sqrt{1-r^2}} \del_r (r \phi) + \frac{\sqrt{1-r^2}}{r} \del_{r r}^2 (r \phi) + \frac{r}{\sqrt{1-r^2}} = - \frac{r}{r^2 \sqrt{1-r^2}} + \frac{\sqrt{1-r^2}}{r} + \frac{r}{\sqrt{1-r^2}} = 0,
\end{align*}
in turn,
\begin{align}\label{relDerPhi}
\del_{r 3}^2 \phi(r, \psi) = -\frac{1}{r} \del_3 \phi(r, \psi).
\end{align}
We will need this relation later on.

Similarly, the no-penetration condition on $\del \Omega$ leads to
\begin{align*}
\phi(r, 0) = 0,
\end{align*}
giving rise to $a(r)=0$. Considering now incompressible Newtonian fluids with $\div \vu = 0$ and $\bS = \mu(\nabla \vu + \nabla^T \vu) = 2 \mu \bD(\vu)$, we infer from the slip boundary conditions that on $\del \Omega$
\begin{align}
\mu \beta_\Omega \del_{33}^2 \phi(r,0) - \del_3 \phi(r,0) &= 0, \label{slipOm}
\end{align}
see Section~\ref{app:slip} for its derivation. This yields $b(r) = 2 \mu \beta_\Omega c(r)$, hence
\begin{align*}
\phi(r, x_3) = c(r) (2 \mu \beta_\Omega x_3 + x_3^2) + d(r) x_3^3.
\end{align*}
As moreover $\phi(r, \psi) = r/2$, we find
\begin{align*}
c(r)(2 \mu \beta_\Omega \psi + \psi^2) + d(r) \psi^3 = \frac{r}{2} \ \Rightarrow \ d(r) = \frac{1}{\psi^3} \bigg( \frac{r}{2} - c(r) (2 \mu \beta_\Omega \psi + \psi^2 ) \bigg),
\end{align*}
in turn,
\begin{align}\label{formPhi1}
\phi(r, x_3) = c(r) \bigg( 2 \mu \beta_\Omega x_3 + x_3^2 - \frac{2 \mu \beta_\Omega + \psi}{\psi^2} x_3^3 \bigg) + \frac{r}{2} \bigg( \frac{x_3}{\psi} \bigg)^3.
\end{align}
The missing coefficient $c(r)$ will be determined from the last boundary condition on $\del \ms$. One finds (see Section~\ref{app:slip})
\begin{align}\label{slipS}
0 = \del_{33}^2 \phi + \bigg( 2 + \frac{1}{\mu \beta_\ms} \bigg) \frac{\sqrt{1-r^2}}{1-2r^2} \del_3 \phi,
\end{align}
which yields
\begin{align*}
c(r) = \frac{3r}{2\psi} \frac{2 \mu \beta_\ms + \big( 2 \mu \beta_\ms + 1 \big) \frac{\sqrt{1-r^2}}{1-2r^2} \psi}{ (4 \mu \beta_\Omega + \psi) ( 4 \mu \beta_\ms + \big( 2 \mu \beta_\ms + 1 \big) \frac{\sqrt{1-r^2}}{1-2r^2} \psi) - 4\mu^2 \beta_\Omega \beta_\ms }.
\end{align*}

Note that in the limit $\beta_\Omega, \beta_\ms \to 0$ representing no-slip conditions, we obtain
\begin{align*}
c(r) \to \frac{3r}{2\psi^2}
\end{align*}
and hence
\begin{align*}
\phi(r, x_3) \to \frac{3r}{2\psi^2} \bigg( x_3^2 - \frac{x_3^3}{\psi} \bigg) + \frac{r}{2} \frac{x_3^3}{\psi^3} = \frac{r}{2} ( 3 t^2 - 2t^3 )|_{t=\frac{x_3}{\psi}}
\end{align*}
as expected. Moreover, we may represent $\phi(r, x_3 \psi) = \frac{r}{2}[P_1(r)x_3 + P_2(r) x_3^2 + P_3(r) x_3^3]$ with
\begin{gather*}
P_1(r) = \frac{6(2+\alpha_\ms)}{12 + 4(\alpha_\ms + \alpha_\Omega) + \alpha_\ms \alpha_\Omega}, \qquad
P_2(r) = \frac{3(2+\alpha_\ms)\alpha_\Omega}{12 + 4(\alpha_\ms + \alpha_\Omega) + \alpha_\ms \alpha_\Omega},\\
P_3(r) = -\frac{2(\alpha_\ms + \alpha_\ms \alpha_\Omega + \alpha_\Omega)}{12 + 4(\alpha_\ms + \alpha_\Omega) + \alpha_\ms \alpha_\Omega},\\
\alpha_\Omega = \frac{1}{\mu \beta_\Omega} \psi, \qquad \alpha_\ms = \bigg( 2 + \frac{1}{\mu \beta_\ms} \bigg) \frac{\sqrt{1-r^2}}{1-2r^2} \psi.
\end{gather*}

\begin{Remark}
Let us notice that the form of $\alpha_\ms$ in \cite{GVHW2015} does not match the one above, namely, the coefficient $\sqrt{1-r^2}/(1-2r^2)$ is missing there. This is due to the fact that for small $r$, this fraction is bounded below and above (and even close to $1$), hence the \emph{qualitative} behavior of the function does not change and one might neglect it.
\end{Remark}

\subsection{Uniform estimates, corresponding pressure, and proof of Theorem~\ref{thm:slip}}
As for the case of no-slip, we need some bounds of the constructed test function, as well as an additional pressure catching the singularities of $\Delta \vc w_h$. This is done in the following Lemmata:

\begin{Lemma}\label{lem3}
Let $\beta_\Omega, \beta_\ms > 0$. For the function $\vc w_h$ constructed in the previous section, there holds
\begin{align*}
\|\vc w_h\|_{L^2(\mf_h)} + \|\beta_\ms [\bS(\vc w_h)\vc n] \times \vc n + (\vc w_h - \vc e_3) \times \vc n\|_{L^2(\del \ms_h)} &\lesssim 1,\\
\|\nabla \vc w_h\|_{L^2(\mf_h)}^2 \lesssim |\log h| &\lesssim \|\bD(\vc w_h)\|_{L^2(\mf_h)}^2.
\end{align*}

Moreover, denoting $x = (r \cos \theta, r \sin \theta, x_3)$,
\begin{align*}
\left\| \int_0^\psi \del_h \vc w_h(r,s) \dd s \right\|_{L^2(\del \Omega \cap \del \Omega_{h, r_0})}^2 &\lesssim |\log h|,\\
\left\| \int_{x_3}^\psi \del_h \vc w_h(r, s) \dd s \right\|_{L^2(\Omega_{h, r_0})} + \left\| \psi \nabla \int_{x_3}^\psi \del_h \vc w_h(r,s) \dd s \right\|_{L^2(\Omega_{h, r_0})} &\lesssim 1,\\
\sup_{x \in \overline{\Omega}_{h, r_0}} \left| \int_{x_3}^\psi \bD(\vc w_h)(r, s) \dd s \right| + \sup_{x \in \overline{\Omega}_{h, r_0}} \left|\psi \nabla \int_{x_3}^\psi \bD(\vc w_h)(r, s) \dd s \right| &\lesssim 1.
\end{align*}
\end{Lemma}

\begin{Lemma}
There exists a pressure $q_h$ satisfying
\begin{align*}
|\log h| \lesssim \int_{\del \ms_h} (\bS(\vc w_h) - q_h \Id)\vc n \cdot (\vc e_3 - \vc w_h) \dd \sigma - \int_{\del \Omega} (\bS(\vc w_h) - q_h \Id)\vc n \cdot \vc w_h \dd \sigma \lesssim |\log h|,
\end{align*}
and for any $\vc v \in W^{1,2}(\mf_h)$ with $\vc v \cdot \vc n = 0$ on $\del \Omega$, it holds
\begin{align*}
\left| \int_{\mf_h} (\Delta \vc w_h - \nabla q_h) \cdot \vc v \dd x \right| \lesssim \|\bD(\vc v)\|_{L^2(\mf_h)} + \|\vc v\|_{L^2(\del \Omega)}.
\end{align*}
\end{Lemma}

As a matter of fact, the pressure will be determined as before, meaning
\begin{align*}
q_h(r, x_3) = \del_{r 3}^2 \phi_h(r,x_3) - \int_0^r \del_{333}^3 \phi_h(t, x_3) \dd t.
\end{align*}
The same calculations as made for the no-slip case (but using the above constructed ``slip-test-function'') yield the desired bounds; we therefore do not repeat them and refer to \cite[Appendix~A]{GVHW2015} for details.\\

Also, the proof of Theorem~\ref{thm:slip} follows essentially the same lines as for the no-slip case, but this time using the weaker (logarithmic) bounds obtained in Lemma~\ref{lem3}. Repeating the steps done in Chapter~\ref{sec:62}, the final outcome is an inequality of the form
\begin{align}\label{ineqSlip1}
\mu \int_{h_0}^{h(T)} \mathcal{D}(s) \dd s + \frac{4\pi}{3} (\rho_\ms - \rho_\mf) g T \leq C_0 (1 + \sqrt{T})
\end{align}
for some generic constant $C_0>0$ that is independent of time, where the \emph{drag} $\mathcal{D}$ is the energy of the test function $\vc w_h$, given by
\begin{align*}
\mathcal{D}(h) = \mathcal{E}_h(\vc w_h) = \int_{\mf_h} |\nabla \vc w_h|^2 \dd x &+ (1+\beta_\ms)^{-1} \int_{\del \ms_h} |(\vc w_h - \vc e_3)\times \vc n|^2 \dd \sigma \\
&+ \beta_\Omega^{-1} \int_{\del \Omega} |\vc w_h \times \vc n|^2 \dd \sigma.
\end{align*}

This formula is the counterpart to \eqref{drag} for slip boundary conditions. Moreover, the bounds obtained for $\vc w_h$ yield a drag force of order
\begin{align*}
|\mathcal{D}(h)| \lesssim |\log h|,
\end{align*}
hence
\begin{align}\label{ineqSlip2}
\int_{h_0}^{h(T)} \mathcal{D}(s) \dd s \gtrsim - \int_0^{\sup\limits_{t \in (0,T)} h(t)} |\mathcal{D}(s)| \dd s \gtrsim -1.
\end{align}
As the fluid's density is smaller than the solid's one, \eqref{ineqSlip1}--\eqref{ineqSlip2} yield collision in finite time.\\

For the mixed case, we focus just on $\beta_\ms = 0$ and $\beta_\Omega>0$. The reverse case follows similar lines. Rigorously, one has to re-define the test function $\vc w_h$, and estimate its corresponding norms accordingly; however, it turns out that this is equivalent in taking the limit $\alpha_\ms \to \infty$ (respectively $\beta_\ms \to 0$) in the definition of $\vc w_h$, and hence the corresponding norms change as one shall expect. This limit procedure yields as in the no-slip case
\begin{align*}
\|\vc w_h\|_{L^2(\mf_h)}^2 + h \|\nabla \vc w_h\|_{L^2(\mf_h)}^2 \lesssim 1 &\lesssim h \|\bD(\vc w_h)\|_{L^2(\mf_h)}^2,
\end{align*}
and for the drag force, one gets
\begin{align*}
\int_{h_0}^{h(T)} \mathcal{D}(s) \dd s \lesssim 1+T,
\end{align*}
yielding
\begin{align*}
|\log h(T)| \lesssim 1+T
\end{align*}
and hence collision cannot occur. More details regarding the calculations can be found in \cite[Section~4]{GVHW2015}.

\subsection{Appendix: tangential slips}\label{app:slip}
Let us show how the equations \eqref{slipOm} and \eqref{slipS} arise. First, we recall the gradient in cylindrical coordinates:
\begin{Lemma}
We have
\begin{align*}
\nabla = \vc e_r \del_r + \frac{1}{r} \vc e_\theta \del_\theta + \vc e_3 \del_3
\end{align*}
as well as
\begin{align*}
\nabla \vc w = \vc e_r \otimes \nabla (\vc w \cdot \vc e_r) + \vc e_\theta \otimes \nabla (\vc w\cdot \vc e_\theta) + \vc e_3 \otimes \nabla (\vc w \cdot \vc e_3) + \frac{1}{r} \big( (\vc w \cdot \vc e_r) \vc e_\theta - (\vc w \cdot \vc e_\theta) \vc e_r \big) \otimes \vc e_\theta.
\end{align*}
\end{Lemma}

Recalling $\phi = \phi(r, x_3)$ is independent of the angle $\theta$, as well as $\vc w = -\del_3 \phi \vc e_r + \frac{1}{r} \del_r(r\phi) \vc e_3$, we find
\begin{align*}
\nabla \vc w &= -\del_{r 3}^2 \phi \vc e_r \otimes \vc e_r - \del_{33}^2 \phi \vc e_r \otimes \vc e_3 - \frac{1}{r} \del_3 \phi \vc e_\theta \otimes \vc e_\theta + \del_r \bigg( \frac{1}{r} \del_r(r \phi) \bigg) \vc e_3 \otimes \vc e_r + \frac{1}{r} \del_r (r \del_3 \phi) \vc e_3 \otimes \vc e_3. 
\end{align*}
Hence, using also that $(\vc a \otimes \vc b) \cdot \vc c = (\vc b \cdot \vc c) \vc a$, we get
\begin{align*}
2 \bD(\vc w) \vc n &= -2\del_{r 3}^2 \phi (\vc e_r \cdot \vc n) \vc e_r + \bigg[ \del_r \bigg( \frac{1}{r} \del_r(r \phi) \bigg) - \del_{33}^2 \phi \bigg] \big( (\vc e_r \cdot \vc n)\vc e_3 + (\vc e_3 \cdot \vc n) \vc e_r \big)\\
&\quad + \frac{2}{r} \del_r (r \del_3 \phi) (\vc e_3 \cdot \vc n) \vc e_3 - \frac{2}{r} \del_3 \phi (\vc e_\theta \cdot \vc n) \vc e_\theta.
\end{align*}
In turn, we have on $\del \Omega$ where $\vc n = - \vc e_3$ and by the relation $\vc e_r \times \vc e_3 = -\vc e_\theta$
\begin{align*}
2[\bD(\vc w) \vc n] \times \vc n |_{\del \Omega} = \bigg[ \del_{33}^2 \phi - \del_r \bigg( \frac{1}{r} \del_r (r \phi) \bigg) \bigg] \vc e_\theta.
\end{align*}
Taking into account that $\bS(\vc w) = 2 \mu \bD(\vc w)$ and that on $\del \Omega$, we have $\phi(r, 0) = 0$, this yields
\begin{align*}
0 = (\beta_\Omega [\bS(\vc w)\vc n] + \vc w) \times \vc n |_{\del \Omega} = [\mu \beta_\Omega \del_{33}^2 \phi - \del_3 \phi](r, 0) \vc e_\theta,
\end{align*}
which is precisely \eqref{slipOm}. Similarly, on $\del \ms$ where $\vc n = -r \vc e_r + \frac{r}{\psi'} \vc e_3$, we find
\begin{align*}
(\vc e_r \cdot \vc n) \vc e_r \times \vc n &= \frac{r^2}{\psi'}\vc e_\theta, &
(\vc e_r \cdot \vc n) \vc e_3 \times \vc n &= r^2 \vc e_\theta,\\
(\vc e_3 \cdot \vc n) \vc e_r \times \vc n &= - \bigg(\frac{r}{\psi'} \bigg)^2 \vc e_\theta, &
(\vc e_3 \cdot \vc n) \vc e_3 \times \vc n &= -\frac{r^2}{\psi'}\vc e_\theta.
\end{align*}
Hence,
\begin{align*}
2 [\bD(\vc w) \vc n] \times \vc n |_{\del \ms} = \bigg[ -2\del_{r 3}^2 \phi \frac{r^2}{\psi'} + \bigg[ \del_r \bigg( \frac{1}{r} \del_r (r \phi) \bigg) - \del_{33}^2 \phi \bigg] \bigg( r^2 - \frac{r^2}{|\psi'|^2} \bigg) - 2 \del_r (r \del_3 \phi) \frac{r}{\psi'} \bigg] \vc e_\theta.
\end{align*}
Moreover,
\begin{align*}
(\vc w - \vc e_3) \times \vc n = \bigg[ \del_3 \phi \frac{r}{\psi'} - (\del_r(r \phi)-r) \bigg] \vc e_\theta.
\end{align*}
Using the specific form $\phi(r, \psi) = r/2$ on $\del \ms$, this yields
\begin{align*}
(\vc w - \vc e_3) \times \vc n &= \frac{r}{\psi'} \del_3 \phi \vc e_\theta,\\
2[\bD(\vc w) \vc n] \times \vc n &= -\bigg[ 2\del_{r 3}^2 \phi \frac{r^2}{\psi'} + \del_{33}^2 \phi \bigg( r^2 - \frac{r^2}{|\psi'|^2} \bigg) + 2 \del_r (r \del_3 \phi) \frac{r}{\psi'} \bigg] \vc e_\theta,
\end{align*}
in turn,
\begin{align*}
0 &= [(\vc w - \vc e_3) \times \vc n + 2 \mu \beta_\ms [\bD(\vc w) \vc n] \times \vc n] \cdot \vc e_\theta \\
&= \frac{r}{\psi'} \del_3 \phi - \mu \beta_\ms \bigg[ 2\del_{r 3}^2 \phi \frac{r^2}{\psi'} + \del_{33}^2 \phi \bigg( r^2 - \frac{r^2}{|\psi'|^2} \bigg) + 2 \del_r (r \del_3 \phi) \frac{r}{\psi'} \bigg]\\
&= - \del_{33}^2 \phi \mu \beta_\ms \bigg( r^2 - \frac{r^2}{|\psi'|^2} \bigg) + \frac{r}{\psi'} \bigg( \del_3 \phi (1 - 2 \mu \beta_\ms) - 4 r \mu \beta_\ms \del_{r3}^2 \phi \bigg).
\end{align*}
Using \eqref{relDerPhi}, we get for the last term
\begin{align*}
\del_3 \phi (1-2 \mu \beta_\ms) - 4r \mu \beta_\ms \del_{r 3}^2 \phi = \del_3 \phi (1-2 \mu \beta_\ms) + 4 \mu \beta_\ms \del_3 \phi = \del_3 \phi (1 + 2 \mu \beta_\ms).
\end{align*}
By \eqref{formPhi1}, we find
\begin{align*}
\del_3 \phi(r, \psi) &= -c(4 \mu \beta_\Omega + \psi) + \frac{3r}{2 \psi},\\
\del_{33}^2 \phi(r, \psi) &= -4c \frac{3 \mu \beta_\Omega + \psi}{\psi} + \frac{3r}{\psi^2}.
\end{align*}
Hence, recalling $\psi = 1+h-\sqrt{1-r^2}$, we are left with
\begin{align*}
0 &= - \mu \beta_\ms \del_{33}^2 \phi \bigg( r^2 - \frac{r^2}{|\psi'|^2} \bigg) + (1+2\mu \beta_\ms) \frac{r}{\psi'} \del_3 \phi \\
&= \mu \beta_\ms(1-2r^2) \left[ \del_{33}^2 \phi + \bigg( 2 + \frac{1}{\mu \beta_\ms} \bigg) \frac{\sqrt{1-r^2}}{1-2r^2} \del_3 \phi \right]
\end{align*}
and consequently
\begin{align*}
0 &= -4c \frac{3 \mu \beta_\Omega + \psi}{\psi} + \frac{3r}{\psi^2} + \bigg( 2 + \frac{1}{\mu \beta_\ms} \bigg) \frac{\sqrt{1-r^2}}{1-2r^2} \bigg( -c (4 \mu \beta_\Omega + \psi) + \frac{3r}{2\psi} \bigg)\\
&= -c \bigg( 4 \frac{3 \mu \beta_\Omega + \psi}{\psi} + \bigg( 2 + \frac{1}{\mu \beta_\ms} \bigg) \frac{\sqrt{1-r^2}}{1-2r^2} (4\mu \beta_\Omega + \psi) \bigg) + \frac{3r}{\psi^2} + \bigg( 2 + \frac{1}{\mu \beta_\ms} \bigg) \frac{\sqrt{1-r^2}}{1-2r^2} \frac{3r}{2\psi},
\end{align*}
so finally
\begin{align*}
c &= \frac{\frac{3r}{\psi^2} + \big( 2 + \frac{1}{\mu \beta_\ms} \big) \frac{\sqrt{1-r^2}}{1-2r^2} \frac{3r}{2\psi}}{ 4 \frac{3 \mu \beta_\Omega + \psi}{\psi} + \big( 2 + \frac{1}{\mu \beta_\ms} \big) \frac{\sqrt{1-r^2}}{1-2r^2} (4\mu \beta_\Omega + \psi) } \\
&= \frac{3r}{2\psi} \frac{2 \mu \beta_\ms + \big( 2 \mu \beta_\ms + 1 \big) \frac{\sqrt{1-r^2}}{1-2r^2} \psi}{ (4 \mu \beta_\Omega + \psi) ( 4 \mu \beta_\ms + \big( 2 \mu \beta_\ms + 1 \big) \frac{\sqrt{1-r^2}}{1-2r^2} \psi) - 4\mu^2 \beta_\Omega \beta_\ms }.
\end{align*}

\section{Tresca's boundary condition}\label{sec:Tresca}
Letting the fluid slip on the boundary makes collision happen, however, this model is not realistic in the sense that the fluid slips regardless of the boundary's shear. Another model to handle this issue are the so-called Tresca\footnote{after Henri Tresca (1814--1885)} boundary conditions as a version of slip boundary conditions, but dependent on the rate of shear exerted by the fluid. More precisely, as stated in \cite{HillairetTakahashi2021}, ``[in] these boundary conditions, the fluid sticks to the interface up to a shear-rate threshold that the fluid is prevented to exceed by allowing slip on the interface. The boundaries of the fluid domain split then in a zone of small shear rates where Dirichlet boundary conditions are imposed and high shear rates where a type of Navier boundary conditions are imposed (but with an unknown slip length which encodes that the shear rate cannot
exceed the threshold value).'' We shall show that under such modified slip conditions, collision can still occur, provided the body is close enough to its container's wall and its mass is large enough.\\

To precise the presentation, we focus on the following setting: let $\Omega = (-L, L) \times (0, L') \subset \R^2$ with $L>1, \ L'>2$, and let the solid $\ms = \ms_h = B_1((h+1) \vc e_2)$ be the unit disk centred as $(h+1)\vc e_2$. Denote as usual $\mf = \Omega \setminus \ms$. The equations governing the fluid's and solid's motion are given by
\begin{align*}
\begin{cases}
\rho_\mf (\del_t \vu + \vu \cdot \nabla \vu) - \div \bS + \nabla p = 0 & \text{in } \mf,\\
\div \vu = 0 & \text{in } \mf,\\
m \ddot{h} = - \int_{\del \ms} (\bS(\nabla \vu) - p \Id) \vc n \cdot \vc e_2 \dd \sigma - (m - \pi \rho_\mf) g,\\
h(0)=h_0, \ \dot{h}(0) = \dot{h}_0, \ \vu(0, \cdot) = \vu_0 & \text{in } \mf(0),
\end{cases}
\end{align*}
where $\bS = 2 \mu \bD(\vu) = \mu (\nabla \vu + \nabla^T \vu)$, and $m = \pi \rho_\ms$ is the solid's mass. The no-penetration boundary conditions are as before
\begin{align*}
(\vu - \dot{h} \vc e_2)\cdot \vc n|_{\del \ms} &= 0, & \vu \cdot \vc n|_{\del \Omega} &= 0.
\end{align*}
For the tangential parts, this time we have on $\del \ms$
\begin{align*}
\begin{cases}
(\vu - \dot{h} \vc e_2) \times \vc n = 0 & \text{if } |[\bS\vc n] \times \vc n|<\varsigma^\ast,\\
\exists \beta \geq 0: \ (\vu - \dot{h}\vc e_2) \times \vc n = - \beta [\bS \vc n] \times \vc n & \text{if } |[\bS\vc n] \times \vc n| \geq \varsigma^\ast,
\end{cases}
\end{align*}
and similar on $\del \Omega$
\begin{align*}
\begin{cases}
\vu \times \vc n = 0 & \text{if } |[\bS\vc n] \times \vc n|<\varsigma^\ast,\\
\exists \beta \geq 0: \ \vu \times \vc n = - \beta [\bS \vc n] \times \vc n & \text{if } |[\bS\vc n] \times \vc n| \geq \varsigma^\ast.
\end{cases}
\end{align*}
Here, the \emph{shear threshold} $\varsigma^\ast>0$, and without loss of generality, we may set $\varsigma^\ast=1$ in the sequel.

Introducing the function
\begin{align*}
\vc w^\ast = \begin{cases}
\vc e_2 & \text{on } \del \ms,\\
0 & \text{on } \del \Omega,
\end{cases}
\end{align*}
these boundary conditions can be shortly written as
\begin{align}\label{TrescaBC}
\begin{cases}
(\vu - \dot{h}\vc w^\ast) \cdot \vc n|_{\del \mf} = 0, & \\
(\vu - \dot{h}\vc w^\ast) \times \vc n|_{\del \mf} = 0 & \text{if } |[\bS \vc n] \times \vc n| < 1,\\
\exists
\beta \geq 0: \ (\vu - \dot{h} \vc w^\ast) \times \vc n|_{\del \mf} = - \beta [\bS \vc n] \times \vc n|_{\del \mf} & \text{if } |[\bS \vc n] \times \vc n| \geq 1.
\end{cases}
\end{align}

The proof that contact occurs relies on several estimates including an adapted Korn inequality and bounds obtained from the energy inequality. These bounds then enable the authors in \cite{HillairetTakahashi2021} to prove the following two statements:
\begin{Lemma}[{\cite[Lemma~2.5]{HillairetTakahashi2021}}]\label{lem:Tresca1}
Let $\dot{h}_0 = 0$, $\vu_0 \in L^2(\mf(0))$ with $\div \vu_0 = 0$, and $\rho_\mf>0$ be given. As long as $h \leq 1$, there holds
\begin{align*}
\dot{h}(t) \leq - \frac{gt}{2} + C^\sharp g h_0 + \frac{C^\flat}{m} \int_0^t |\log h(s)| \dd s,
\end{align*}
where the constants $C^\sharp, C^\flat > 0$ are independent of $m$ and $h_0$.
\end{Lemma}

\begin{Lemma}[{\cite[Lemma~2.6]{HillairetTakahashi2021}}]\label{lem:Tresca2}
Let $\sigma \in (0, 1/2)$ and define a sequence of times $(t_n)_{n \geq 0}$ via
\begin{align*}
\begin{cases}
t_0 = \frac14 \sqrt{\frac{h_0}{g}},\\
t_{n+1} = t_n + \sigma \frac{h(t_n)}{2 \sqrt{g h_0}}.
\end{cases}
\end{align*}
This sequence is well defined for $h_0$ sufficiently small and $m h_0$ sufficiently large. Moreover, there holds
\begin{align}\label{TrescaH}
\frac{h_0}{2} (1-\sigma)^n \leq h(t_n) \leq \frac32 h_0 \Big( 1 - \frac{\sigma^2}{32} \Big)^n.
\end{align}
\end{Lemma}

Similar to the chapters before, the proof of Lemma~\ref{lem:Tresca1} relies on the construction of a suitable test function $\vc w_h$ for the momentum equation satisfying
\begin{align*}
\div \vc w_h = 0 \ \text{ in } \ \mf, \qquad (\vc w_h - \vc w^\ast) \cdot \vc n|_{\del \mf} = 0, \qquad [\bS(\nabla \vc w_h)\vc n] \times \vc n|_{\del \Omega} = 0,
\end{align*}
and a corresponding pressure $q_h$ (see \cite[Section~3]{HillairetTakahashi2021} for details). The proof of Lemma~\ref{lem:Tresca2} is given by induction and an application of Lemma~\ref{lem:Tresca1}. The terms ``sufficiently small'' and ``sufficiently large'' are interpreted as
\begin{align*}
h_0 < \max \left \lbrace \frac{2}{3(1+\sigma)}, \frac{1}{(32 C^\sharp)^2 g} \right \rbrace
\end{align*}
and
\begin{align*}
-\frac{1}{32} \sqrt{g h_0} + \frac{3 C^\flat \sigma}{4m} \sqrt{\frac{h_0}{g}} \sum_{k=0}^\infty \Big( 1 - \frac{\sigma^2}{32} \Big)^k \Big| \log \Big[ (1-\sigma)^{k+1} \frac{h_0}{2} \Big] \Big| \leq - \frac{\sigma}{16} \sqrt{g h_0},
\end{align*}
which ensure that $h(t) \in (0,1)$ for all $t \in [0, t_n]$ and all $n \geq 0$, and also that $h(t)$ satisfies \eqref{TrescaH}. The final step is now to see that the definition of $t_n$, together with \eqref{TrescaH}, ensures that the sequence $(t_{n+1} - t_n)_{n \geq 0}$ is bounded from above by a convergent geometric series. In turn, we have $\lim_{n \to \infty} t_n = T_\ast < \infty$, and additionally $\lim_{n \to \infty} h(t_n) = 0$ by \eqref{TrescaH}. Since $h$ is continuous, this shows $h(T_\ast)=0$, meaning collision happens, and the maximal existence time of the solution is in this case finite.

\begin{Remark}
Note carefully that the assumptions of a large mass and zero initial speed fit, to some extend, the assumptions of Theorem~\ref{theo1} for the compressible setting with no-slip boundary conditions. This might be interpreted in such a way that for compressible fluids, the non-constant density plays the role of a kind of ``pseudo-slip'', hence allowing for collision.
\end{Remark}



\chapter*{Some bibliographical remarks}
Among the literature cited in this work, there is a bunch of works dealing with different situations, although most of them are focussed on incompressible fluids. We will give here some literature, and refer the interested reader to the works cited in there. First, let us mention the article \cite{Hillairet2007navier}, where the author summarizes the up-to-date known results of collisions for incompressible Newtonian fluids with no-slip boundary conditions. In \cite{ChemetovNecasova2017}, the authors show existence for incompressible fluids even in the presence of contacts, where no-slip boundary conditions are imposed on the container's wall, and Navier slip boundary conditions are imposed on the solid, extending the work of \cite{DesjardinsEsteban1999, GerardVaretHillairet2014}, where existence was shown \emph{up to collision}. Existence of strong solutions for incompressible two-dimensional fluid-solid interaction with no-slip boundary conditions and a solid of class $C^2$ was proven in \cite{Takahashi2003}, where either the maximal existence time $T_\ast=\infty$, or collision occurs in finite time. For a deformable structure in two dimensions, global existence of strong solution was shown in \cite{GrandmontHillairet2016}. Short time existence for elastic structures has been shown in \cite{Boulakia2007}, and in \cite{CoutandShkoller2005} for quite high regularity assumptions on the container's and solid's boundary.\\

Besides the examples of roughness given in previous chapters, the authors in \cite[Section~4]{GerardVaretHillairet2012} investigated the case of a corrugated container's wall, that is, for some given smooth $1$-periodic function $\Gamma: \R^2 \to (-\infty, 0]$ with $\max_{(x_1, x_2) \in \R^2} \Gamma(x_1, x_2) = \Gamma(0,0)=0$ and some given $\e>0$, we have
\begin{align*}
\del \Omega \cap (\R^2 \times (-\infty, 0]) = \left\{ x \in \R^3: x_3 = \e \Gamma \left( \frac{x_1}{\e}, \frac{x_2}{\e} \right) \right\}.
\end{align*}
Focussing just on the case $\e \ll h \ll 1$, the authors ``quote that this is the only regime for which this modelling of the roughness is relevant. Indeed, when $h$ becomes comparable to $\e$, a rescaling in space by a factor $1/ \e$ brings back to the classical situation of smooth boundaries. From this point of view, the model we consider in this section is peculiar: it does not allow to conclude anything about the possibility of collisions for a given small roughness size $\e$.''

Concerning the other parameters of the configuration, they assumed $\ms$ to be a ball of unit radius, and no-slip boundary conditions on every part of the boundary. The outcome of their computations is a drag force of order
\begin{align*}
\frac{6 \pi}{h+\lambda \e} + \mathcal{O}(|\log(h + \lambda \e)|) \leq \mathcal{D}_h \leq \frac{6 \pi}{h} + \mathcal{O}(|\log h|), \qquad \lambda = - \min_{(x_1, x_2) \in \R^2} \Gamma > 0,
\end{align*}
which matches the experimental results found in \cite{KunertHartingVinogradova2010, VinogradovaYakubov2006}. In this context, let us also cite \cite{LACSF2004}, where the authors investigate the drag force of a sphere falling over a corrugated wall. Moreover, the references in there show that for a flat bottom, we have $\mathcal{D}_h \sim h^{-1}$, hence giving the experimental verification of our foregoing calculations.



\chapter*{Acknowledgements}
{\it The author thanks Valentin Calisti and \v{S}\'arka Ne\v{c}asov\'a for several helpful discussions. F. O. has been supported by the Czech Science Foundation (GA\v CR) project 22-01591S, and the Czech Academy of Sciences project L100192351. The Institute of Mathematics, CAS is supported by RVO:67985840.}

\addcontentsline{toc}{chapter}{Bibliography}
\bibliography{Lit-Master.bib}

\bibliographystyle{plain}

\end{document}